\documentclass[reqno,11pt,a4paper]{amsart}

\usepackage{amsmath}
\usepackage{amsthm,amssymb,mathrsfs}
\usepackage{graphicx}
\usepackage{hyperref}

\usepackage{fullpage}
\usepackage{amsaddr}
\usepackage{etoolbox}

\let\origmaketitle\maketitle
\def\maketitle{
  \begingroup
  \def\uppercasenonmath##1{} 
  \let\MakeUppercase\relax 
  \origmaketitle
  \endgroup
}

\newcommand{\Rm}{\mathbb{R}}

\newcommand{\Nm}{\mathbb{N}}
\newcommand{\Zm}{\mathbb{Z}}
\newcommand{\Sm}{\mathbb{S}}
\newcommand{\be}{\[}
\newcommand{\ee}{\]}
\newcommand{\ba}{\[\begin{aligned}}
\newcommand{\ea}{\end{aligned}\]}
\newcommand{\va}{\varphi}

\newcommand{\pp}{\partial}
\newcommand{\bv}[1]{\boldsymbol{\mathrm{#1}}}

\newcommand{\uv}{\theta}

\newcommand{\argmin}{\mathop{\mathrm{arg\,min}}}

\newtheorem{thm}{Theorem}[section]
\newtheorem{lem}[thm]{Lemma}

\newtheorem{prop}[thm]{Proposition}

\theoremstyle{remark}\newtheorem{rmk}[thm]{Remark}

\title[]{{\Large Lecture note on}\\ \vspace{5mm}
{\huge Inverse problems and reconstruction methods}}

\author{Manabu Machida\footnote{{\it Web page}: www.mmachida.com}}
\address[Machida]{Department of Informatics,\\
Faculty of Engineering, Kindai University,\\
Higashi-Hiroshima 739-2116, Japan}
\email{machida@hama-med.ac.jp}

\date{September 10, 2024}

\begin{document}

\maketitle

\begin{abstract}
The area of inverse problems in mathematics is highly interdisciplinary. In various fields of science, engineering, medicine, and industry, there arises a need to reconstruct information about unknown entities that cannot be directly observed. Examples include medical imaging techniques such as X-ray CT and optical tomography. Indeed, the mathematics of inverse problems has often originated from challenges posed by other fields. Inverse problems are often ill-posed and solutions are unstable. In this lecture, we will explore methods to solve such inverse problems.
\end{abstract}

\setcounter{tocdepth}{1}
\tableofcontents

\clearpage
\section{What are inverse problems?}
\label{sec1}

\subsection{Introduction}
\hfill\vskip1mm

Let us look at the following first-order equation.
\begin{equation}
\left\{\begin{aligned}
\pp_tu(x,t)+\pp_xu(x,t)+p(x)u(x,t)=0,&\quad 0<x<\ell,\quad 0<t<T,
\\
u(x,0)=0,&\quad 0<x<\ell,
\\
u(0,t)=b(t),&\quad 0<t<T,
\end{aligned}\right.
\label{sec1:eq1}
\end{equation}
where $\ell,T$ are positive constants. The problem of determining unknown $u$ by using given $p,b$ is called the forward problem. The problem of determining unknown $b$ or $p$, or both from $u$ is called the inverse problem.

In the above equation (\ref{sec1:eq1}), the problem of finding $\ell$ is also an inverse problem. In general, determining of the shape of a domain or subdomain is an inverse problem. An example of such inverse problems is concrete crack detection. The problem of determining the shape of a drum is a well-known inverse problem \cite{Kac66}. This inverse problem can be formulated as an problem of determining the domain $\Omega$ using eigenvalues $\lambda_n$:
\be
\left\{\begin{aligned}
\Delta u+\lambda u=0,&\quad\mbox{in}\;\Omega,
\\
u=0,&\quad\mbox{on}\;\pp\Omega.
\end{aligned}\right.
\ee

Medical imaging modalities such as X-ray CT requires solving inverse problems to obtain reconstructed images. Indeed, in many cases, mathematics of inverse problems has developed from issues in industry and different fields of science and engineering. Such inverse problems include estimation of underground structures from the surface gravity and seismic surveys in oil field exploration, determination of the internal temperature of a blast furnace from external observations in ironmaking, investigation of neutron stars through gravitational wave observations in astronomy, and estimation of gene networks in biology by the study of which gene expressions are activated or suppressed. Finding the equation of motion by an apple which drops from a tree is another example of inverse problems.

In short, problems which seek effects from causes are forward problems and problems which determine causes from effects are inverse problems. It is, however, not so easy to give a definition of inverse problems. Although it is difficult to define inverse problems, in most cases, we can rather easily tell if a problem is a forward problem or an inverse problem.

When the solution of a problem exists, is unique, and is stable, the problem is said to be well-posed in the sense of Hadamard. Since an inverse problem stems from its corresponding forward problem, quite often the existence of a solution is a priori assumed for inverse problems. When the solution of a problem is unique and stable, the problem is said to be well-posed in the sense of Tikhonov. A problem is ill-posed if it is not well-posed. Inverse problems are ill-posed in many cases \cite{Kaipio-Somersalo04}. See also Sec.~4 of \cite{Colton-Kress98} for ill-posed problems.

Different phenomena which appear in science, engineering, and industry are quite often governed by partial differential equations. Inverse problems for partial differential equations emerge when we seek reasons behind a phenomenon. In this note, we particularly focus on inverse problems for partial differential equations. It is also important to numerically reconstruct the solution to an inverse problem. Thus, the research on inverse problems consists of numerical reconstruction and the analysis of uniqueness and stability.

\subsection{Example 1: Heat equation}
\hfill\vskip1mm

Let us consider the heat equation:
\be
\left\{\begin{aligned}
\pp_tu(x,t)-\pp_x^2u(x,t)=\lambda(t)f(x),&\quad 0<x<1,\quad 0<t<T,
\\
u(x,0)=0,&\quad 0<x<1,
\\
u(0,t)=u(1,t)=0,&\quad 0<t<T.
\end{aligned}\right.
\ee
We assume that $f(x)$ is known. Suppose that $u(x_0,t)$ is observed for fixed $x_0\in(0,1)$. We want to solve the inverse problem of determining $\lambda(t)$ using $u(x_0,t)$ ($0<t<T$):
\be
\mbox{Inverse problem:}\quad u(x_0,t)\rightarrow \lambda(t)\quad (0<t<T)
\ee

We assume that $f(x_0)\neq0$ and $df/dx\neq0$ ($0<x<1$). By introducing a linear operator $A$, we rewrite the heat equation as
\be
\left\{\begin{aligned}
\frac{du}{dt}=-Au+\lambda(t)f(x),&\quad0<t<T,
\\
u(0)=0.&
\end{aligned}\right.
\ee
There exists a semigroup with the generator $A$ and we have
\be
u(x,t)=\int_0^te^{-(t-s)A}\lambda(s)f(x)\,ds.
\ee
Hence,
\be
\pp_tu(x_0,t)=\lambda(t)f(x_0)-\int_0^tAe^{-(t-s)A}\lambda(s)f(x_0)\,ds.
\ee
Thus we obtain the Volterra integral equation of the second kind:\footnote{
The Volterra integral equation of the first kind is given by
\[
\xi(t)=\int_{t_0}^tK(t,s)\lambda(s)\,ds.
\]
The Volterra integral equation of the second kind is written as
\[
\lambda(t)=\xi(t)+\int_{t_0}^tK(t,s)\lambda(s)\,ds.
\]
}
\ba
\lambda(t)
&=
\frac{\pp_tu(x_0,t)}{f(x_0)}+\int_0^tAe^{-(t-s)A}\lambda(s)\,ds
\\
&=
\frac{\pp_tu(x_0,t)}{f(x_0)}+(B\lambda)(t),
\ea
where the operator $B$ was introduced.

Let us assume that there exists $M\in(0,1)$ such that
\be
\|B\lambda\|_{C([0,T])}\le M\|\lambda\|_{C([0,T])}.
\ee
Hence,
\be
\|B^k\lambda\|_{C([0,T])}\le M^k\|\lambda\|_{C([0,T])}
\ee
for integer $k$. This means that the series $(I-B)^{-1}=I+B+B^2+\cdots$, where $I$ is the identity, exists and
\be
\left\|(I-B)^{-1}\right\|\le e^M.
\ee
Since $(I-B)\lambda(t)=\pp_tu(x_0,t)/f(x_0)$, we obtain
\be
\|\lambda\|_{C([0,T])}\le C\|u(x_0,\cdot)\|_{C^1([0,T])}.
\ee
Thus, $\lambda$ is obtained.

In the above problem, we assumed $f(x_0)\neq0$ as a priori knowledge. However, such a priori information is not always available. In general, a priori knowledge makes it easier to solve inverse problems. In Sec.~\ref{sec9}, we will come back to the inverse problem for the heat equation with a slightly different setup.

\subsection{Example 2: Contaminated water in a tank}
\hfill\vskip1mm

Let us consider a problem of contaminated water flowing into a tank \cite{Kamimura14}. As shown in Fig.~\ref{sec1:fig01}, contaminated water flows into a water tank of volume $V$. The contaminated water of volume $v$ enters the tank per unit time. Although $v$ is unknown, the volume of water in the tank remains constant because the same amount $v$ of water leaks out. The concentration $b$ of contaminants in the contaminated water is unknown. We can observe the concentration $a(t)$ of contaminants in the water tank at time $t$. We here wish to solve the inverse problem of determining $b,v$ from $a(t)$:
\be
\mbox{Inverse problem:}\quad a(t)\rightarrow b,v
\ee

\begin{figure}[ht]
\begin{center}
\includegraphics[width=0.4\textwidth]{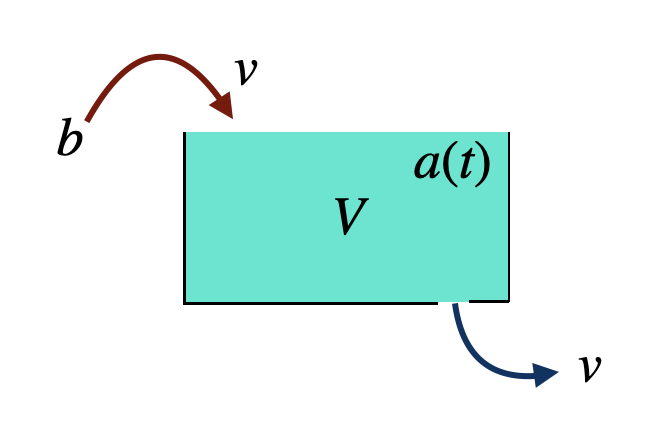}
\end{center}
\caption{
Contaminated water flows into a tank.
}
\label{sec1:fig01}
\end{figure}

The concentration of contaminants in the tank after small time $\Delta t$ ($>0$) can be written as
\be
a(t+\Delta t)=\frac{(V-v\Delta t)a(t)+v(\Delta t)b}{V}.
\ee
That is,
\be
\frac{a(t+\Delta t)-a(t)}{\Delta t}=(b-a(t))y,\quad y=\frac{v}{V}.
\ee
Thus, in the limit of $\Delta t\to0$, we arrive at the following model:
\be
\frac{da}{dt}=y(b-a).
\ee

If we solve the above differential equation, we obtain
\be
a(t)=b-e^{-yt}(b-a_0),\quad a_0=a(0).
\ee
Let us introduce
\be
A(t)=\frac{a(t)-a_0}{a_0},\quad x=\frac{b-a_0}{a_0}.
\ee
Then the above solution is expressed as
\begin{equation}
A(t)=\left(1-e^{-yt}\right)x.
\label{sec1:tank06}
\end{equation}

Suppose we observe $A(t)$ twice at $t=1,2$ (we note that $A(0)=0$):
\begin{equation}
A_1=A(1)=\left(1-e^{-y}\right)x,\quad A_2=A(2)=\left(1-e^{-2y}\right)x.
\label{sec1:tank07}
\end{equation}
Now, the inverse problem is formulated as
\be
\mbox{Inverse problem:}\quad A_1,A_2\rightarrow x,y
\ee
Let us assume the following parameters:\footnote{
The becquerel (${\rm Bq}$) is the SI unit for measuring the amount of radioactivity.}
\be
V=10^4\,{\rm kL},\quad v=200\,{\rm kL}/{\rm day},\quad b=120\,{\rm Bq}/{\rm L},\quad
a_0=24\,{\rm Bq}/{\rm L}.
\ee
Then we have
\be
x=4,\quad y=0.02.
\ee
The question is if we can reconstruct these $x,y$ from $A_1,A_2$. According to the relation (\ref{sec1:tank06}), if there is no noise, we obtain
\be
A_1=A_1^{\rm true}=0.079205\dots,\quad
A_2=A_2^{\rm true}=0.15684\dots.
\ee

Suppose we obtain $A_1,A_2$ from observation. By solving the inverse problem using (\ref{sec1:tank07}), we obtain
\be
x=\frac{A_1^2}{2A_1-A_2},\quad y=\ln\frac{A_1}{A_2-A_1}.
\ee

Since actual measurements contain noise, we can write $A_1=A_1^{\rm obs}$, $A_2=A_2^{\rm obs}$. First we suppose that
\be
A_1^{\rm obs}=0.0792,\quad A_2^{\rm obs}=0.157.
\ee
That is, there are three significant digits. Then we obtain
\be
x=4.48\dots,\quad y=0.0178\dots.
\ee
Next let us suppose that observed values are obtained as
\be
A_1^{\rm obs}=0.0791,\quad A_2^{\rm obs}=0.158.
\ee
In this case, we obtain
\be
x=31.28\dots,\quad y=0.0253\dots.
\ee
The situation is summarized in Table \ref{table1}. We see that in particular the reconstruction of $x$ is unstable.

\begin{table}[h]
\begin{tabular}{c|ll|ll}
 & $A_1$ & $A_2$ & $x$ & $y$ \\ \hline
Exact & $0.079205\dots$ & $0.15684\dots$ & \phantom{0}$4$ & $0.02$ \\
Case 1 & $0.0792$ & $0.157$ & \phantom{0}$4.48\dots$ & $0.0178\dots$ \\
Case 2 & $0.0791$ & $0.158$ & $31.28\dots$ & $0.0253\dots$
\end{tabular}
\label{table1}
\end{table}

By this example, we have seen that even though the forward problem is to solve a simple first-order differential equation, its corresponding inverse problem might be a lot harder.

\section{Stability}
\label{sec2}

In the example of the contaminated water in a tank, finding $x$ was not stable. It is of interest to estimate stability of inverse problems.

\subsection{Transport equation}
\label{sec2.1}
\hfill\vskip1mm

Let us consider the following first-order partial differential equation.
\begin{equation}
\left\{\begin{aligned}
\pp_tu(x,t)+\pp_xu(x,t)+p(x)u(x,t)=0,&\quad 0<x<\ell,\quad-T<t<T,
\\
u(x,0)=a(x),&\quad 0<x<\ell,
\\
u(0,t)=b(t),&\quad -T<t<T.
\end{aligned}\right.
\label{transport1d}
\end{equation}
This is the radiative transport equation in one spatial dimension without scattering. During the X-ray examination at a hospital, X-rays in the human body obey the above transport equation (if negligible scattering is ignored). The coefficient $p(x)$ describes absorption of X-rays in the medium.

Here we suppose that $p\in L^{\infty}(0,\ell)$ is unknown but $u$ can be observed at $x=\ell$. We wish to solve the following inverse problem:
\be
\mbox{Inverse problem:}\quad u(\ell,t)\rightarrow p(x)\quad
(-T<t<T,\;0<x<\ell)
\ee
Our goal is to obtain a stability estimate using some norm $\|\cdot\|_*$ such as
\be
\|p-q\|_{L^2(0,\ell)}\le C\left\|u[p](\ell,\cdot)-u[q](\ell,\cdot)\right\|_*.
\ee
In the above estimate and hereafter in this section, we let $C>0$ be a generic constant.

\subsection{Linearization}
\hfill\vskip1mm

Let us consider two transport equations with different absorption coefficients:
\be
\left\{\begin{aligned}
\pp_tu[p](x,t)+\pp_xu[p](x,t)+p(x)u[p](x,t)=0,&\quad 0<x<\ell,\quad-T<t<T,
\\
u[p](x,0)=a(x),&\quad 0<x<\ell,
\\
u[p](0,t)=b(t),&\quad -T<t<T,
\end{aligned}\right.
\ee
and
\be
\left\{\begin{aligned}
\pp_tu[q](x,t)+\pp_xu[q](x,t)+q(x)u[q](x,t)=0,&\quad 0<x<\ell,\quad-T<t<T,
\\
u[q](x,0)=a(x),&\quad 0<x<\ell,
\\
u[q](0,t)=b(t),&\quad -T<t<T.
\end{aligned}\right.
\ee
By subtraction we obtain
\be
\left\{\begin{aligned}
\left(\pp_t+\pp_x+p\right)\left(u[p]-u[q]\right)=-(p-q)u[q],&\quad 0<x<\ell,\quad-T<t<T,
\\
\left(u[p]-u[q]\right)(x,0)=0,&\quad 0<x<\ell,
\\
\left(u[p]-u[q]\right)(0,t)=0,&\quad -T<t<T.
\end{aligned}\right.
\ee

Let us set
\be
\widetilde{u}(x,t)=u[p](x,t)-u[q](x,t),\quad
f(x)=p(x)-q(x),\quad R(x,t)=-u[q](x,t).
\ee
Then we have
\be
\left\{\begin{aligned}
\left(\pp_t+\pp_x+p(x)\right)\widetilde{u}(x,t)=f(x)R(x,t),&\quad 0<x<\ell,\quad-T<t<T,
\\
\widetilde{u}(x,0)=0,&\quad 0<x<\ell,
\\
\widetilde{u}(0,t)=0,&\quad -T<t<T.
\end{aligned}\right.
\ee
In this way, the inverse problem of determining the coefficient $p$ has become the inverse problem of determining $f$ in the source term. We note that in Sec.~\ref{sec2:weightfunc}, we assume $|R|\le M$ with a positive constant $M$.

Let us introduce
\be
y=\pp_t\widetilde{u}.
\ee
We have
\begin{equation}
\left\{\begin{aligned}
Py(x,t)=\left(\pp_t+\pp_x+p(x)\right)y(x,t)=f(x)\pp_tR(x,t),&\quad 0<x<\ell,\quad-T<t<T,
\\
y(x,0)=f(x)R(x,0),&\quad 0<x<\ell,
\\
y(0,t)=0,&\quad -T<t<T.
\end{aligned}\right.
\label{sec2:leq1}
\end{equation}
In (\ref{sec2:leq1}), the initial value is nonzero. This is necessary for the stability estimate below.

\subsection{Carleman estimates}
\hfill\vskip1mm

A Carleman estimate is an $L^2$-weighted estimate with large parameters \cite{Carleman39,Hoermander63}. This provides a powerful tool for inverse problems for partial differential equations \cite{Bukhgeim-Klibanov81,Imanuvilov-Yamamoto01,Isakov06,Yamamoto09}. Let us set the weight function $\va(x,t)$ as
\begin{equation}
\va(x,t)=|x-x_0|^2-\beta t^2,\quad x_0\neq[0,\ell],\quad 0<\beta<1.
\label{sec2:weightfunc}
\end{equation}

\begin{prop}
Assume that $w\in C_0^1([0,\ell]\times[-T,T])$ satisfies $w(0,t)=w(x,\pm T)=0$. Then there exist constants $C>0$, $s_0>0$ such that
\be
\int_{-T}^T\int_0^{\ell}s|w|^2e^{2s\va}\,dxdt\le
C\int_{-T}^T\int_0^{\ell}\left|Pw\right|^2e^{2s\va}\,dxdt+
Cs\int_{-T}^T|w(\ell,t)|^2e^{2s\va(\ell,t)}\,dt
\ee
for all $s\ge s_0$.
\end{prop}

\begin{proof}
Let us set
\be
P_0w(x,t):=\pp_tw(x,t)+\pp_xw(x,t),\quad 0<x<\ell,\quad -T<t<T.
\ee
By the change $z=e^{s\va}w$, we have
\be
P_0z=sA(x,t)z+e^{s\va}P_0w,
\ee
where
\be
A(x,t)=\pp_t\va+\pp_x\va=-2\beta t+2(x-x_0).
\ee
Hence \cite{Machida-Yamamoto14},
\ba
\int_{-T}^T\int_0^{\ell}\left|P_0w\right|^2e^{2s\va}\,dxdt
&=
\int_{-T}^T\int_0^{\ell}\left(P_0z-sAz\right)^2\,dxdt
\\
&\ge
-2s\int_{-T}^T\int_0^{\ell}A(\pp_tz+\pp_xz)z\,dxdt
\\
&=
-s\int_{-T}^T\int_0^{\ell}A(\pp_tz^2+\pp_xz^2)\,dxdt
\\
&=
s\int_{-T}^T\int_0^{\ell}(\pp_tA+\pp_xA)z^2\,dxdt-s\int_{-T}^TA(\ell,t)z(\ell,t)^2\,dt
\\
&=
2s(1-\beta)\int_{-T}^T\int_0^{\ell}z^2\,dxdt-s\int_{-T}^TA(\ell,t)z(\ell,t)^2\,dt
\\
&\ge
2s(1-\beta)\int_{-T}^T\int_0^{\ell}z^2\,dxdt-Cs\int_{-T}^Tz(\ell,t)^2\,dt.
\ea
Thus,
\be
2(1-\beta)\int_{-T}^T\int_0^{\ell}s|w|^2e^{2s\va}\,dxdt
\le
\int_{-T}^T\int_0^{\ell}\left|P_0w\right|^2e^{2s\va}\,dxdt+
Cs\int_{-T}^T|w(\ell,t)|^2e^{2s\va(\ell,t)}\,dt.
\ee

Recalling $P_0w=Pw-pw$, we have
\be
\int_{-T}^T\int_0^{\ell}\left|P_0w\right|^2e^{2s\va}\,dxdt
\le
2\int_{-T}^T\int_0^{\ell}\left|Pw\right|^2e^{2s\va}\,dxdt+
2\int_{-T}^T\int_0^{\ell}|pw|^2e^{2s\va}\,dxdt.
\ee
We obtain
\ba
2(1-\beta)\int_{-T}^T\int_0^{\ell}s|w|^2e^{2s\va}\,dxdt
&\le
2\int_{-T}^T\int_0^{\ell}\left|Pw\right|^2e^{2s\va}\,dxdt+
2\int_{-T}^T\int_0^{\ell}|pw|^2e^{2s\va}\,dxdt
\\
&+
Cs\int_{-T}^T|w(\ell,t)|^2e^{2s\va(\ell,t)}\,dt.
\ea
By moving the second term on the right-hand side to the left-hand side, we obtain the inequality in the proposition for sufficiently large $s$. This completes the proof.
\end{proof}

\subsection{Stability estimate}
\label{sec2:weightfunc}
\hfill\vskip1mm

Let us consider (\ref{sec2:leq1}). We assume that $R,\pp_tR\in L^2(-T,T;\,L^{\infty}(0,\ell))$, $p\in L^{\infty}(0,\ell)$, $f\in L^2(0,\ell)$, and $|R(x,0)|\ge a_0$ ($x\in[0,\ell]$) with a constant $a_0>0$. To estimate $f$ in the source term, we will use the Carleman estimate for $\pp_tu$. We begin by introducing the cut-off function $\chi(t)$ to have a function which is zero at $t=\pm T$: The cut-off function $\chi\in C_0^{\infty}(\Rm)$, $0\le\chi\le1$, is introduced as shown in Fig.~\ref{sec2:fig02}. A more explicit form of $\chi$ is given below.

\begin{figure}[ht]
\begin{center}
\includegraphics[width=0.4\textwidth]{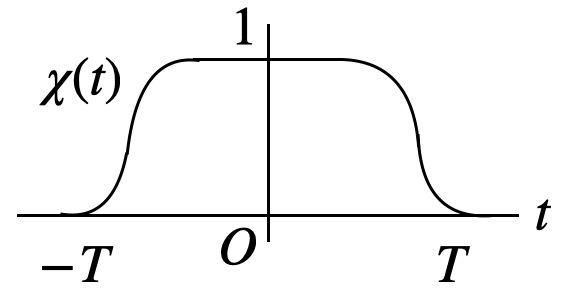}
\end{center}
\caption{
Cutoff function $\chi$.
}
\label{sec2:fig02}
\end{figure}

We set
\be
z=\chi ye^{s\va}.
\ee
We note that
\ba
\pp_tz&=
(\pp_t\chi)ye^{s\va}+\chi(\pp_ty)e^{s\va}+\chi y\pp_te^{s\va},
\\
\pp_xz&=
\chi(\pp_xy)e^{s\va}+\chi ys(\pp_x\va)e^{s\va}.
\ea
Hence,
\begin{equation}\begin{aligned}
Pz&=
\pp_tz(x,t)+\pp_xz(x,t)+p(x)z(x,t)
\\
&=
\chi e^{s\va}\left(\pp_ty+\pp_xy+py\right)+(\pp_t\chi)ye^{s\va}+\chi ys\left(\pp_x\va+\pp_t\va\right)e^{s\va}
\\
&=
\chi e^{s\va}Py+(\pp_t\chi)ye^{s\va}+\chi ys\left(\pp_x\va+\pp_t\va\right)e^{s\va}
\\
&=
\chi e^{s\va}f\pp_tR+(\pp_t\chi)ye^{s\va}+\chi sy\left(\pp_x\va+\pp_t\va\right)e^{s\va}.
\end{aligned}
\label{Pzeq}
\end{equation}
Let us introduce
\be
w=\chi y.
\ee
We note that $w(0,\cdot)=w(\cdot,\pm T)=0$. Recalling (\ref{sec2:leq1}), we obtain
\be
Pw=\chi f\pp_tR+(\pp_t\chi)y.
\ee

We assume that $T$ is sufficiently large such that
\be
T>\frac{1}{\sqrt{\beta}}\sup_{0\le x\le\ell}|x-x_0|,
\ee
where $\beta$ was introduced in (\ref{sec2:weightfunc}). The physical reason for this condition is that the propagation speed of a hyperbolic equation is finite. We note that the weight function $\va$ in (\ref{sec2:weightfunc}) satisfies
\ba
&
\va(x,\pm T)=|x-x_0|^2-\beta T^2<0,\quad x\in[0,\ell],
\\
&
\va(x,0)>0,\quad x\in[0,\ell].
\ea
For any $\varepsilon>0$ such that $\varepsilon<\min(|\max_{0<x<\ell}\va(x,T)|,\min_{0<x<\ell}\va(x,0))$, there exists $\delta>0$ such that
\ba
&
\va(x,t)<-\varepsilon,\quad -T\le t\le-T+2\delta,\quad T-2\delta\le t\le T, \quad x\in[0,\ell],
\\
&
\va(x,t)>\varepsilon,\quad -\delta\le t\le\delta,\quad x\in[0,\ell].
\ea

We give the cutoff function $\chi\in C_0^{\infty}(\Rm)$ as
\be
\chi(t)=\left\{\begin{aligned}
1,&\quad -T+2\delta\le t\le T-2\delta,
\\
0,&\quad -T\le t\le -T+\delta,\quad T-\delta\le t\le T.
\end{aligned}\right.
\ee
The relation between $\chi(t)$ and $\va$ is shown in Fig.~\ref{sec2:fig03}.

\begin{figure}[ht]
\begin{center}
\includegraphics[width=0.4\textwidth]{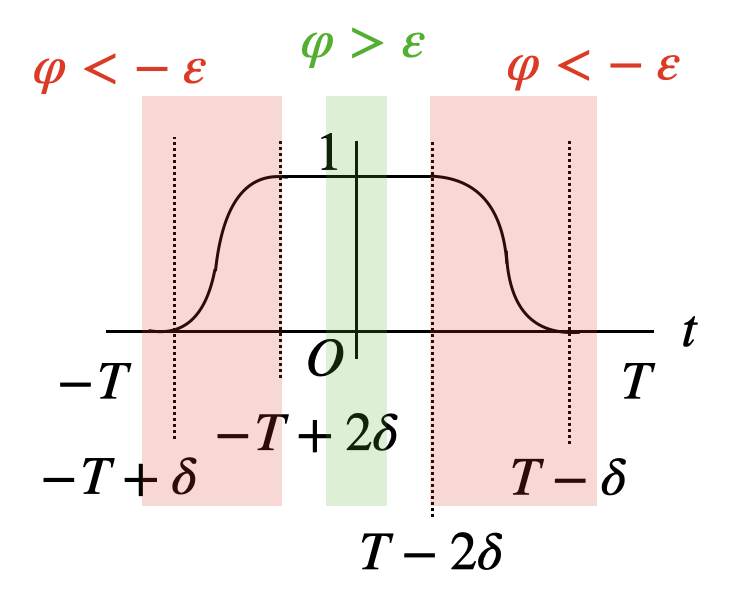}
\end{center}
\caption{
The weight function $\va$ and cutoff function $\chi$.
}
\label{sec2:fig03}
\end{figure}

We use the Carleman estimate for $w$:
\ba
&
s\int_{-T}^T\int_0^{\ell}|w|^2e^{2s\va}\,dxdt
\\
&\le
C\int_{-T}^T\int_0^{\ell}\left|\chi f\pp_tR+(\pp_t\chi)y\right|^2
e^{2s\va(x,t)}\,dxdt+
Cs\int_{-T}^T|w(\ell,t)|^2e^{2s\va(\ell,t)}\,dt
\\
&\le
C\int_{-T}^T\int_0^{\ell}\left|\chi f\pp_tR\right|^2e^{2s\va(x,t)}\,dxdt
+C\int_{-T}^T\int_0^{\ell}|\pp_t\chi|^2|y|^2e^{2s\va}\,dxdt
\\
&+
Cs\int_{-T}^T|w(\ell,t)|^2e^{2s\va(\ell,t)}\,dt.
\ea
Hereafter we assume that there exists a constant $M>0$ such that $\|\widetilde{u}\|_{L^2((0,\ell)\times(-T,T))}\le M$ and $\|y\|_{L^2((0,\ell)\times(-T,T))}\le M$. We also assume that $s$ is sufficiently large.

Since we have
\be
\int_{-T}^T\int_0^{\ell}\left|\chi f\pp_tR\right|^2e^{2s\va(x,t)}\,dxdt\le
C\int_{-T}^T\int_0^{\ell}|f|^2e^{2s\va(x,t)}\,dxdt,
\ee
we obtain
\ba
s\int_{-T}^T\int_0^{\ell}|w|^2e^{2s\va}\,dxdt
&\le
C\int_0^{\ell}|f|^2\left(\int_{-T}^Te^{-2s\beta t^2}\,dt\right)e^{2s\va(x,0)}\,dx
\\
&+
C\int_{-T}^T\int_0^{\ell}(\pp_t\chi)^2|y|^2e^{2s\va}\,dxdt
\\
&+
Cs\int_{-T}^T|w(\ell,t)|^2e^{2s\va(\ell,t)}\,dt.
\ea
By the Lebesgue theorem we can exchange the integral and limit, and\footnote{
If $f(x)=o(g(x))$ as $x\to\infty$ with Landau's little-o notation, it means that for any $\varepsilon>0$, there exists a constant $M>0$ such that $|f(x)|\le\varepsilon|g(x)|$ for all $x>M$. In other words, we have $|f/g|\to0$ as $x\to\infty$.
}
\be
\int_{-T}^Te^{-2s\beta t^2}\,dt=o(1)\quad(s\to\infty).
\ee
We note that
\be
\left\{\begin{aligned}
\pp_t\chi=0&\quad\mbox{for}\;t\in(-T,-T+\delta)\cup(-T+2\delta,T-2\delta)\cup(T-\delta,T),
\\
\pp_t\chi\neq0,\;\va<-\varepsilon&\quad\mbox{for}\;t\in[-T+\delta,-T+2\delta]\cup[T-2\delta,T-\delta].
\end{aligned}\right.
\ee
Hence,
\be
\int_{-T}^T\int_0^{\ell}(\pp_t\chi)^2|y|^2e^{2s\va}\,dxdt\le
Ce^{-2s\varepsilon}\int_{-T}^T\int_0^{\ell}|y|^2\,dxdt
\le
CM^2e^{-2s\varepsilon}.
\ee
Therefore,
\begin{equation}
\begin{aligned}
s\int_{-T}^T\int_0^{\ell}|w|^2e^{2s\va}\,dxdt
&\le
o(1)\int_0^{\ell}|f|^2e^{2s\va(x,0)}\,dx
+CM^2e^{-2s\varepsilon}
\\
&+
Ce^{Cs}\int_{-T}^T|y(\ell,t)|^2\,dt,
\end{aligned}
\label{o1inequality}
\end{equation}
where we used $s\exp(2s\max_t\va(\ell,t))\le\exp(Cs)$.

Now, let us multiply (\ref{Pzeq}) by $2z$ and integrate over $x,t$:
\begin{equation}\begin{aligned}
\int_{-T}^0\int_0^{\ell}\pp_t|z|^2\,dxdt
&=
-\int_{-T}^0\int_0^{\ell}\pp_x|z|^2\,dxdt-\int_{-T}^0\int_0^{\ell}2p|z|^2\,dxdt
\\
&+
\int_{-T}^0\int_0^{\ell}2\left(\chi e^{s\va}f(\pp_tR)z+(\pp_t\chi)ye^{s\va}z+s(\pp_t\va+\pp_x\va)z^2\right)\,dxdt.
\end{aligned}
\label{intPzeq}
\end{equation}
Recall $|R(x,0)|=a(x)\ge a_0$, $x\in[0,\ell]$. Then the left-hand side of (\ref{intPzeq}) becomes
\be
\mbox{LHS of (\ref{intPzeq})}=\int_0^{\ell}|z(x,0)|^2\,dx=
\int_0^{\ell}|y(x,0)|^2e^{2s\va(x,0)}\,dx\ge
a_0^2\int_0^{\ell}|f(x)|^2e^{2s\va(x,0)}\,dx,
\ee
where we used $z(\cdot,\pm T)=0$ for the first equality and $y(x,0)=f(x)R(x,0)$ for the last inequality. Let us focus on the last term on the right-hand side of (\ref{intPzeq}). We have
\ba
\left|\int_{-T}^0\int_0^{\ell}2\chi e^{s\va}f(\pp_tR)z\,dxdt\right|
&\le
C\left|\int_{-T}^0\int_0^{\ell}fe^{s\va}z\,dxdt\right|
\\
&\le
C\int_{-T}^0\int_0^{\ell}|f|^2e^{2s\va}\,dxdt+C\int_{-T}^0\int_0^{\ell}|z|^2\,dxdt.
\ea
Hence for sufficiently large $s$, the right-hand side of (\ref{intPzeq}) can be estimated as
\ba
\left|\mbox{RHS of (\ref{intPzeq})}\right|
&\le
\int_{-T}^0|z(\ell,t)|^2\,dt+
Cs\int_{-T}^0\int_0^{\ell}|z|^2\,dxdt
\\
&+
C\int_{-T}^0\int_0^{\ell}(\pp_t\chi)^2|y|^2e^{2s\va}\,dxdt+
C\int_{-T}^0\int_0^{\ell}|f|^2e^{2s\va}\,dxdt,
\ea
where we used $z(0,\cdot)=0$ for the first term on the right-hand side of the above inequality. The third term on the right-hand side can be estimated as
\be
\int_{-T}^T\int_0^{\ell}(\pp_t\chi)^2|y|^2e^{2s\va}\,dxdt\le CM^2e^{-2s\varepsilon}.
\ee

After all, we have
\ba
a_0^2\int_0^{\ell}|f|^2e^{2s\va(x,0)}\,dx
&\le
\int_{-T}^T\int_0^{\ell}|f|^2e^{2s\va}\,dxdt+CM^2e^{-2s\varepsilon}
+Cs\int_{-T}^T\int_0^{\ell}|w|^2e^{2s\va}\,dxdt
\\
&+
\int_{-T}^T|z(\ell,t)|^2\,dt.
\ea
Let us recall (\ref{o1inequality}). By the same calculation which we did to obtain (\ref{o1inequality}), the first term on the right-hand side of the above inequality can be estimated as
\be
\int_{-T}^T\int_0^{\ell}|f|^2e^{2s\va}\,dxdt\le
o(1)\int_0^{\ell}|f|^2e^{2s\va(x,0)}\,dx.
\ee
Since $z=we^{s\va}$, we have
\be
a_0^2\int_0^{\ell}|f|^2e^{2s\va(x,0)}\,dx
\le
o(1)\int_0^{\ell}|f|^2e^{2s\va(x,0)}\,dx+CM^2e^{-2s\varepsilon}+Ce^{Cs}d^2,
\ee
where
\begin{equation}
d^2=\int_{-T}^T|y(\ell,t)|^2\,dt.
\label{sec2:d2def}
\end{equation}
Since $\va(x,t)>\varepsilon$ ($-\delta\le t\le\delta$, $x\in[0,\ell]$),
\be
\va(x,0)=|x-x_0|^2>\varepsilon+\beta\delta^2.
\ee
Hence,
\begin{equation}
\left(a_0^2-o(1)\right)e^{2s(\varepsilon+\beta\delta^2)}\int_0^{\ell}|f|^2\,dx
\le
CM^2e^{-2s\varepsilon}+Ce^{Cs}d^2
\le
C\left(e^{-2s\varepsilon}+e^{Cs}d^2\right).
\label{sec2:almostdone}
\end{equation}

The inequality (\ref{sec2:almostdone}) is not informative when $d^2$ in (\ref{sec2:d2def}) is large because $\|\widetilde{u}\|_{L^2((0,\ell)\times(-T,T))}\le M$ implies that $\|f\|_{L^2(0,\ell)}$ is not large (In other words, in the calculation in Sec.~\ref{sec2:weightfunc}, we have already used the assumption that $\|p\|_{L^2(0,\ell)}$ is not very large.). Suppose $d^2$ is small. Then we choose $s$ which minimizes the right-hand side of (\ref{sec2:almostdone}):
\be
\frac{d}{ds}\left(e^{-2s\varepsilon}+e^{Cs}d^2\right)=0\quad\Rightarrow\quad
s=\frac{1}{2\varepsilon+C}\ln\frac{2\varepsilon}{Cd^2}.
\ee
We obtain
\be
\|f\|_{L^2(0,\ell)}^2\le
C\left(\int_{-T}^T|y(\ell,t)|^2\,dt\right)^{\theta},\quad
\theta=\frac{2\varepsilon}{2\varepsilon+C}.
\ee
We note that $0<\theta<1$. That is, for sufficiently large $T>0$, there exists a constant $C>0$ such that
\be
\|f\|_{L^2(0,\ell)}\le C\|y(\ell,\cdot)\|_{L^2(-T,T)}^{\theta},\quad
\theta\in(0,1).
\ee
Thus we have arrived at the theorem below.

We define
\be
X=H^1(-T,T;\,L^{\infty}(0,\ell))\cap H^2(-T,T;\,L^2(0,\ell)).
\ee
Let $M>0$ be a fixed constant. We introduce
\be
\mathcal{U}=\left\{u\in X;\;\|u\|_X+\|\nabla u\|_{H^1(-T,T;\,L^2(0,\ell))}\le M\right\}.
\ee

\begin{thm}
Let $u\in\mathcal{U}$ be a solution to (\ref{transport1d}). We assume that $|a(x)|>0$ ($0\le x\le\ell$) and $\|p\|_{L^{\infty}(0,\ell)},\|q\|_{L^{\infty}(0,\ell)}\le M$. For sufficiently large $T>0$, there exists a constant $C=C(\ell,T,a,b,M)>0$ such that
\be
\|p-q\|_{L^2(0,\ell)}\le C\left\|\frac{\pp(u[p](\ell,\cdot)-u[q](\ell,\cdot))}{\pp t}\right\|_{L^2(-T,T)}^{\theta},
\quad 0<\theta<1.
\ee
\end{thm}

\subsection{Remarks}
\hfill\vskip1mm

There are some remarks:
\begin{itemize}
\item As we have seen, a stability estimate for one input can be obtained by using Carleman estimates.
\item By extending $(0,T)$ to $(-T,T)$ when the Carleman estimate is used, we can prove the same H\"{o}lder estimate for the time interval $(0,T)$ instead of $(-T,T)$.
\item Indeed, the first-order equation which we have considered can be solved by the method of characteristics. Hence the stability analysis by Carlemn estimates is overkill. But the aim here was to understand how to use Carleman estimates for stability analysis. The above stability analysis can be extended to different inverse problems for partial differential equations and stability can be proved with Carleman estimates.
\item Here we obtained the H\"{o}lder stability (i.e., $0<\theta<1$). It is possible to improve the above proof without using estimates $\|\widetilde{u}\|_{L^2((0,\ell)\times(-T,T))}\le M$, $\|y\|_{L^2((0,\ell)\times(-T,T))}\le M$, and obtain the Lipschitz stability (i.e., $\theta=1$).
\end{itemize}

\section{The heat equation}
\label{sec3}

\subsection{Derivation}
\label{sec3:der}
\hfill\vskip1mm

Let us consider the heat conduction in one dimension (see Fig.~\ref{sec3:fig01}). Let $u(x,t)$ be the temperature at position $x$ at time $t$. The following relation, which is called the Fourier law, has been established by various experiments:
\be
J=-\kappa(x)\pp_xu,
\ee
where $J(x,t)$ is the heat flux (energy per unit area per unit time) and $\kappa(x)$ is the thermal conductivity.

\begin{figure}[ht]
\begin{center}
\includegraphics[width=0.4\textwidth]{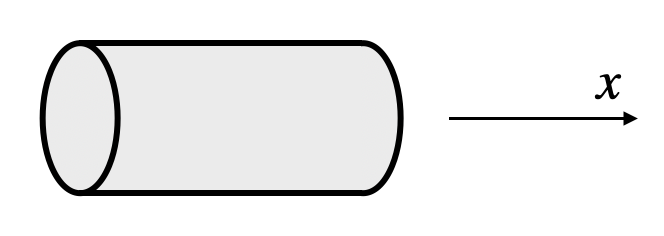}
\end{center}
\caption{
The heat conduction in one dimension.
}
\label{sec3:fig01}
\end{figure}

Let $A$ be the cross section of a thin cylinder of length $\Delta\ell$ (Fig.~\ref{sec3:fig02}). We consider the temperature change in this small volume. Let $C_V$ be the volumetric heat capacity.\footnote{
Using the SI, units of $u,J,\kappa,C_V$ are ${\rm K}$, ${\rm W}/{\rm m}^2$, ${\rm W}/({\rm m}\cdot{\rm K})$, and ${\rm J}/({\rm K}\cdot {\rm m}^3)$, respectively.
}
The energy balance in the small volume can be expressed as 
\be
C_Vu(x,t+\Delta t)A\Delta\ell-C_Vu(x,t)A\Delta\ell=
J(x,t)A\Delta t-J(x+\Delta\ell,t)A\Delta t,
\ee
where the left-hand side shows the energy change, the first and second terms on the right-hand side mean the gain and loss, respectively. By dividing both sides by $A\Delta\ell\Delta t$,
\be
C_V\frac{u(x,t+\Delta t)-u(x,t)}{\Delta t}=
\frac{-J(x+\Delta\ell,t)+J(x,t)}{\Delta\ell}.
\ee
With the help of the Fourier law,
\be
C_V\frac{u(x,t+\Delta t)-u(x,t)}{\Delta t}=
\frac{\kappa(x+\Delta\ell)\left.\pp_yu(y,t)\right|_{y=x+\Delta\ell}-\kappa(x)\left.\pp_yu(y,t)\right|_{y=x}}{\Delta\ell}.
\ee

\begin{figure}[ht]
\begin{center}
\includegraphics[width=0.4\textwidth]{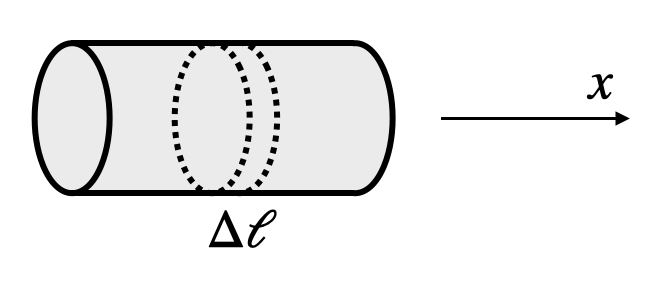}
\end{center}
\caption{
The heat conduction and the heat equation.
}
\label{sec3:fig02}
\end{figure}

In the limits of $\Delta t\to0$ and $\Delta\ell\to 0$, we obtain
\begin{equation}
\pp_tu(x,t)=\pp_xD(x)\pp_xu(x,t),
\label{sec3:de1}
\end{equation}
where $D(x)=\kappa(x)/C_V$. This equation (\ref{sec3:de1}) is called the diffusion equation and $D$ is called the diffusion coefficient. Suppose that $\kappa,D$ are constants and independent of $x$. In this case we can write $D=D_0$ with a positive constant $D_0$. Then we have
\begin{equation}
\pp_tu(x,t)=D_0\pp_x^2u(x,t).
\label{sec3:he1}
\end{equation}
The equation (\ref{sec3:he1}) is said to be the heat equation.

Let us introduce dimensionless variables $x^*,t^*$ as
\be
x^*=\frac{x}{\ell},\quad t^*=\frac{t}{\ell^2/D_0},
\ee
where $\ell>0$ is a typical length for the heat diffusion. Then the heat equation (\ref{sec3:he1}) becomes
\be
\pp_{t^*}u^*(x^*,t^*)=\pp_{x^*}^2u^*(x^*,t^*),
\ee
where $u^*(x^*,t^*)=u(x,t)$. Hereafter in this section we drop the superscript $*$ and consider $\pp_tu(x,t)=\pp_x^2u(x,t)$.

\subsection{Solutions}
\hfill\vskip1mm

Let us consider the following equation:
\begin{equation}
\left\{\begin{aligned}
\pp_tu(x,t)=\pp_x^2u(x,t),&\quad 0<x<\pi,\quad 0<t<T,
\\
u(0,t)=u(\pi,t)=0,&\quad 0<t<T,
\\
u(x,0)=f(x),&\quad 0<x<\pi.
\end{aligned}\right.
\label{sec3:heq2}
\end{equation}
Here, $f$ is the distribution of temperature at $t=0$. We suppose that $f$ is piecewise smooth and $f(0)=f(\pi)=0$. We solve the problem by separation of variables. Let us write
\be
u(x,t)=\phi(x)\tau(t),\quad u\not\equiv0.
\ee
Then we have
\be
\frac{1}{\tau(t)}\frac{d\tau}{dt}(t)=\frac{1}{\phi(x)}\frac{d^2\phi}{dx^2}(x).
\ee
By using a constant $p$, we can separate the above equation as
\be
\frac{d^2\phi}{dx^2}(x)-p\phi(x)=0,\quad
\frac{d\tau}{dt}(t)-p\tau(t)=0.
\ee
We note the boundary condition, $\phi(0)=\phi(\pi)=0$. We end up with $u\equiv0$ when $p$ is nonnegative. To obtain $u\not\equiv0$, we set $p=-\mu^2$ with a positive constant $\mu$. From the equation for $\phi$, we obtain
\be
\phi(x)=A\cos{\mu x}+B\sin{\mu x}.
\ee
By the boundary condition we obtain $A=0$, $B\sin{\mu \pi}=0$. Since $u\not\equiv0$, $B\neq0$. Hence,
\be
\mu=\mu_n=n,\quad n=1,2,\dots.
\ee
Although $B$ is an arbitrary constant, let us impose $\int_0^{\pi}\phi(x)^2\,dx=1$ and fix $B$. Then we obtain
\be
\phi(x)=\phi_n(x)=\sqrt{\frac{2}{\pi}}\sin{nx},\quad n\in\Nm.
\ee
We note that $-\mu_n^2$ are eigenvalues and $\phi_n(x)$ are eigenfunctions:
\be
\frac{d^2}{dx^2}\phi_n(x)=-\mu_n^2\phi_n(x),\quad n\in\Nm.
\ee
The eigenfunctions $\{\phi_n\}$ form an orthonormal basis in $L^2(0,\pi)$:\footnote{
The completeness can also be shown (the Sturm-Liouville theory) \cite{Yosida78}. That is, for $f\in L^2(0,\pi)$,
\[
\int_0^{\pi}f(x)\phi_n(x)\,dx=0\quad(\forall n\in\Nm)\quad\Rightarrow\quad
f=0.
\]
}
\be
\int_0^{\pi}\phi_n(x)\phi_m(x)\,dx=\left\langle\phi_n,\phi_m\right\rangle
=\delta_{nm},
\ee
where $\delta_{nm}$ is the Kronecker delta ($\delta_{nm}=1$ for $n=m$ and $\delta_{nm}=0$ otherwise).

We obtain the general solution for $\tau(t)$ as
\be
\tau(t)=\tau_n(t)=C_ne^{-n^2t},\quad n\in\Nm,
\ee
where $C_n$ are constants. Hence, $u(x,t)$ can be expressed as
\be
u(x,t)=\sum_{n=1}^{\infty}C_ne^{-n^2t}\phi_n(x).
\ee
Since
\be
u(x,0)=\sum_{n=1}^{\infty}C_n\phi_n(x)=f(x),\quad 0<x<\pi,
\ee
we obtain
\be
C_n=\int_0^{\pi}f(x)\phi_n(x)\,dx,\quad n=1,2,\dots.
\ee
Finally, the solution to the heat equation (\ref{sec3:heq2}) is obtained as
\begin{equation}
u(x,t)=\sum_{n=1}^{\infty}\left(\int_0^{\pi}f(x)\phi_n(x)\,dx\right)
e^{-n^2t}\phi_n(x).
\label{sec3:heqsol}
\end{equation}

\subsection{Smoothing effect and infinite propagation speed}
\hfill\vskip1mm

Equation (\ref{sec3:heqsol}) implies that $u(x,t_0)$ is infinitely differentiable for any $t_0>0$ even if the initial value $f\in L^{\infty}(0,\pi)$ is not continuous but has discontinuous points. This is called the smoothing effect of the heat equation.

The solution $u$ in (\ref{sec3:heqsol}) and its derivatives converge if the following $s_k,c_k$ converge for $k=0,1,2,\dots$.
\be
s_k=\sum_{n=1}^{\infty}C_nn^ke^{-n^2t_0}\sqrt{\frac{2}{\pi}}\sin{nx},\quad
c_k=\sum_{n=1}^{\infty}D_nn^ke^{-n^2t_0}\sqrt{\frac{2}{\pi}}\cos{nx},
\ee
where $D_n$ are constants. Since $|C_n|\le\sqrt{2/\pi}\int_0^{\pi}|f(x)|\,dx$, we have
\be
|s_k|,|c_k|\le\left(\frac{2}{\pi}\int_0^{\pi}|f(x)|\,dx\right)n^ke^{-n^2t_0},
\quad 0\le x\le\pi.
\ee
By d'Alembert's ratio test,\footnote{
A series $\sum_{n=1}^{\infty}a_n$ converges absolutely if $\lim_{n\to\infty}|a_{n+1}/a_n|<1$.
} $\sum_{n=1}^{\infty}n^ke^{-n^2t_0}$ converges absolutely. Hence we differentiate $u(x,t_0)$ $k$ times.\footnote{
We note the following two lemmas. Lemma 1 ($M$-test): Let $\{f_n\}$ be a sequence of functions in $a\le x\le b$, $c\le t\le d$. If the series $\sum_{n=1}^{\infty}a_n$ converges for $|f_n(x,t)|\le a_n$ ($x\in[a,b]$, $t\in[c,d]$, $n\in\Nm$), then $\sum_{n=1}^{\infty}f_n$ converges absolutely and uniformly in $[a,b]\times[c,d]$. Lemma 2 (interchanging differentiation and integration): Let $\{f_n\}$ be a differentiable sequence of functions in $x\in[a,b]$, $t\in[c,d]$. We suppose that $\pp_xf_n$, $\pp_tf_n$ are continuous in $Q=[a,b]\times[c,d]$. We further assume that $f_n$ converges to $f$ pointwise for each $(x,t)\in Q$, and $\pp_xf_n$, $\pp_tf_n$ converge to $g,h$ uniformly in $Q$. Then $f_n$ converges to $f$ uniformly in $Q$. Moreover, $\lim_{n\to\infty}\pp_xf_n=\pp_x(\lim_{n\to\infty}f_n)$, $\lim_{n\to\infty}\pp_tf_n=\pp_t(\lim_{n\to\infty}f_n)$, that is, $\pp_xf=g$, $\pp_tf=h$.
} Since $k$ is arbitrary, we can conclude that $u(x,t_0)$ is infinitely differentiable.

If $t_0>0$ is small, $u(x,t_0)$ and $f(x)$ should be similar. Still, $u$ is continuous (infinitely differentiable) even when $f$ is not continuous. This means that heat propagates at an infinite speed. This is said to be the infinite propagation speed for the heat equation. Of course, this is a property of the partial differential equation and heat conduction as a natural phenomenon has a finite speed.

\begin{figure}[ht]
\begin{center}
\includegraphics[width=0.4\textwidth]{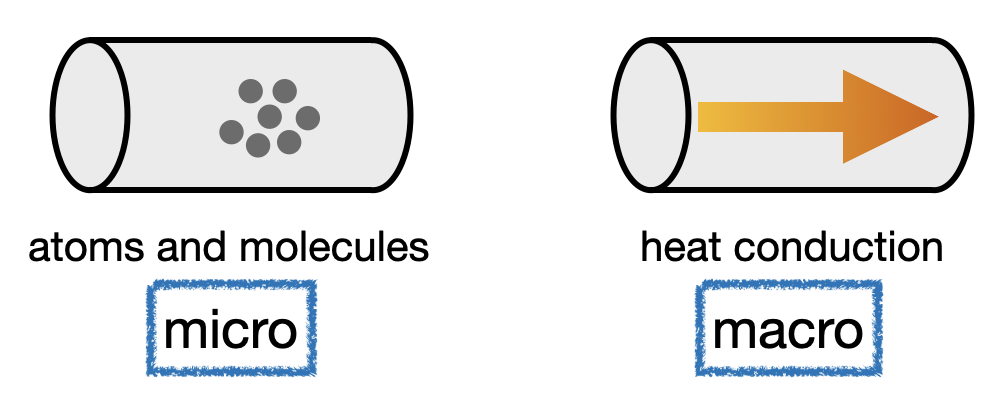}
\end{center}
\caption{
The heat equation governs heat conduction on a macroscopic scale.
}
\label{sec3:fig03}
\end{figure}

Nature has hierarchy; every differential equation that describes a natural phenomenon has its own domain of applicability. When the heat equation is used as a governing equation for heat conduction, it is implicitly assumed that $x,t$ are at least reasonably large. The heat equation does not describe phenomena at the atomic and molecular scale (see Fig.~\ref{sec3:fig03}).

As mathematics, it is possible to assume any small $t_0>0$. When $t_0$ is very small, however, the relation of the heat equation to natural phenomena is lost.

\section{Singular value decomposition}
\label{sec4}

\subsection{Matrix decomposition}
\label{sec4:mat}
\hfill\vskip1mm

Let us consider $A\in\Rm^{m\times n}$ for integers $m,n$. By the singular value decomposition (SVD), the matrix $A$ can be expressed as
\be
A=UDV^T,
\ee
where $U\in\Rm^{m\times m}$ and $V\in\Rm^{n\times n}$ are orthogonal matrices, and the rectangular diagonal matrix $D\in\Rm^{m\times n}$ contains singular values $\mu_j\ge0$ ($j=1,\dots,\min(m,n)$). We label singular values as $\mu_1\ge\cdots\ge\mu_{\min(m,n)}$. Let $\lambda_j$ ($j=1,\dots,\min(m,n)$) ($\lambda_1\ge\cdots\ge\lambda_{\min(m,n)}\ge0$) be the eigenvalues of $A^TA$ ($m\ge n$) or $AA^T$ ($m<n$). We have the relation $\mu_j=\sqrt{\lambda_j}$ ($j=1,\dots,\min(m,n)$). We write
\be
U=(\cdots\psi_i\cdots),\quad V=(\cdots \va_k\cdots),
\ee
where $\psi_i\in\Rm^m$ ($i=1,\dots,m$), $\va_k\in\Rm^n$ ($k=1,\dots,n$).

\subsection{Linear operator}
\hfill\vskip1mm

Let $H_1,H_2$ be real Hilbert spaces \cite{Yamamoto-Kim08}. We consider operator $K\colon H_1\to H_2$ which linear, one-to-one, and compact. We write the conjugate operator of $K$ as $K^*\colon H_2\to H_1$. We have $\langle Kf,g\rangle=\langle f,K^*g\rangle$ ($f\in H_1$, $g\in H_2$). Since $(K^*K)^*=K^*(K^*)^*=K^*K$, we see that $K^*K$ is a compact self-adjoint operator.\footnote{
We note that $(AB)^*=B^*A^*$, $(A^*)^*=A$.
} We have
\be
K^*K\xi=0\quad\Rightarrow\quad
\langle K^*K\xi,\xi\rangle=\langle K\xi,(K^*)^*\xi\rangle=\langle K\xi,K\xi\rangle=\|K\xi\|_{H_2}^2=0
\ee
for $\xi\in H_1$. Since $K$ is one-to-one, this means $\xi=0$. Hence, $K^*K$ is also one-to-one.

Let us write
\begin{equation}
K^*K\phi_n=\lambda_n\phi_n,\quad\phi_n\neq0,
\label{KstaKeigen}
\end{equation}
where $\phi_n,\lambda_n$ are eigenfunctions and eigenvalues of $K^*K$. We assume that $\{\phi_n\}$ form an orthonormal basis in $H_1$. We have
\be
\langle K^*K\phi_n,\phi_n\rangle=\lambda_n\|\phi_n\|_{H_1}^2.
\ee
Since the left-hand side is $\|K\phi\|_{H_2}^2$, we have
\be
\lambda_n=\frac{\|K\phi_n\|_{H_2}^2}{\|\phi_n\|_{H_1}^2}>0.
\ee
Thus, $\lambda_n$ are positive. Let us arrange them in descending order:
\be
\lambda_1\ge\lambda_2\ge\cdots.
\ee

Suppose that $\langle\xi,\phi_n\rangle=0$ ($n\in\Nm$) for some $\xi\in H_1$. Then $K^*K\xi=0$ because
\be
K^*K\xi=\sum_{n=1}^{\infty}\lambda_n\langle\xi,\phi_n\rangle\phi_n.
\ee
Since $K^*K$ is one-to-one, we have $\xi=0$. This implies that $\{\phi_n\}$ forms an orthonormal basis in $H_1$.

Let us write
\be
\psi_n=\frac{1}{\sigma_n}K\phi_n,\quad\sigma_n=\sqrt{\lambda_n},\quad n\in\Nm.
\ee
These $\sigma_n$ are called singular values:
\be
\sigma_1\ge\sigma_2\ge\cdots.
\ee
Equation (\ref{KstaKeigen}) implies
\be
K^*\psi_n=\sigma_n\phi_n,\quad K\phi_n=\sigma_n\psi_n.
\ee
These $\{\psi_n\}$ are orthonormal:
\be
\langle\psi_n,\psi_m\rangle=
\frac{1}{\sigma_n\sigma_m}\langle K\phi_n,K\phi_m\rangle=
\frac{1}{\sigma_n\sigma_m}\langle K^*K\phi_n,\phi_m\rangle=
\frac{\sigma_n^2}{\sigma_n\sigma_m}\langle\phi_n,\phi_m\rangle=\delta_{nm}.
\ee
We note that in general $\{\psi_n\}$ is not a basis.

Let us consider the inverse problem of $g\to f$ such that
\be
Kf=g,\quad f\in H_1,\quad g\in H_2.
\ee
Since $\{\phi_n\}$ forms an orthonormal basis in $H_1$, we can write
\be
f=\sum_{n=1}^{\infty}\langle f,\phi_n\rangle\phi_n.
\ee
Hence,
\be
g=Kf=\sum_{n=1}^{\infty}\langle f,\phi_n\rangle K\phi_n=
\sum_{n=1}^{\infty}\sigma_n\langle f,\phi_n\rangle \psi_n.
\ee
The expression on the rightmost-hand side is said to be the singular value decomposition (SVD) of $K$ \cite{Engl-Hanke-Neubauer00}.

We note that
\be
\langle g,\psi_m\rangle=
\left\langle Kf,\frac{1}{\sigma_m}K\phi_m\right\rangle=
\left\langle f,\frac{1}{\sigma_m}K^*K\phi_m\right\rangle=
\sigma_m\langle f,\phi_m\rangle.
\ee
Hence,
\begin{equation}
f=\sum_{n=1}^{\infty}\frac{1}{\sigma_n}\langle g,\psi_n\rangle\phi_n.
\label{sec4:fsol1}
\end{equation}
Thus, the inverse problem to obtain $f$ was solved ($f=K^{-1}g$) at least formally.

Indeed, the solution (\ref{sec4:fsol1}) requires some care for ill-posed inverse problems. Below, we will see this situation by using an inverse problem for the heat equation.

\subsection{SVD for the heat equation}
\hfill\vskip1mm

In general, $K^{-1}$ is not continuous and $f^{\varepsilon}$ from noisy measurements is not necessarily close to $f$ even if measurement error $\delta$ is small.

Let us consider the inverse problem of finding the initial temperature distribution $f$ for the heat equation (\ref{sec3:heq2}). Let us set $H_1=H_2=L^2(0,\pi)$. We observe $g(x)=u(x,T)$. The linear operator $K$ is given by (see (\ref{sec3:heqsol})):
\begin{equation}
(Kf)(x)=\int_0^{\pi}
\left(\sum_{n=1}^{\infty}\phi_n(x)\phi_n(y)e^{-n^2T}\right)f(y)\,dy,\quad f\in L^2(0,\pi).
\label{sec4:opK}
\end{equation}

We note that $K=K^*$. Using the expression (\ref{sec4:opK}), we have
\be
\begin{aligned}
&
(K^*K\phi_n)(x)
\\
&=
\int_0^{\pi}
\left(\sum_{m=1}^{\infty}\phi_m(x)\phi_m(y)e^{-m^2T}\right)\left[\int_0^{\pi}\left(
\sum_{l=1}^{\infty}\phi_l(y)\phi_l(z)e^{-l^2T}\right)\phi_n(z)\,dz
\right]\,dy
\\
&=
\int_0^{\pi}
\left(\sum_{m=1}^{\infty}\phi_m(x)\phi_m(y)e^{-m^2T}\right)\left[\phi_n(y)e^{-n^2T}\right]\,dy
\\
&=
e^{-2n^2T}\phi_n(x).
\end{aligned}
\ee
Hence, eigenvalues and singular values are obtained as
\be
\lambda_n=e^{-2n^2T},\quad\sigma_n=e^{-n^2T}.
\ee
Moreover,
\be
\phi_n=\psi_n=\sqrt{\frac{2}{\pi}}\sin{nx}.
\ee
Using (\ref{sec4:fsol1}),
\begin{equation}
f(x)=\frac{2}{\pi}\sum_{n=1}^{\infty}e^{n^2T}\left(\int_0^{\pi}g(y)\sin{ny}\,dy\right)\sin{nx}.
\label{sec4:fKinvg}
\end{equation}

Measurement error is unavoidable. Although it is a bit artificial, let us assume the measurement error as
\be
g^{\delta}=g+\delta,
\ee
where $\delta=\sqrt{\sigma_m}\psi_m$ for some $m$. That is,
\be
\begin{aligned}
g^{\delta}(x)-g(x)
&=
\sqrt{\sigma_m}\psi_m(x)
\\
&=
\sqrt{\frac{2}{\pi}}e^{-\frac{1}{2}m^2T}\sin{mx}.
\end{aligned}
\ee
Since $\lim_{m\to\infty}\|\sqrt{\sigma_m}\psi_m\|_{H_2}=\lim_{m\to\infty}\sqrt{\sigma_m}=0$, error becomes small for large $m$. Since $|\sin{mx}|\le1$, $g^{\delta}\to g$ uniformly as $m\to\infty$ ($x\in[0,\pi]$).

Using (\ref{sec4:fsol1}), we have for the solution $f^{\delta}$ to the inverse problem with $g^{\delta}$,
\be
\begin{aligned}
f^{\delta}(x)-f(x)
&=
\frac{1}{\sqrt{\sigma_m}}\phi_m(x)
\\
&=
\sqrt{\frac{2}{\pi}}e^{\frac{1}{2}m^2T}\sin{mx}.
\end{aligned}
\ee
Then we obtain
\ba
\lim_{m\to\infty}\|f^{\delta}-f\|_{L^2(0,\pi)}
&=\lim_{m\to\infty}\frac{1}{\sqrt{\sigma_m}}=\infty,
\\
\lim_{m\to\infty}\max_{0\le x\le\pi}|f^{\delta}(x)-f(x)|
&=
\lim_{m\to\infty}\sqrt{\frac{2}{\pi}}e^{\frac{1}{2}m^2T}=\infty.
\ea
We see that a small measurement error $g^{\delta}-g$ causes a large error $f^{\delta}-f$ for the solution of the inverse problem. Thus, the error of the solution gets large even when the measurement error is very small. 

The solution $f$ should exist because $g^{\delta}$ is $Kf$ plus some measurement error. However, the the solution $f$ cannot be obtained. To remedy this situation, we approximate the unstable inverse operator $K^{-1}$ to another stable operator $R$ and obtain an approximate solution $f^{\delta}$ instead of $f$. The remedy of stabilizing $K^{-1}$ is called regularization.

\section{Truncated SVD}
\label{sec5}

\subsection{Regularization}
\hfill\vskip1mm

We have found that the error $f^{\delta}-f$ becomes large for large $m$ (i.e., when high frequencies are taken into account). Here we consider regularization by ignoring high frequencies. That is, we make a low-pass filter.

Let $\alpha>0$ be a constant. Let us consider the inverse problem of $Kf=g$. We define $R_{\alpha}\colon H_2\to H_1$ as
\begin{equation}
R_{\alpha}g=\sum_{n=1\atop \sigma_n^2\ge\alpha}^{\infty}\frac{1}{\sigma_n}\langle g,\psi_n\rangle\phi_n.
\label{sec5:ftrunc}
\end{equation}
The definition (\ref{sec5:ftrunc}) can be compared to (\ref{sec4:fsol1}). The remedy which was done in (\ref{sec5:ftrunc}) is called the regularization by truncated SVD. The parameter $\alpha>0$ is said to be the regularization parameter.

For arbitrary $\xi\in H_1$,
\begin{equation}
\|R_{\alpha}K\xi-\xi\|_{H_1}\to0,\quad\alpha\to0.
\label{sec5:hzero}
\end{equation}
That is, the role of the operator $R_{\alpha}$ is similar to $K^{-1}$. Suppose $\|g^{\delta}-g\|_{H_2}\le\delta$, where the constant $\delta>0$ is assumed to be known. By choosing a suitable $\alpha>0$, we wish to obtain an approximate solution
\be
f^{\delta,\alpha}=R_{\alpha}g^{\delta}.
\ee
It can be expected that $\|f^{\delta,\alpha}-f\|_{H_1}$ is small for small $\delta>0$.

We observe that
\be
\begin{aligned}
\|f^{\delta,\alpha}-R_{\alpha}g\|_{H_1}
&=
\|R_{\alpha}g^{\delta}-R_{\alpha}g\|_{H_1}=
\left\|\sum_{n=1\atop\sigma_n^2\ge\alpha}^{\infty}\frac{1}{\sigma_n}\langle g^{\delta}-g,\psi_n\rangle\phi_n\right\|_{H_1}
\\
&\le
\frac{1}{\sqrt{\alpha}}\left\|\sum_{n=1\atop\sigma_n^2\ge\alpha}^{\infty}\langle g^{\delta}-g,\psi_n\rangle\phi_n\right\|_{H_1}
=
\frac{1}{\sqrt{\alpha}}\left(\sum_{n=1\atop\sigma_n^2\ge\alpha}^{\infty}\left|\langle g^{\delta}-g,\psi_n\rangle\right|^2\right)^{1/2}
\\
&\le
\frac{1}{\sqrt{\alpha}}\|g^{\delta}-g\|_{H_2}
\le
\frac{\delta}{\sqrt{\alpha}},
\end{aligned}
\ee
where we used the Bessel inequality. That is, $\|f^{\delta,\alpha}-R_{\alpha}g\|_{H_1}\le\delta/\sqrt{\alpha}$. We note that $\|f^{\delta,\alpha}-f\|_{H_1}\le\|f^{\delta,\alpha}-R_{\alpha}g\|_{H_1}+\|R_{\alpha}g-f\|_{H_1}$ due to the triangle inequality. Hence,
\begin{equation}
\|f^{\delta,\alpha}-f\|_{H_1}\le
\frac{\delta}{\sqrt{\alpha}}+\|R_{\alpha}Kf-f\|_{H_1}.
\label{sec5:finequal}
\end{equation}
For an arbitrary fixed $\delta>0$, the first term on the right-hand side becomes large but the second term on the right-hand side tends to $0$ as $\alpha\to0$. To make the left-hand side small, a suitably small $\alpha$ needs to be chosen depending on $\delta$.

\begin{thm}
If $\alpha(\delta)>0$ is chosen such that $\lim_{\delta\to0}\alpha(\delta)=0$ and $\lim_{\delta\to0}\delta/\sqrt{\alpha(\delta)}=0$, then
\be
\lim_{\delta\to0}\|R_{\alpha(\delta)}g^{\delta}-f\|_{H_1}=0.
\ee
\end{thm}

\begin{proof}
From (\ref{sec5:hzero}), we see that the second term on the right-hand side of (\ref{sec5:finequal}) goes to $0$ as $\delta\to0$. That is, both the first and second terms on the right-hand side of (\ref{sec5:finequal}) go to $0$ as $\delta\to0$.
\end{proof}

\subsection{Inverse problem for the heat equation with truncated SVD}
\hfill\vskip1mm

We assume that $f(x)=x(\pi-x)$ at $t=0$. Suppose we observe $g(x)=u(x,T)$ (we set $T=1$). Let us obtain $f^{\alpha}$ by solving the inverse problem of $Kf=g$ with the truncated SVD.

According to (\ref{sec4:opK}), we have
\be
\begin{aligned}
g(x)&=
u(x,T)=(Kf)(x)
\\
&=\int_0^{\pi}
\left(\sum_{n=1}^{\infty}\phi_n(x)\phi_n(y)\sigma_n\right)y(\pi-y)\,dy
\\
&=
4\sqrt{\frac{2}{\pi}}\sum_{m=1}^{\infty}\frac{\sigma_{2m-1}}{(2m-1)^3}\phi_{2m-1}(x)
\\
&=
\frac{8}{\pi}\sum_{m=1}^{\infty}\frac{e^{-(2m-1)^2}}{(2m-1)^3}\sin{[(2m-1)x]},
\end{aligned}
\ee
where $\sigma_n=\exp(-n^2)$. Let us write $f^{\alpha}=R_{\alpha}g$. We have (see (\ref{sec5:ftrunc}))
\be
f^{\alpha}(x)=\frac{8}{\pi}\sum_{m=1}^M\frac{\sin{(2m-1)x}}{(2m-1)^3},
\ee
where $M$ is the largest integer such that $\exp(-(2M-1)^2)\ge\sqrt{\alpha}$.\footnote{
We obtain
\be
M=\left\lfloor\frac{1}{2}\left(\sqrt{\frac{-1}{2}\ln{\alpha}}+1\right)\right\rfloor,
\ee
where $\lfloor x\rfloor=[x]=\max\{n\in\Zm;\;n\le x\}$.
}

Numerical results are shown in Fig.~\ref{sec5:fig01}. Since no noise is assumed for $g$, we see that smaller $\alpha>0$ gives a better reconstruction. In this case, the same model is used both for the forward and inverse problems. Such a situation is called the inverse crime.

\begin{figure}[ht]
\begin{center}
\includegraphics[width=0.4\textwidth]{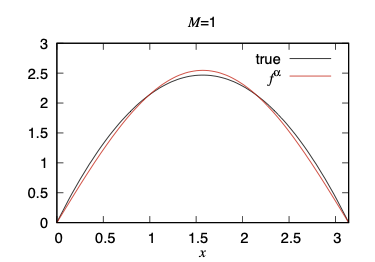}
\includegraphics[width=0.4\textwidth]{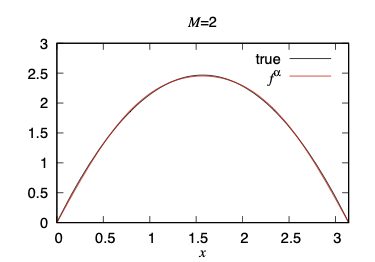}
\end{center}
\caption{
In the case that no measurement noise is assumed. The regularization parameter $\alpha$ is set to $0.1$ on the left panel and $10^{-8}$ on the right panel.
}
\label{sec5:fig01}
\end{figure}

To avoid inverse crime, we add $3\%$ noise:
\be
g^{\delta}(x)=g(x)(1+X),\quad X\sim \mathcal{N}(0,0.03^2),
\ee
where $\mathcal{N}({\rm MEAN},{\rm VAR})$ is the Gaussian distribution with mean ${\rm MEAN}$ and variance ${\rm VAR}$, and a random variable $X$ is normally distributed with mean $0$ and standard deviation $0.03$. Now, the data contains noise as shown in Fig.~\ref{sec5:fig02}.

\begin{figure}[ht]
\begin{center}
\includegraphics[width=0.4\textwidth]{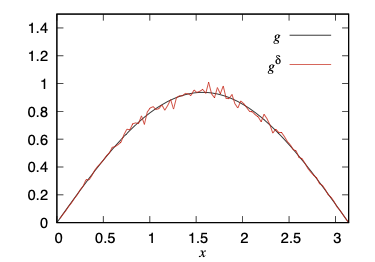}
\end{center}
\caption{
The forward data with $3\%$ Gaussian noise.
}
\label{sec5:fig02}
\end{figure}

We have
\be
f^{\delta,\alpha}=
\sum_{n=1\atop\sigma_n^2\ge\alpha}^N\frac{1}{\sigma_n}\langle g^{\delta},\psi_n\rangle\phi_n=
\frac{2}{\pi}\sum_{n=1\atop\sigma_n^2\ge\alpha}^Ne^{n^2}\left(
\int_0^{\pi}g^{\delta}(y)\sin{ny}\,dy
\right)\sin{nx},
\ee
where the integral on the right-hand side can be numerically evaluated, for example, by the trapezoidal rule. The first a few singular values are obtained as
\begin{equation}
\sigma_1^2=0.14,\quad
\sigma_2^2=3.4\times10^{-4},\quad\sigma_3^2=1.5\times10^{-8},\quad
\sigma_4^2=1.3\times10^{-14}.
\label{sec5:singularvals}
\end{equation}

Reconstructed results are shown in Figs.~\ref{sec5:fig03} and \ref{sec5:fig04}. This time, the reconstruction is nonsense if $\alpha$ is too small ($\alpha=10^{-8}$).

\begin{figure}[ht]
\begin{center}
\includegraphics[width=0.4\textwidth]{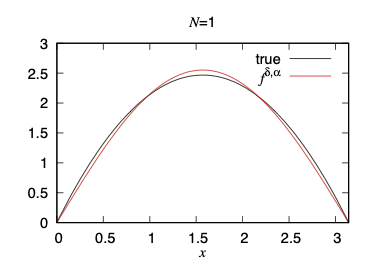}
\includegraphics[width=0.4\textwidth]{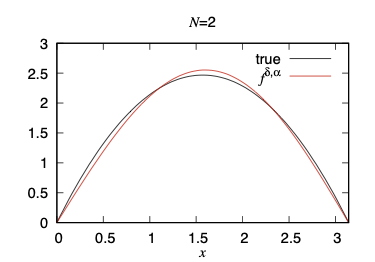}
\end{center}
\caption{
In the case that $3\%$ measurement noise is assumed. The regularization parameter $\alpha$ is set to $0.1$ on the left panel and $10^{-4}$ on the right panel.
}
\label{sec5:fig03}
\end{figure}

\begin{figure}[ht]
\begin{center}
\includegraphics[width=0.4\textwidth]{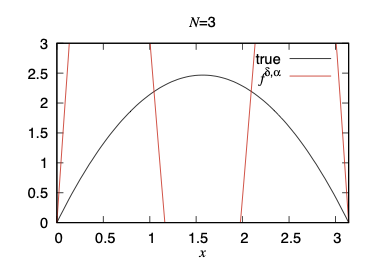}
\end{center}
\caption{
In the case that $3\%$ measurement noise is assumed. The regularization parameter $\alpha$ is set to $10^{-8}$.
}
\label{sec5:fig04}
\end{figure}

\section{Tikhonov regularization}
\label{sec6}

\subsection{Regularized Moore-Penrose pseudoinverse}
\hfill\vskip1mm

Instead of removing small singular values, we here modify small singular values, so all singular values do not exceed a certain threshold.

We will make use of SVD in Sec.~\ref{sec4:mat}. Let $E\in\Rm^{n\times n}$ be the identity. We consider the inverse problem of $Ax=b$, where $A\in\Rm^{m\times n}$, $b\in\Rm^m$, and $x\in\Rm^n$. Let us suppose $m\ge n$ (a similar calculation is possible for $m<n$).

\begin{thm}
The matrix $\alpha^2E+A^TA$ is invertible ($\alpha>0$).
\end{thm}

\begin{proof}
Let us define $T\in\Rm^{n\times n}$ as
\be
T=\left(\cdots \xi_j\cdots\right),
\ee
where
\be
\xi_j=\sum_{k=1}^n\frac{\va_k^{(j)}}{\alpha+\mu_k^2}\va_k,
\quad j=1,\dots,n,
\ee
where $\va_k^{(j)}$ is the $j$th element of the vector $\va_k$. We have
\be
\left(\alpha E+A^TA\right)T
=\alpha T+A^TAT=(\cdots\alpha\xi_j\cdots)+(\cdots A^TA \xi_j\cdots).
\ee
Since
\be
\left\{A^TA\xi_j\right\}_i=\left\{VD^TDV^T\xi_j\right\}_i
=\sum_{k=1}^n\frac{\mu_k^2}{\alpha+\mu_k^2}\va_k^{(i)}\va_k^{(j)}
,\quad i=1,\dots,n,
\ee
we obtain
\be
\left\{\alpha\xi_j+A^TA\xi_j\right\}_i=\sum_{k=1}^n\va_k^{(i)}\va_k^{(j)}.
\ee
Hence,
\ba
\left(\alpha E+A^TA\right)T=VV^T=E
\ea
Similarly,
\be
T\left(\alpha E+A^TA\right)=E.
\ee
This means,
\be
T=\left(\alpha E+A^TA\right)^{-1}.
\ee
\end{proof}

Let us give $R_{\alpha}$ as
\begin{equation}
R_{\alpha}=\left(\alpha E+A^TA\right)^{-1}A^T.
\label{sec6:RalTikh}
\end{equation}
This regularization with $\alpha E$ is called the Tikhonov regularization. We note that this $R_{\alpha}$ is a regularized Moore-Penrose pseudoinverse. Here, the Moore-Penrose pseudoinverse $A^+$ is given by $A^+=(A^TA)^{-1}A^T=V(D^TD)^{-1}D^TU^T$.

Thus, $x^{\delta,\alpha}$ can be computed as
\be
x^{\delta,\alpha}=R_{\alpha}b^{\delta}.
\ee

Let $x^{\alpha}$ be the solution of
\begin{equation}
\alpha x^{\alpha}+A^TAx^{\alpha}=A^Tb.
\label{cond1}
\end{equation}
We can write
\ba
x^{\alpha}
&=
(\alpha E+A^TA)^{-1}A^Tb=R_{\alpha}b
\\
&=
TVD^TU^Tb
\\
&=
\sum_{k=1}^n\frac{1}{\alpha+\mu_k^2}\va_k
\sum_{i=1}^n\sum_{j=1}^n\va_k^{(i)}\va_j^{(i)}\mu_j\psi_j^Tb
\\
&=
\sum_{k=1}^n\frac{\mu_k}{\alpha+\mu_k^2}(\psi_k\cdot b)\va_k.
\ea
We obtain
\begin{equation}
x^{\alpha}=
\sum_{k=1}^n\frac{1}{\mu_k}q(\alpha,\mu_k)(\psi_k\cdot b)\va_k,
\label{Salpha1}
\end{equation}
where
\be
q(\alpha,\mu)=\frac{\mu^2}{\alpha+\mu^2}.
\ee
The weight $q(\alpha\mu)$ is called the Tikhonov filter. We note that the Tikhonov filter is close to $1$ when $\alpha$ is small compared to $\mu^2$. The filter becomes much smaller than $1$ if $\mu^2$ is significantly smaller than $\alpha$. Thus, the instability due to small singular values can be suppressed. 

We note that $0<q(\alpha,\mu)<1$ and $q(\alpha,\mu)\le\mu/(2\sqrt{\alpha})$ since
\be
\frac{1}{2\sqrt{\alpha}}-\frac{\mu}{\alpha+\mu^2}=
\frac{(\sqrt{\alpha}-\mu)^2}{2\sqrt{\alpha}(\alpha+\mu^2)}\ge0.
\ee

\begin{prop}
Suppose $\|b^{\delta}-b\|_{\ell^2}\le\delta$ for $\delta>0$. Let $\mu_{n'}>0$ be the smallest nonzero singular value. The following inequality holds.
\begin{equation}
\|x^{\delta,\alpha}-x\|_{\ell^2}\le
\min\left(\frac{1}{2\sqrt{\alpha}},\frac{1}{\mu_{n'}}\right)\delta+
\|R_{\alpha}Ax-x\|_{\ell^2}.
\label{sec6:propeq1}
\end{equation}
\end{prop}

\begin{proof}
We have
\ba
&\|x^{\delta,\alpha}-R_{\alpha}b\|_{\ell^2}=
\|R_{\alpha}b^{\delta}-R_{\alpha}b\|_{\ell^2}=
\left\|\sum_{k=1}^n\frac{1}{\mu_k}q(\alpha,\mu_k)\left(\psi_k\cdot(b^{\delta}-b)\right)\va_k\right\|_{\ell^2}
\\
&\le
\min\left(\frac{1}{2\sqrt{\alpha}},\frac{1}{\mu_{n'}}\right)
\left\|\sum_{k=1}^n\left(\psi_k\cdot(b^{\delta}-b)\right)\va_k\right\|_{\ell^2}=
\min\left(\frac{1}{2\sqrt{\alpha}},\frac{1}{\mu_{n'}}\right)
\left(\sum_{k=1}^n\left|\psi_k\cdot(b^{\delta}-b)\right|^2\right)^{1/2}
\\
&\le
\min\left(\frac{1}{2\sqrt{\alpha}},\frac{1}{\mu_{n'}}\right)
\|b^{\delta}-b\|_{\ell^2}
\le
\min\left(\frac{1}{2\sqrt{\alpha}},\frac{1}{\mu_{n'}}\right)\delta.
\ea
Since $\|x^{\delta,\alpha}-x\|_{\ell^2}\le\|x^{\delta,\alpha}-R_{\alpha}b\|_{\ell^2}+\|R_{\alpha}b-x\|_{\ell^2}$, we obtain
\be
\|x^{\delta,\alpha}-x\|_{\ell^2}\le
\min\left(\frac{1}{2\sqrt{\alpha}},\frac{1}{\mu_{n'}}\right)\delta+
\|R_{\alpha}b-x\|_{\ell^2}.
\ee
The proof is complete.
\end{proof}

\begin{rmk}
The right-hand side of (\ref{sec6:propeq1}) implies that we can set $\alpha=0$ if $\delta/\mu_{n'}$ is sufficiently small. When $\delta$ is large or $\mu_{n'}$ is small, however, a better solution $x^{\delta,\alpha}$ is obtained with a finite $\alpha>0$.
\end{rmk}

\begin{lem}
$\|R_{\alpha}\|\le\frac{1}{2\sqrt{\alpha}}$.
\end{lem}

\begin{proof}
We note that $\|x\|_{\ell^2}^2=\sum_{j=1}^n|x\cdot\va_j|^2$ for any $x\in\Rm^n$. We have
\be
\|R_{\alpha}b\|_{\ell^2}^2=
\sum_{j=1}^n\frac{1}{\mu_j^2}q(\alpha,\mu_j)^2|b\cdot\psi_j|^2
\le\frac{1}{(2\sqrt{\alpha})^2}\sum_{j=1}^n|b\cdot\psi_j|^2
\le\frac{1}{(2\sqrt{\alpha})^2}\|b\|_{\ell^2}^2
\ee
for any $b$. Here, we used the arithmetic-geometric mean inequality.
\end{proof}

\begin{thm}
Suppose that the singular values are positive. For given $x\in\Rm^n$ ($x\neq0$) and $\varepsilon>0$, there exists $\alpha_0=\alpha_0(\varepsilon,\|x\|_{\ell^2})>0$ such that
\be
\|R_{\alpha_0}Ax-x\|_{\ell^2}^2<\varepsilon.
\ee
\end{thm}

\begin{proof}
Since
\be
\va_k^TR_{\alpha}Ax=
\frac{1}{\mu_k}q(\alpha,\mu_k)\psi_k^TAx=
q(\alpha,\mu_k)(x\cdot\va_k),
\ee
we have
\be
\|R_{\alpha}Ax-x\|_{\ell^2}^2=
\sum_{k=1}^n\left|\va_k^T(R_{\alpha}Ax-x)\right|^2
=\sum_{k=1}^n\left(q(\alpha,\mu_k)-1\right)^2|x\cdot\va_k|^2.
\ee
Since $\lim_{\alpha\to0}q(\alpha,\mu)=1$ for $\mu>0$, there exists $\alpha_0(\varepsilon,\|x\|_{\ell^2})$ such that
\be
\left(q(\alpha,\mu_k)-1\right)^2<\frac{\varepsilon}{\|x\|_{\ell^2}^2}
\ee
for all $k\in[1,n]$ and any $\alpha\in(0,\alpha_0]$. Hence,
\be
\|R_{\alpha}Ax-x\|_{\ell^2}^2<
\frac{\varepsilon}{\|x\|_{\ell^2}^2}\sum_{k=1}^n|x\cdot\va_k|^2=\varepsilon
\ee
for any $0<\alpha\le\alpha_0$.
\end{proof}

\subsection{Regularized operator}
\hfill\vskip1mm

Let us solve the inverse problem of $Kf=g$ with the Tikhonov regularization. We define $R_{\alpha}\colon H_2\to H_1$ ($\alpha>0$) as
\begin{equation}
R_{\alpha}g=
\sum_{n=1}^{\infty}\frac{\sigma_n}{\alpha+\sigma_n^2}\langle g,\psi_n\rangle\phi_n
=\sum_{n=1}^{\infty}q(\alpha,\sigma_n)\frac{1}{\sigma_n}\langle g,\psi_n\rangle\phi_n.
\label{sec6:fTikhonov}
\end{equation}
With the Tikhonov regularization, $R_{\alpha}$ approximates $K^{-1}$. An example of the Tikhonov filter is shown in Fig.~\ref{sec6:fig00}. For example, when $g^{\delta}-g=\sqrt{\sigma_m}\psi_m$, we have $f^{\varepsilon}-f=\frac{\sigma_m^{3/2}}{(\alpha+\sigma_m^2)}\phi_m$ and the error does not increase.

\begin{figure}[ht]
\begin{center}
\includegraphics[width=0.4\textwidth]{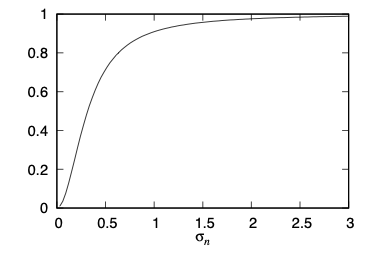}
\end{center}
\caption{
The Tikhonov filter $\frac{1}{1+\alpha e^{2n^2T}}=\frac{1}{1+\alpha/\sigma_n^2}$ for the case of the heat equation is shown for $\alpha=0.1$.
}
\label{sec6:fig00}
\end{figure}

\begin{thm}
If $\alpha(\delta)>0$ is chosen such that $\lim_{\delta\to0}\alpha(\delta)=0$ and $\lim_{\delta\to0}\delta/\alpha(\delta)=0$, then
\be
\lim_{\delta\to0}\|R_{\alpha(\delta)}g^{\delta}-f\|_{H_1}=0.
\ee
\end{thm}

\begin{proof}
We note that
\be
\|R_{\alpha}g^{\delta}-f\|_{H_1}\le
\|R_{\alpha}g^{\delta}-R_{\alpha}g\|_{H_1}+\|R_{\alpha}g-f\|_{H_1}.
\ee
For the first term on the right-hand side, we have
\be
\begin{aligned}
\|R_{\alpha}g^{\delta}-R_{\alpha}g\|_{H_1}^2
&=
\left\|\sum_{n=1}^{\infty}\frac{\sigma_n}{\sigma_n^2+\alpha}\langle g^{\delta}-g,\psi_n\rangle\phi_n\right\|_{H_1}^2
\\
&\le
\sum_{n=1}^{\infty}\left(\frac{\sigma_n}{\sigma_n^2+\alpha}\right)^2
\langle g^{\delta}-g,\psi_n\rangle^2
\le
\sum_{n=1}^{\infty}\frac{\sigma_1^2}{\alpha^2}\langle g^{\delta}-g,\psi_n\rangle^2
\\
&\le
\frac{\sigma_1^2}{\alpha^2}\|g^{\delta}-g\|_{H_1}^2
\le
\frac{\sigma_1^2\delta^2}{\alpha^2}.
\end{aligned}
\ee
Since $\lim_{\delta\to0}\delta/\alpha(\delta)=0$, $\lim_{\delta\to0}\|R_{\alpha}g^{\delta}-R_{\alpha}g\|_{H_1}^2=0$.

Next, we note that
\be
\|R_{\alpha}g-f\|_{H_1}=
\|R_{\alpha}Kf-f\|_{H_1}^2=\sum_{n=1}^{\infty}\left(\frac{\alpha}{\alpha+\sigma_n^2}\right)^2\left|\langle f,\phi_n\rangle\right|^2,
\ee
where we used
\be
R_{\alpha}Kf=\sum_{n=1}^{\infty}\frac{\sigma_n^2}{\sigma_n^2+\alpha}\langle f,\phi_n\rangle\phi_n.
\ee
Since $f\in H_1$, $\lim_{N\to\infty}\sum_{n=N}^{\infty}|\langle f,\phi_n\rangle|^2=0$. Hence for any $\alpha>0$, we can choose sufficiently large $N_0$ such that
\be
\sum_{n=N_0+1}^{\infty}|\langle f,\phi_n\rangle|^2<\frac{\alpha}{2}.
\ee
If $\delta$ is sufficiently small, we can choose $\alpha$ such that
\be
0<\alpha<\frac{\sigma_{N_0}^4}{2\|f\|_{H_1}}.
\ee
Then
\be
\begin{aligned}
\|R_{\alpha}Kf-f\|_{H_1}^2
&=
\left(\sum_{n=1}^{N_0}+\sum_{n=N_0+1}^{\infty}\right)\left(\frac{\alpha}{\alpha+\sigma_n^2}\right)^2|\langle f,\phi_n\rangle|^2
\\
&\le
\left(\frac{\alpha}{\sigma_{N_0}^2}\right)^2\sum_{n=1}^{N_0}|\langle f,\phi_n\rangle|^2+
\sum_{n=N_0+1}^{\infty}|\langle f,\phi_n\rangle|^2
\\
&\le
\frac{\alpha}{2}+\frac{\alpha}{2}=\alpha.
\end{aligned}
\ee
We note that $\alpha\to0$ as $\delta\to0$. Thus, the proof is complete.
\end{proof}

\begin{rmk}
If $\alpha$ is large (small), $N_0$ is small (large), and then $\sigma_{N_0}$ becomes large (small). If $\delta>0$ is small, we can choose small $\alpha>0$. 
\end{rmk}

\subsection{Inverse problem for the heat equation with the Tikhonov regularization}
\hfill\vskip1mm

We assume that $f(x)=x(\pi-x)$ at $t=0$. Suppose we observe $g(x)=u(x,T)$ (we set $T=1$). Let us obtain $f^{\alpha}$ by solving the inverse problem of $Kf=g$ with the Tikhonov regularization.

By solving the forward problem we obtain $g(x)$ as
\be
\begin{aligned}
g(x)&=
u(x,T)=(Kf)(x)
\\
&=
4\sqrt{\frac{2}{\pi}}\sum_{m=1}^{\infty}\frac{\sigma_{2m-1}}{(2m-1)^3}\phi_{2m-1}(x).
\end{aligned}
\ee
Let us write $f^{\alpha}=R_{\alpha}g$ (see (\ref{sec6:fTikhonov})). We have
\be
f^{\alpha}(x)=
\frac{8}{\pi}\sum_{m=1}^{\infty}\frac{1}{1+\alpha e^{2(2m-1)^2}}\frac{\sin{(2m-1)x}}{(2m-1)^3}.
\ee

Numerical results are shown in Figs.~\ref{sec6:fig01} and \ref{sec6:fig02}. Since no noise is assumed for $g$, the smaller $\alpha>0$ becomes, the better $f$ is reconstructed.

\begin{figure}[ht]
\begin{center}
\includegraphics[width=0.4\textwidth]{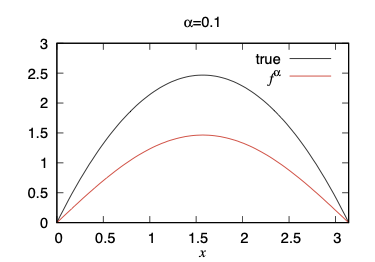}
\includegraphics[width=0.4\textwidth]{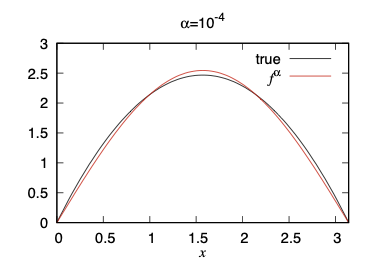}
\end{center}
\caption{
In the case that no measurement noise is assumed. The regularization parameter $\alpha$ is set to $0.1$ on the left panel and $10^{-4}$ on the right panel.
}
\label{sec6:fig01}
\end{figure}

\begin{figure}[ht]
\begin{center}
\includegraphics[width=0.4\textwidth]{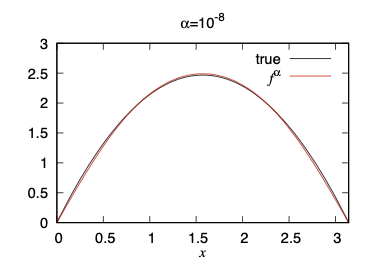}
\end{center}
\caption{
In the case that no measurement noise is assumed. The regularization parameter $\alpha$ is set to $10^{-8}$.
}
\label{sec6:fig02}
\end{figure}

Let us add $3\%$ noise. We have
\be
\begin{aligned}
f^{\delta,\alpha}
&=
\sum_{n=1}^{\infty}\left(\frac{\sigma_n^2}{\alpha+\sigma_n^2}\right)\frac{1}{\sigma_n}\langle g^{\delta},\psi_n\rangle\phi_n
\\
&=
\frac{2}{\pi}\sum_{n=1}^{\infty}\frac{e^{n^2}}{1+\alpha e^{2n^2}}\left(
\int_0^{\pi}g^{\delta}(y)\sin{ny}\,dy\right)\sin{nx}.
\end{aligned}
\ee
Recall singular values in (\ref{sec5:singularvals}). Reconstructed results are shown in Figs.~\ref{sec6:fig03} and \ref{sec6:fig04}. It is seen that the truncated SVD and Tikhonov regularization similarly work for a suitable regularization parameter $\alpha$.

\begin{figure}[ht]
\begin{center}
\includegraphics[width=0.4\textwidth]{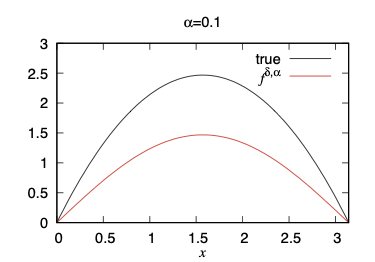}
\includegraphics[width=0.4\textwidth]{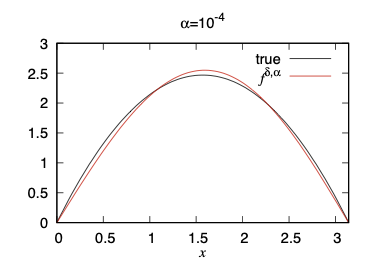}
\end{center}
\caption{
In the case that $3\%$ measurement noise is assumed. The regularization parameter $\alpha$ is set to $0.1$ on the left panel and $10^{-4}$ on the right panel.
}
\label{sec6:fig03}
\end{figure}

\begin{figure}[ht]
\begin{center}
\includegraphics[width=0.4\textwidth]{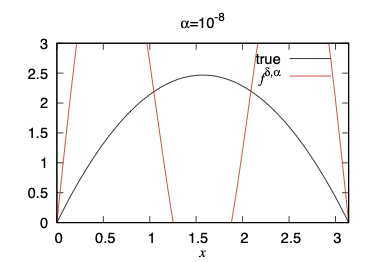}
\end{center}
\caption{
In the case that $3\%$ measurement noise is assumed. The regularization parameter $\alpha$ is set to $10^{-8}$.
}
\label{sec6:fig04}
\end{figure}

\subsection{Cost function}
\hfill\vskip1mm

To deepen the understanding of the Tikhonov regularization, we revisit the inverse problem of $Ax=b$, where $A\in\Rm^{m\times n}$, $b\in\Rm^m$, and $x\in\Rm^n$. We suppose $m\ge n$ (the case of $m<n$ can be considered similarly). We wish to obtain $x$ by solving $Ax=b$ for given $A$ and $b$. That is, we want to find the solution by minimizing the cost function $\|Ax-b\|_{\ell^2}$ as
\be
\argmin_{x\in\Rm^n}\|Ax-b\|_{\ell^2}.
\ee
We will solve this problem by suitably regularizing $A$. Although we here consider the matrix-vector equation, a similar discussion is also possible for the inverse problem $Kf=g$ with a linear operator $K$ ($f\in H_1$, $g\in H_2$).

In the actual situation, $b$ contains noise and $b^{\delta}$ is observed:
\be
\|b^{\delta}-b\|_{\ell^2}\le\delta,\quad\delta>0.
\ee
In general, $b^{\delta}$ does not belong to the range of $A$. Hence the solution,
\be
\argmin_{x\in\Rm^n}\|Ax-b^{\delta}\|_{\ell^2}
\ee
might be quite different from the true $x$.

To remedy this situation, we consider the cost function $\Phi_{\alpha}(x)$ as
\begin{equation}
\Phi_{\alpha}(x)=\|Ax-b^{\delta}\|_{\ell^2}^2+\alpha\|x\|_{\ell^2}^2
\label{sec6:functional}
\end{equation}
and find the minimizer $x^{\delta,\alpha}$ of
\be
x^{\delta,\alpha}=\argmin_{x\in\Rm^n}\Phi_{\alpha}(x)
\ee
where $\alpha>0$. The second term $\alpha\|x\|_{\ell^2}^2$ on the right-hand side was added to stabilize the solution (see also the Ridge regression \cite{McDonald09}). This term is called the Tikhonov regularization term.

\begin{thm}
For $\alpha>0$, there exists a unique $x^{\alpha}\in\Rm^n$ such that
\be
\|Ax^{\alpha}-b\|_{\ell^2}^2+\alpha\|x^{\alpha}\|_{\ell^2}^2=
\inf_{x\in\Rm^n}\left\{\|Ax-b\|_{\ell^2}^2+\alpha\|x\|_{\ell^2}^2\right\}.
\ee
The minimizer $x^{\alpha}$ is given by the solution in (\ref{Salpha1}), which is obtained with $R_{\alpha}$ in (\ref{sec6:RalTikh}).
\end{thm}

\begin{proof}
We observe that the condition (\ref{cond1}) is necessary and sufficient for $x^{\alpha}$ to minimize the Tikhonov functional (\ref{sec6:functional}) because for all $x\in\Rm^n$,
\ba
\|Ax-b\|_{\ell^2}^2+\alpha\|x\|_{\ell^2}^2
&=
\|Ax^{\alpha}-b\|_{\ell^2}^2+\alpha\|x^{\alpha}\|_{\ell^2}^2
+2(x-x^{\alpha})\cdot\left(\alpha x^{\alpha}+A^T(Ax^{\alpha}-b)\right)
\\
&+
\|A(x-x^{\alpha})\|_{\ell^2}^2+\alpha\|x-x^{\alpha}\|_{\ell^2}^2.
\ea
The above relation implies
\be
\alpha x^{\alpha}+A^T(Ax^{\alpha}-b)=
(\alpha E+A^TA)x^{\alpha}-A^Tb=0.
\ee
\end{proof}

\subsection{Remarks}
\hfill\vskip1mm

In Secs.~\ref{sec4} through \ref{sec6}, we have considered linear inverse problems. Inverse problems for which singular values decay slowly (e.g., as a power-law) are said to be mildly ill-posed. If singular values decay rapidly (e.g., exponentially), the inverse problem is called severely ill-posed.

When inverse problems are nonlinear, basically there are two approaches to numerically solve those problems: iterative schemes (Sec.~\ref{sec7}) and inverse series (\ref{sec9}). Below, we will explore these approaches.

\section{Iterative schemes}
\label{sec7}

\subsection{Landweber iteration}
\hfill\vskip1mm

The minimization problem of a cost function can be numerically solved with an iterative scheme. Although the Landweber iteration can be extended to nonlinear inverse problems, let us see how it works using the inverse problem $Kf=g$, which we have studied above. That is, $f\in H_1$, $g\in H_2$ ($H_1,H_2$ are some Hilbert spaces) and $K\colon H_1\to H_2$ is linear.

We formulate the problem with the relaxation factor $\omega>0$:
\be
\omega K^*Kf=\omega K^*g.
\ee
We have
\be
f-\omega K^*Kf=f-\omega K^*g.
\ee
Hence,
\begin{equation}
f=\left(I-\omega K^*K\right)f+\omega K^*g,
\label{sec7:fite0}
\end{equation}
where $I$ is the identity. This means that the solution $f$ of the inverse problem is a fixed point.

Since measurement $g$ contains noise, we use $g^{\delta}$ in (\ref{sec7:fite0}). To find the fixed point, we can consider the following algorithm:
\begin{itemize}
\item[Step 1.] Starts with an initial guess $f^0$.
\item[Step 2.] Repeat
\[
f^m=\left(I-\omega K^*K\right)f^{m-1}+\omega K^*g^{\delta}=
f^{m-1}-\omega K^*\left(Kf^{m-1}-g^{\delta}\right)
\]
for $m=1,2,\dots$.
\end{itemize}
The first two iterations are obtained as follows.
\be
\begin{aligned}
f^1&=
f^0-\omega K^*(Kf^0-g^{\delta}),
\\
f^2&=
f^1-\omega K^*(Kf^1-g^{\delta})
\\
&=
[f^0-\omega K^*(Kf^0-g^{\delta})]-
\omega K^*\left(K[f^0-\omega K^*(Kf^0-g^{\delta})]-g^{\delta}\right)
\\
&=
(I-\omega K^*K)^2f^0+
\omega\left(2I-\omega K^*K\right)K^*g^{\delta}.
\end{aligned}
\ee

We define operator $R_m\colon H_2\to H_1$ as
\be
R_m=\omega\sum_{k=0}^{m-1}\left(I-\omega K^*K\right)^kK^*
\ee
for $m=1,2,\dots$. After $m$ iterations, we obtain
\be
f^m=\left(I-\omega K^*K\right)^mf^0+R_mg^{\delta}.
\ee
For simplicity, let us set $f^0=0$.

\begin{thm}
Let $\omega>0$ be a constant and $m=m(\delta)\in\Nm$ such that
\be
\lim_{\delta\to0}m(\delta)=\infty,\quad
\lim_{\delta\to0}\delta^2m(\delta)=0,\quad
0<\omega<\frac{1}{\|K\|^2}.
\ee
Then for the operator $R_m$, which is defined above,
\be
\lim_{\delta\to0}\|R_mg^{\delta}-f\|_{H_1}=0.
\ee
\end{thm}

\begin{proof}
Since $K^*K\phi_1=\lambda\phi_1$, we have $\left\langle K^*K\phi_1,\phi_1\right\rangle=\left\langle\lambda_1\phi_1,\phi_1\right\rangle$. Thus,
\be
\|K\phi_1\|_{H_2}^2=\lambda\|\phi_1\|_{H_1}^2.
\ee
Noting that $\|K\phi_1\|_{H_2}\le\|K\|\|\phi_1\|_{H_1}$,\footnote{
$\|K\|=\sup_{\phi\in H_1,\;\phi\neq0}\frac{\|K\phi\|_{H_2}}{\|\phi\|_{H_1}}$.
} we find
\be
\lambda_1\le\|K\|^2.
\ee

By the triangle inequality, we have
\be
\|R_mg^{\delta}-f\|_{H_1}\le\|R_mg^{\delta}-R_mg\|_{H_1}+\|R_mg-f\|_{H_1}.
\ee
We obtain
\be
\begin{aligned}
&
(I-\omega K^*K)^kK^*g
=\sum_{j=1}^{\infty}\langle(I-\omega K^*K)^kK^*g,\phi_j\rangle\phi_j
\\
&=
\sum_{j=1}^{\infty}\langle K^*g,(I-\omega K^*K)^k\phi_j\rangle\phi_j
=\sum_{j=1}^{\infty}(1-\omega\sigma_j^2)^k\langle K^*g,\phi_j\rangle\phi_j,
\end{aligned}
\ee
where we used
\be
K^*K\phi_n=\sigma_n^2\phi_n.
\ee
Since $\langle K^*g,\phi_j\rangle=\langle g,K\phi_j\rangle=\sigma_j\langle g,\psi_j\rangle$, we have
\be
\begin{aligned}
R_mg
&=
\omega\sum_{k=0}^{m-1}(I-\omega K^*K)^kK^*g
=\omega\sum_{j=1}^{\infty}\sigma_j\langle g,\psi_j\rangle\sum_{k=0}^{m-1}(1-\omega\sigma_j^2)^k\phi_j
\\
&=
\omega\sum_{j=1}^{\infty}\sigma_j\langle g,\psi_j\rangle\frac{1-(1-\omega\sigma_j^2)^m}{1-(1-\omega\sigma_j^2)}\phi_j
=\sum_{j=1}^{\infty}\frac{1-(1-\omega\sigma_j^2)^m}{\sigma_j}\langle g,\psi_j\rangle\phi_j.
\end{aligned}
\ee

We note that $\omega\sigma_j^2\le\omega\lambda_1<1$ because $0<\omega<1/\|K\|^2$ and $\lambda_1\le\|K\|^2$. Hence,
\be
1-\left(1-\omega\sigma_j^2\right)^m\le1.
\ee
For each $m\in\Nm$,
\be
1-m\eta\le(1-\eta)^m,
\ee
where $\eta=\omega\sigma_j^2\in[0,1]$. From Bessel's inequality, we obtain
\be
\begin{aligned}
\|R_mg\|_{H_1}^2
&=
\sum_{j=1}^{\infty}\frac{\left(1-(1-\omega\sigma_j^2)^m\right)^2}{\sigma_j^2}|\langle g,\psi_j\rangle|^2
\le
\sum_{j=1}^{\infty}\frac{1-\left(1-\omega\sigma_j^2\right)^m}{\sigma_j^2}|\langle g,\psi_j\rangle|^2
\\
&\le
\sum_{j=1}^{\infty}\omega m|\langle g,\psi_j\rangle|^2
\le
\omega m\|g\|_{H_2}^2.
\end{aligned}
\ee
Thus we have
\be
\|R_m\|\le\sqrt{\omega m},\quad
\|R_mg^{\delta}-R_mg\|_{H_1}\le\sqrt{\omega m}\|g^{\delta}-g\|_{H_2}\le
\delta\sqrt{\omega m}.
\ee
The following estimate holds:
\be
\begin{aligned}
\|R_mg-f\|_{H_1}^2
&=
\left\|\sum_{j=1}^{\infty}\frac{1-(1-\omega\sigma_j^2)^m}{\sigma_j}\langle g,\psi_j\rangle\phi_j-\sum_{j=1}^{\infty}\frac{1}{\sigma_j}\langle g,\psi_j\rangle\phi_j\right\|_{H_1}^2
\\
&=
\left(\sum_{j=1}^N+\sum_{j=N+1}^{\infty}\right)\frac{(1-\omega\sigma_j^2)^{2m}}{\sigma_j^2}|\langle g,\psi_j\rangle|^2
\\
&\le
\sum_{j=1}^N\frac{(1-\omega\sigma_j^2)^{2m}}{\sigma_j^2}|\langle g,\psi_j\rangle|^2+
\sum_{j=N+1}^{\infty}\frac{1}{\sigma_j^2}|\langle g,\psi_j\rangle|^2.
\end{aligned}
\ee
The first term on the right-hand side of the above inequality goes to $0$ ($m\to\infty$) for an arbitrary fixed large $N$. Furthermore, the second term on the right-hand side of the above inequality can become arbitrarily small for sufficiently large $N$ because
\be
\|f\|_{H_1}^2=\sum_{j=1}^{\infty}\frac{1}{\sigma_j^2}|\langle g,\psi_j\rangle|^2.
\ee
Therefore,
\be
\lim_{m\to\infty}\|R_mg-f\|_{H_1}=0.
\ee
The proof is complete using $\|R_mg^{\delta}-R_mg\|_{H_1}\le\delta\sqrt{\omega m}$.
\end{proof}

We found that an approximate solution by the Landweber algorithm converges if suitably small $\omega$ is taken ($0<\omega<1/\|K\|^2$):
\be
\lim_{\delta\to0}\|R_mg^{\delta}-f\|_{H_1}=0.
\ee
We should choose large $m$ ($\lim_{\delta\to0}m(\delta)=\infty$) but $m$ should not be too large ($\lim_{\delta\to0}\delta^2m(\delta)=0$). In the Landweber iterative scheme, the iteration number takes the part of the regularization term and $1/m$ can be regarded as the regularization parameter.

The Landweber algorithm repeats iterations to minimize the cost function below:
\be
\Phi(f)=\frac{1}{2}\|Kf-g^{\delta}\|_{H_2}^2,\quad f\in H_1.
\ee
For any $\eta\in H_1$,
\be
\begin{aligned}
\Phi(f+\eta)-\Phi(f)
&=
\frac{1}{2}\langle Kf-g^{\delta}+K\eta,Kf-g^{\delta}+K\eta\rangle-\Phi(f)
\\
&=
\langle Kf-g^{\delta},K\eta\rangle+\frac{1}{2}\|K\eta\|_{H_2}^2.
\end{aligned}
\ee
Hence,
\be
|\Phi(f+\eta)-\Phi(f)-\nabla\Phi(f)\eta|\le
\frac{1}{2}\|K\|^2\|\eta\|_{H_1}^2,
\ee
where
\be
\nabla\Phi(f)\eta=\langle K^*(Kf-g^{\delta}),\eta\rangle.
\ee
Here, $\nabla\Phi(f)\colon H_1\to\Rm$ is a linear operator from $\eta\in H_1$ to $\langle K^*(Kf-g^{\delta}),\eta\rangle$, and is called the Fr\'{e}chet derivative of $\Phi(f)$ with respect to $f$:
\be
\lim_{\|\eta\|_{H_1}\to0}\frac{|\Phi(f+\eta)-\Phi(f)-\nabla\Phi(f)\eta|}{\|\eta\|_{H_1}}=0.
\ee

\begin{thm}
If $\Phi(f)$ takes its minimum at $f_*$, then
\be
\nabla\Phi(f_*)\eta=0
\ee
for all $\eta\in H_1$. This means $K^*(Kf_*-g^{\delta})=0$.
\end{thm}

\begin{proof}
For any $\eta\in H_1$ and $t\in\Rm$,
\be
\begin{aligned}
\Phi(f_*+t\eta)-\Phi(f_*)
&=
t\langle K^*(Kf_*-g^{\delta}),\eta\rangle+\frac{t^2}{2}\|K\eta\|_{H_2}^2
\\
&=
t\left(\langle K^*(Kf_*-g^{\delta}),\eta\rangle+\frac{t}{2}\|K\eta\|_{H_2}^2\right)
\ge0.
\end{aligned}
\ee
Since the inequality holds for any $t$, we have $\langle K^*(Kf_*-g^{\delta}),\eta\rangle=0$. Therefore,
\be
K^*(Kf_*-g^{\delta})=0.
\ee
\end{proof}

\subsection{Conjugate gradient method}
\hfill\vskip1mm

Even when the partial differential equation of interest is linear, its inverse problems are quite often nonlinear. In this section, the operator $K$ is not necessarily linear but can be nonlinear. Here, we develop the conjugate gradient method and solve the inverse problem of the contaminated water in a tank which we considered in Sec.~\ref{sec1}.

Let us introduce vectors $\bv{f},\bv{g}$ as
\be
\bv{f}=
\begin{pmatrix}x\\y\end{pmatrix},\quad
\bv{g}=\begin{pmatrix}A_1\\A_2\end{pmatrix},
\ee
where $A_1=x\left(1-e^{-y}\right)$, $A_2=x\left(1-e^{-2y}\right)$. Let us consider the inverse problem $\bv{g}\to\bf{f}$ of
\be
K{\bf f}={\bf g},
\ee
where the operator $K$ is introduced as
\be
K\bv{f}=\begin{pmatrix}x\left(1-e^{-y}\right)\\x\left(1-e^{-2y}\right)\end{pmatrix}.
\ee
As we assumed in Sec.~\ref{sec1}, the forward data contains noise:
\be
\bv{g}^{\delta}=
\begin{pmatrix}A_1^{\rm obs}\\A_2^{\rm obs}\end{pmatrix}=
\begin{pmatrix}
0.0791\\0.158
\end{pmatrix}.
\ee
We will solve the inverse problem by minimizing the cost function,
\be
\Phi({\bf f})=\frac{1}{2}\|K{\bf f}-{\bf g}^{\delta}\|_{\ell^2}^2+\alpha\|{\bf f}\|_{\ell^2}^2,
\ee
where the second term on the right-hand side is the Tikhonov regularization term.

Starting with an initial vector $\bv{f}^0$, the conjugate gradient method calculate recursively $\bv{f}^1,\bv{f}^2,\dots$ as \cite{Neculai20}
\be
\left\{\begin{aligned}
{\bf f}^{k+1}&={\bf f}^k+\ell^k{\bf d}^k
\\
{\bf d}^k&=
-\nabla\Phi({\bf f}^k)+\beta^k{\bf d}^{k-1}
\end{aligned}\right.
\ee
for $k=0,1,2,\dots$. Then the solution $\bv{f}_*$ is obtained as
\be
{\bf f}_*={\bf f}^0+\ell^0{\bf d}^0+\ell^1{\bf d}^1+\cdots.
\ee

If $K$ is linear, the inverse problem is stated in matrix-vector form with the matrix $\underline{K}$. In this case,
\be
\begin{aligned}
\Phi(\bv{f})
&=
\frac{1}{2}\|\underline{K}\bv{f}-\bv{g}^{\delta}\|_{\ell^2}^2+\alpha\|\bv{f}\|_{\ell^2}^2
\\
&=
\frac{1}{2}\left(\underline{K}\bv{f}-\bv{g}^{\delta}\right)^T\left(\underline{K}\bv{f}-\bv{g}^{\delta}\right)+\alpha\bv{f}^T\bv{f}
\\
&=
\frac{1}{2}\bv{f}^T\left(\underline{K}^T\underline{K}+2\alpha\underline{I}\right)\bv{f}-(\bv{g}^{\delta})^T\underline{K}\bv{f}
+\frac{1}{2}(\bv{g}^{\delta})^T\bv{g}^{\delta}
\\
&=
\frac{1}{2}\bv{f}^T\underline{Q}\bv{f}-\bv{b}^T\bv{f}
+\frac{1}{2}\|\bv{g}^{\delta}\|_{\ell^2}^2,
\end{aligned}
\ee
where
\be
\underline{Q}=\underline{K}^T\underline{K}+2\alpha\underline{I},
\quad
\bv{b}=\underline{K}^T\bv{g}^{\delta}.
\ee
We obtain
\be
\nabla\Phi(\bv{f})=\underline{Q}\bv{f}-\bv{b}.
\ee
Hence,
\be
\nabla\Phi(\bv{f}^k)=\underline{Q}\bv{f}^k-\bv{b}=:\bv{r}^k.
\ee

For a linear operator $K$, we have
\be
\left\{\begin{aligned}
\bv{f}^{k+1}&=\bv{f}^k+\ell^k\bv{d}^k
\\
\bv{d}^k&=
-\bv{r}^k+\beta^k\bv{d}^{k-1},
\end{aligned}\right.
\ee
where
\be
\beta^k=\frac{(\bv{r}^k)^T\underline{Q}\bv{d}^{k-1}}{(\bv{d}^{k-1})^T\underline{Q}\bv{d}^{k-1}}.
\ee
We obtain
\be
\begin{aligned}
(\bv{d}^k)^TQ\bv{d}^{k-1}
&=
\left(-\bv{r}^k+\beta^k\bv{d}^{k-1}\right)^T\underline{Q}\bv{d}^{k-1}
\\
&=
\left(-\bv{r}^k+\frac{(\bv{r}^k)^T\underline{Q}\bv{d}^{k-1}}{(\bv{d}^{k-1})^T\underline{Q}\bv{d}^{k-1}}\bv{d}^{k-1}\right)^T\underline{Q}\bv{d}^{k-1}
\\
&=
-(\bv{r}^k)^T\underline{Q}\bv{d}^{k-1}+
(\bv{r}^k)^T\underline{Q}\bv{d}^{k-1}=0.
\end{aligned}
\ee
Since $(\bv{d}^k)^TQ\bv{d}^{k-1}=0$, $\bv{d}^{k-1}$ and $\bv{d}^k$ are conjugate.

Now, $\ell^k$ can be determined by line search. Furthermore we have (the Polak-Ribiere-Polyak conjugate gradient \cite{Polak-Ribiere69,Polyak69}
\be
\beta^k=\frac{
\left(\nabla\Phi(\bv{f}^k)-\nabla\Phi(\bv{f}^{k-1})\right)\cdot\nabla\Phi(\bv{f}^k)}
{|\nabla\Phi(\bv{f}^{k-1})|^2}.
\ee
More choices of $\beta^k$ have been proposed. One is \cite{Fletcher-Reeves64}
\be
\beta^k=\frac{|\nabla\Phi(\bv{f}^k)|^2}{|\nabla\Phi(\bv{f}^{k-1})|^2}.
\ee
Another choice is \cite{Hestensen-Stiefel52}
\be
\beta^k=\frac{
\left(\nabla\Phi(\bv{f}^k)-\nabla\Phi(\bv{f}^{k-1})\right)\cdot\nabla\Phi(\bv{f}^k)}{
\left(\nabla\Phi(\bv{f}^k)-\nabla\Phi(\bv{f}^{k-1})\right)\cdot\bv{d}^{k-1}}.
\ee

The algorithm of the conjugate gradient method can be summarized as follows.
\begin{itemize}
\item[Step 1.] Calculate $\nabla\Phi(\bv{f}^0)=(\pp_x\Phi(\bv{f}^0)\;\;\pp_y\Phi(\bv{f}^0))^T$ for a chosen initial guess $\bv{f}^0$.
\item[Step 2.] Set the search direction $\bv{d}^0=-\nabla\Phi(\bv{f}^0)$.
\item[Step 3.] Set $k=1$.
\item[Step 4.] By line search, find $\ell^{k-1}$ that minimizes $\Phi(\bv{f}^{k-1}+\ell^{k-1}\bv{d}^{k-1})$.
\item[Step 5.] $\bv{f}^k=\bv{f}^{k-1}+\ell^{k-1}\bv{d}^{k-1}$.
\item[Step 6.] $\nabla\Phi(\bv{f}^k)=(\pp_x\Phi(\bv{f}^k)\;\;\pp_y\Phi(\bv{f}^k))^T$.
\item[Step 7.] Set the next search direction as $\bv{d}^k=-\nabla\Phi(\bv{f}^k)+\beta^k\bv{d}^{k-1}$.
\item[Step 8.] If $k=k_{\rm max}$, stop the iteration. Otherwise put $k=k+1$ and return to Step 4. 
\end{itemize}

For Step 4, different line searches have been proposed. Here we use the backtracking line search, which is given as follows.

\begin{itemize}
\item[Step 1.] Set $\ell_0^{k-1}=\ell_0$ for an initial guess $\ell_0$ (for example, $\ell_0=1$).
\item[Step 2.] Put $j=1$.
\item[Step 3.] If the inequality $\Phi(\bv{f}^{k-1}+\ell^{k-1}_{j-1}\bv{d}^{k-1})<\Phi(\bv{f}^{k-1})+\gamma\ell^{k-1}_{j-1}\nabla\Phi(\bv{f}^{k-1})\cdot\bv{d}^{k-1}$ holds, then set $\ell^{k-1}=\ell_{j-1}^{k-1}$ and finishes the line search. Otherwise set $\ell_j^{k-1}=\kappa\ell_{j-1}^{k-1}$.
\item[Step 4.] If $j=j_{\rm max}$, then set $\ell^{k-1}=\ell_{j_{\rm max}}^{k-1}$ and finish the computation. Otherwise put $j=j+1$ and return to Step 3.
\end{itemize}

In the above line search, constants $\gamma,\kappa$ are fixed (for example, $\gamma=0.1$, $\kappa=1/2$). The inequality at Step 3 is called the Armijo condition. We can naively look at the inequality
\be
\Phi(\bv{f}^{k-1}+\ell^{k-1}_{j-1}\bv{d}^{k-1})<\Phi(\bv{f}^{k-1}).
\ee
At Step 3, the term $\gamma\ell^{k-1}_{j-1}\nabla\Phi(\bv{f}^{k-1})\cdot\bv{d}^{k-1}$ was added to prevent large changes in $\bv{f}$.

The line search is an inevitable step for the conjugate gradient method but its role is supportive. Hence it is not necessary to obtain the exact $\ell^{k-1}$ which minimizes the cost function. The line search for a rough estimate is said to be the inexact line search.

Figure \ref{sec7:fig01} shows the convergence of the iteration in the conjugate gradient method for the water leakage problem in Sec.~\ref{sec1}. The regularization parameter was set to $\alpha=8\times10^{-9}$. Starting with initial values $x^0=3$, $y^0=0.01$, the calculation converges to the true values of $x=4$, $y=0.02$. The regularization parameter $8\times10^{-9}$ was chosen by trial and error. Although the $L$-curve method is know, there is no general theory of determining regularization parameters. The situation is described in Fig.~\ref{sec7:fig02}. Moreover in Fig.~\ref{sec7:fig03}, the dependence of initial guesses is investigated.

\begin{figure}[ht]
\begin{center}
\includegraphics[width=0.4\textwidth]{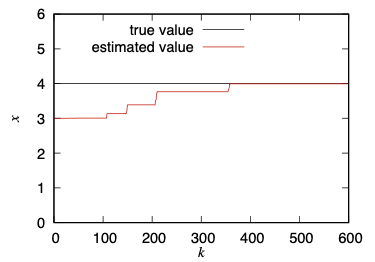}
\includegraphics[width=0.4\textwidth]{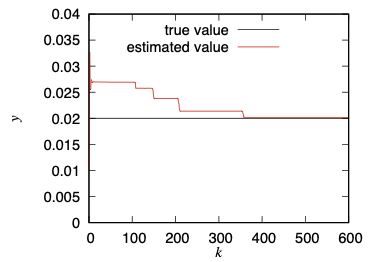}
\end{center}
\caption{
Convergence of the iteration in the conjugate gradient method for the water leakage problem in Sec.~\ref{sec1}. The parameters were set to $\alpha=8\times10^{-9}$, $x^0=3$, $y^0=0.01$. The true values are $x=4$, $y=0.02$.
}
\label{sec7:fig01}
\end{figure}

\begin{figure}[ht]
\begin{center}
\includegraphics[width=0.4\textwidth]{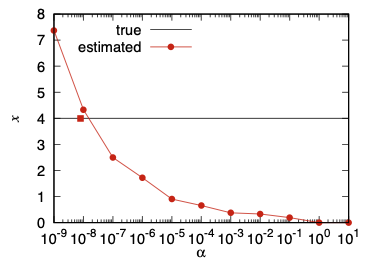}
\includegraphics[width=0.4\textwidth]{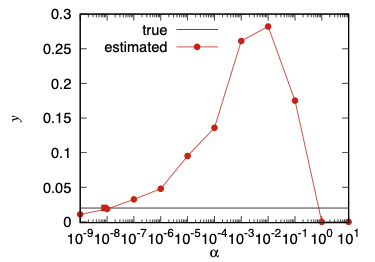}
\end{center}
\caption{
Obtained $x,y$ after $1000$ iterations are shown for different $\alpha$ when the water leakage problem in Sec.~\ref{sec1} is solved with the conjugate gradient method.
}
\label{sec7:fig02}
\end{figure}

\begin{figure}[ht]
\begin{center}
\includegraphics[width=0.8\textwidth]{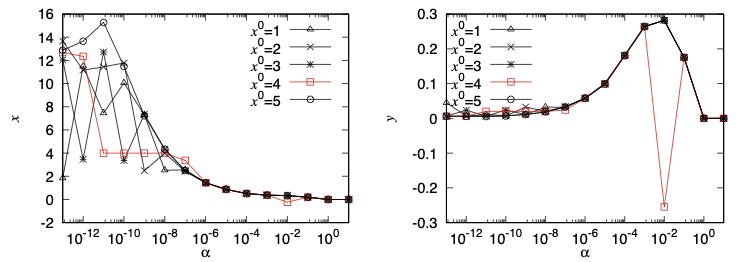}
\end{center}
\caption{
Obtained $x,y$ after $10000$ iterations are shown for different $\alpha$ when the water leakage problem in Sec.~\ref{sec1} is solved with the conjugate gradient method with different initial values ($x^0=1,2,3,4,5$, $y^0=0.01$).
}
\label{sec7:fig03}
\end{figure}

Let us further investigate how $\alpha$ should be chosen. Let $A_1^{\rm est}(\alpha),A_2^{\rm est}(\alpha)$ be results of $K\bv{f}$ in which estimated values $x,y$ are used. The error $\varepsilon^{\rm est}(\alpha)$ between observed values and $A_1^{\rm est}(\alpha),A_2^{\rm est}(\alpha)$ is given by
\be
\varepsilon^{\rm est}(\alpha)=\sqrt{\left(\frac{A_1^{\rm est}(\alpha)-A_1^{\rm obs}}{A_1^{\rm obs}}\right)^2+\left(\frac{A_2^{\rm est}(\alpha)-A_2^{\rm obs}}{A_2^{\rm obs}}\right)^2}.
\ee
Let us introduce the measurement error ($\varepsilon^{\rm obs}$) as
\be
\varepsilon^{\rm obs}=\sqrt{\left(\frac{A_1^{\rm obs}-A_1^{\rm true}}{A_1^{\rm true}}\right)^2+\left(\frac{A_2^{\rm obs}-A_2^{\rm true}}{A_2^{\rm true}}\right)^2}.
\ee
We note that there is no way to know the true values and true error $\varepsilon^{\rm obs}$. In the actual situation, we need to estimate $\varepsilon^{\rm obs}$ with other methods.

Since it is not likely that $x,y$ can be determined with accuracy exceeding the measurement error, we will choose $\alpha>0$ such that
\be
\varepsilon^{\rm est}(\alpha)\approx\varepsilon^{\rm obs}.
\ee
Results of this numerical experiment are shown in Fig.~\ref{sec7:fig04}.

\begin{figure}[ht]
\begin{center}
\includegraphics[width=0.4\textwidth]{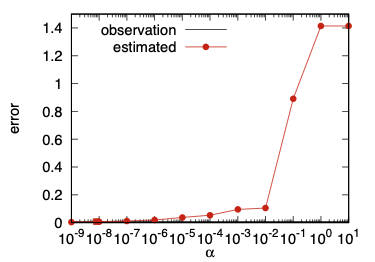}
\includegraphics[width=0.4\textwidth]{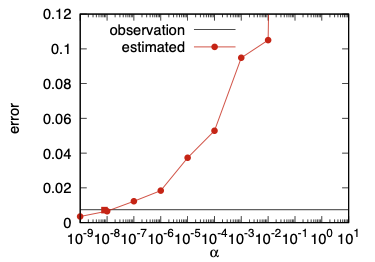}
\end{center}
\caption{
The behavior of the error as a function of $\alpha$ (Left) and a magnified figure (Right). In the right panel, the blue arrow shows $\alpha=8\times10^{-9}$, which is used in Fig.~\ref{sec7:fig01}.
}
\label{sec7:fig04}
\end{figure}

\section{Optical tomography}

\label{sec8}

Optical tomography is an imaging modality which uses near-infrared light (wavelength about from $700\,{\rm nm}$ to $1\,\mu{\rm m}$.\footnote{
The term ``tomo'' originated from $\tau$\'{o}$\mu$o$\zeta$ (tomos) in Greek, which means section, cut, or slice.
} As a result of the fact that light in biological tissue undergoes multiple scattering, the energy density $u$ of light is governed by the following diffusion equation.\footnote{
Indeed, the radiative transport equation (below) is considered to be the fundamental governing equation for the light propagation in biological tissue:
\[
\left\{\begin{aligned}
\frac{1}{c}\pp_tI+\uv\cdot\nabla I+(\mu_a+\mu_s)I=\mu_s\int_{\Sm^2}p(x,\uv,\uv')I(x,\uv',t)\,d\uv',
&\quad x\in\Omega,\quad\uv\in\Sm^2,\quad 0<t<T,
\\
I(x,\uv,t)=I_0(x,\uv,t)+(\mathcal{R}I)(x,\uv,t),
&\quad x\in\pp\Omega,\quad\uv\in\Sm^2,\quad 0<t<T,
\\
I=0,&\quad x\in\Omega,\quad\uv\in\Sm^2,\quad t=0,
\end{aligned}\right.
\]
where $I(x,\uv,t)$ is the specific intensity of light in direction $\uv\in\Sm^2$. Here, $c$ is the speed of light in the medium, $\mu_a(x),\mu_s(x)$ are absorption and scattering coefficients, respectively, and $p$ is the scattering phase function. The incident light is denoted by $I_0$ and $(\mathcal{R}I)$ is the reflected light. On a large scale of $t$ and $x$ under the condition of small $\mu_a$, the integrated specific intensity $\int_{\Sm^2}I(x,\uv,t)\,d\uv$ coincides with the solution $u(x,t)$ of the diffusion equation. This is called the diffusion approximation. In the case of constant coefficients and a position-independent phase function, we have the relation $D_0=c/(3(1-g)\mu_s)$, $g=\int_{\Sm^2}(\uv\cdot\uv')p(\uv,\uv')\,d\uv'$ and $\alpha=c\mu_a$.
}
\be
\left\{\begin{aligned}
\pp_tu+\nabla\cdot D\nabla u+\alpha u=S,&\quad x\in\Omega,\quad 0<t<T,
\\
D\pp_{\nu}u+\beta u=0,&\quad x\in\pp\Omega,\quad 0<t<T,
\\
u=0,&\quad x\in\Omega,\quad t=0,
\end{aligned}\right.
\ee
where $D(x)$ is the diffusion coefficient, $\alpha(x)$ is the absorption coefficient, and $S(x,t)$ denotes the light source. Here, $T$ is the observation time. In the boundary condition, $\pp_{\nu}$ is the directional derivative in the direction of $\nu(x)$, which is the outer unit normal vector at $x\in\pp\Omega$. The coefficient $\beta>0$ is determined by the Fresnel reflection of light on the boundary. When light is absorbed on the surface, for example, by a holder which holds optical fibers, we have the Dirichlet boundary condition.

For simplicity, we assume that $D_0=D(x)$ is a constant. If light is applied to the sample continuously in time (the continuous-wave (CW) measurement), the light obey the following time-independent diffusion equation.
\be
\left\{\begin{aligned}
-D_0\Delta u+\alpha u=S,&\quad x\in\Omega,
\\
D_0\pp_{\nu}u+\beta u=0,&\quad x\in\pp\Omega.
\end{aligned}\right.
\ee

\section{Inverse series}
\label{sec9}

The Born approximation is a method to reconstruct coefficients of a partial differential equation \cite{Isakov06}. The method reconstructs perturbation of a coefficient. Usually, the Born approximation is accompanied by linearization of nonlinear inverse problems. Here, we will see that the inversion by the Born approximation can be extended to nonlinear with the inverse series.

\subsection{Inverse Born series}
\hfill\vskip1mm

Let us set (see the scaling in Sec.~\ref{sec3:der})
\be
D_0=1,\quad \alpha(x)=k^2\left(1+\eta(x)\right),\quad\beta=\ell D_0,
\ee
where $k,\ell$ are positive constants. Since $\alpha$ is nonnegative, $\eta(x)\ge-1$. We assume that the support of $\eta$ is contained in an open ball $B_a$ of radius $a$. In particular, $\left.\eta\right|_{\pp\Omega}=0$. Then we have \cite{Moskow-Schotland08}
\be
\left\{\begin{aligned}
-\Delta u(x)+k^2(1+\eta)u(x)=S(x),&\quad x\in\Omega,
\\
\ell\pp_{\nu}u(x)+u(x)=0,&\quad x\in\pp\Omega.
\end{aligned}\right.
\ee
We suppose that coefficients $k,\ell$ and the source term $S$ are known.

To find the unknown function $\eta$, we prepare another diffusion equation:
\be
\left\{\begin{aligned}
-\Delta u_0+k^2u_0=S,&\quad x\in\Omega,
\\
\ell\pp_{\nu}u_0+u_0=0,&\quad x\in\pp\Omega.
\end{aligned}\right.
\ee
We note that there are no unknown parameters nor functions in the above diffusion equation. By subtraction we have
\be
\left\{\begin{aligned}
-\Delta (u-u_0)+k^2(u-u_0)=-k^2\eta u,&\quad x\in\Omega,
\\
\ell\pp_{\nu}(u-u_0)+u-u_0=0,&\quad x\in\pp\Omega.
\end{aligned}\right.
\ee
Let us introduce the Green's function as
\be
\left\{\begin{aligned}
-\Delta G(x,y)+k^2G(x,y)=\delta(x-y),&\quad x\in\Omega,
\\
\ell\pp_{\nu}G(x,y)+G(x,y)=0,&\quad x\in\pp\Omega.
\end{aligned}\right.
\ee
Thus we arrive at the identity:
\begin{equation}
u(x)=u_0(x)-k^2\int_{\Omega}G(x,y)\eta(y)u(y)\,dy.
\label{sec9:ident}
\end{equation}

By recursively substituting $u$ on the left-hand side of (\ref{sec9:ident}) for $u$ in the integral on the right-hand side, we obtain the Born series:
\begin{equation}
u=u_0+u_1+u_2+\cdots,
\label{sec9:Bornu0u1}
\end{equation}
where
\be
u_n(x)=-k^2\int_{\Omega}G(x,y)\eta(y)u_{n-1}(y)\,dy,\quad n\in\Nm.
\ee

Let us introduce
\be
\phi(x)=u_0(x)-u(x),\quad x\in\pp\Omega.
\ee
Then we can rewrite the Born series as
\begin{equation}
\phi=K_1\eta+K_2\eta\otimes\eta+K_3\eta\otimes\eta\otimes\eta+\cdots,
\label{sec9:fwdBorn}
\end{equation}
where
\be
u_n(x)=-K_n\eta^{\otimes n}\quad(n=1,2,\dots).
\ee
Here, $\otimes$ means tensor product. We can explicitly write
\begin{equation}
\begin{aligned}
(K_nf)(x^1,x^2)
&=
(-1)^{n+1}k^{2n}\int_{B_a\times\cdots\times B_a}
G(x^1,y^1)G(y^1,y^2)\cdots G(y^{n-1},y^n)
\\
&\times
G(y^n,x^2)f(y^1,\dots,y^n)\,dy^1\cdots dy^n,\quad x^1,x^2\in\pp\Omega,
\end{aligned}
\label{Kdef}
\end{equation}
where $f\in L^{\infty}(B_a\times\cdots\times B_a)$ is a multilinear function.

From the Born series (\ref{sec9:fwdBorn}), we can write the following inverse Born series.
\begin{equation}
\eta=\mathcal{K}_1\phi+\mathcal{K}_2\phi\otimes\phi+\mathcal{K}_3\phi\otimes\phi\otimes\phi+\cdots.
\label{sec9:invBorn}
\end{equation}
The notion of the inverse series already existed. Let us consider a simple power-law series:
\be
y=a_1x+a_2x^2+a_3x^3+\cdots.
\ee
It is possible to express $x$ as a series in $y$:
\be
x=A_1y+A_2y^2+A_3y^3+\cdots,
\ee
where the first a few terms are given by (see, for example, \cite{Dwight57})
\be
A_1=\frac{1}{a_1},\quad A_2=-\frac{a_2}{a_3},\quad
A_3=\frac{2a_2^2-a_1a_3}{a_1^5}.
\ee
Here, $K_n,\mathcal{K}_n$ are operators and we wish to obtain the general terms in the series instead of just the first a few terms.

Let us substitute $\phi$ in (\ref{sec9:fwdBorn}) for $\phi$ on the right-hand side of (\ref{sec9:invBorn}). By comparing both sides of the resulting equation, we obtain the following relations up to the third order terms.
\be
\begin{aligned}
&
\mathcal{K}_1K_1=I,
\\
&
\mathcal{K}_2K_1\otimes K_1+
\mathcal{K}_1K_2=0,
\\
&
\mathcal{K}_3K_1\otimes K_1\otimes K_1+
\mathcal{K}_2K_1\otimes K_2+\mathcal{K}_2K_2\otimes K_1+
\mathcal{K}_1K_3=0.
\end{aligned}
\ee
Hence we obtain
\be
\begin{aligned}
\mathcal{K}_2&=-\mathcal{K}_1K_2\mathcal{K}_1\otimes\mathcal{K}_1,
\\
\mathcal{K}_3&=-\left(\mathcal{K}_2K_1\otimes K_2+
\mathcal{K}_2K_2\otimes K_1+
\mathcal{K}_1K_3\right)\mathcal{K}_1\otimes\mathcal{K}_1\otimes\mathcal{K}_1.
\end{aligned}
\ee
In general, we have the relations
\be
\sum_{m=1}^{n-1}\mathcal{K}_m\sum_{i_1+\cdots+i_m=n}K_{i_1}\otimes\cdots\otimes K_{i_m}+
\mathcal{K}_nK_1\otimes\cdots\otimes K_1=0,\quad
n\ge2.
\ee
Thus general terms are given by \cite{Markel-OSullivan-Schotland03,Moskow-Schotland08}
\begin{equation}
\mathcal{K}_n=-\left(
\sum_{m=1}^{n-1}\mathcal{K}_m\sum_{i_1+\cdots+i_m=n}K_{i_1}\otimes\cdots\otimes K_{i_m}\right)
\mathcal{K}_1\otimes\cdots\otimes\mathcal{K}_1,\quad
n\ge2.
\label{sec9:invgenterms}
\end{equation}
For nonlinear terms, the number of $n$th-order terms is
\be
\sum_{m=1}^{n-1}\begin{pmatrix}n-1\\m-1\end{pmatrix}=2^{n-1}-1
\quad(n=2,3,\dots),
\ee
where $\begin{pmatrix}n-1\\m-1\end{pmatrix}$ is the number of ordered partitions of the integer $n$ into $m$ parts. 

The first term of the inverse Born series (i.e., the conventional Born approximation) needs some care. Indeed, $\mathcal{K}_1K_1=I$ does not hold due to the fact that the inverse problem is ill-posed. We have $\mathcal{K}_1K_1\approx I$. Thus, $\mathcal{K}_1$ is a regularized pseudoinverse (e.g., truncated SVD), which can be given by \cite{Machida-Schotland15}
\be
\eta^*=\mathop{\mathrm{argmin}}_{\eta\in B_a}\left(
\frac{1}{2}\left\|K_1\eta-\phi\right\|_{L^2(\pp\Omega)}^2+\alpha R(\eta)
\right),
\ee
where $R(\eta)$ is the penalty term with the regularization parameter $\alpha>0$.

Let us introduce
\be
\mu_{\infty}=k^2\sup_{x\in B_a}\|G(x,\cdot)\|_{L^1(B_a)},\quad
\nu_{\infty}=k^2|B_a|\sup_{x\in B_a}\sup_{y\in\pp\Omega}|G(x,y)|^2.
\ee

\begin{lem}
\label{sec9:lem1}
For each $n\in\Nm$, the operator $K_n\colon L^{\infty}(B_a\times\cdots\times B_a)\to L^{\infty}(\pp\Omega\times\pp\Omega)$ is bounded and
\be
\|K_n\|_{\infty}\le\nu_{\infty}\mu_{\infty}^{n-1}.
\ee
\end{lem}

\begin{proof}
From (\ref{Kdef}), we have
\ba
\|K_nf\|_{L^{\infty}(\pp\Omega\times\pp\Omega)}
&=
\sup_{(x^1,x^2)\in\pp\Omega\times\pp\Omega}|(K_nf)(x^1,x^2)|
\\
&\le
k^{2n}\|f\|_{\infty}\sup_{(x^1,x^2)\in\pp\Omega\times\pp\Omega}
\int_{B_a\times\cdots\times B_a}
|G(x^1,y^1)\cdots G(y^n,x^2)|\,dy^1\cdots dy^n.
\ea
We begin by estimating the above integral for $n=1$:
\be
\|K_1\|_{\infty}\le
k^2\sup_{(x^1,x^2)\in\pp\Omega\times\pp\Omega}\int_{B_a}|G(x^1,y)G(y,x^2)|\,dy
\le
k^2|B_a|\sup_{x\in B_a}\sup_{y\in\pp\Omega}|G(x,y)|^2.
\ee
For $n\ge2$, we take out the first and last factors of $G$ in the integral to obtain
\ba
\|K_n\|_{\infty}
&\le
\sup_{(x^1,x^2)\in\pp\Omega\times\pp\Omega}\sup_{y^1\in B_a,\;y^n\in B_a}
|G(x^1,y)G(y^n,x^2)|
\\
&\le
k^{2n}\int_{B_a\times\cdots\times B_a}|G(y^1,y^2)\cdot G(y^{n-1},y^n)|
\,dy^1\cdots dy^n.
\ea
Then we obtain
\be
\|K_n\|_{\infty}\le\left(\sup_{x\in B_a}\sup_{y\in\pp\Omega}|G(x,y)|\right)^2I_{n-1},
\ee
where
\be
I_{n-1}=k^{2n}\int_{B_a\times\cdots\times B_a}|G(y^1,y^2)\cdots G(y^{n-1},y^n)|\,dy^1\cdots dy^n.
\ee
We note that $I_{n-1}$ can be estimated recursively. Since
\ba
I_{n-1}
&\le
k^2\sup_{y^{n-1}\in B_a}\int_{B_a}|G(y^{n-1},y^n)|\,dy^n
\\
&\times
k^{2n-2}\int_{B_a\times\cdots\times B_a}|G(y^1,y^2)\cdots G(y^{n-2},y^{n-1})|\,dy^1\cdots dy^{n-1},
\ea
we obtain
\be
I_{n-1}\le\mu_{\infty}I_{n-2}.
\ee

We have
\be
I_1=k^4\int_{B_a\times B_a}|G(x,y)|\,dxdy\le k^2|B_a|\mu_{\infty}.
\ee
Thus,
\be
I_{n-1}\le k^2|B_a|\mu_{\infty}^{n-1}.
\ee
\end{proof}

\begin{lem}
\label{sec9:lem2}
Suppose $(\mu_{\infty}+\nu_{\infty})\|\mathcal{K}_1\|_{\infty}<1$. For each $n\in\Nm$, the operator $\mathcal{K}_n\colon L^{\infty}(\pp\Omega\times\cdots\times \pp\Omega)\to L^{\infty}(B_a)$ is bounded and
\be
\|\mathcal{K}_n\|_{\infty}\le
C(\nu_{\infty}+\mu_{\infty})^n\|\mathcal{K}_1\|_{\infty}^n,
\ee
where $C=C(\mu_{\infty},\nu_{\infty},\|\mathcal{K}_1\|_{\infty})$ is independent of $n$.
\end{lem}

\begin{proof}
From (\ref{sec9:invgenterms}), we have
\ba
\|\mathcal{K}_n\|_{\infty}
&\le
\sum_{m=1}^{n-1}\sum_{i_1+\cdots+i_m=n}\|\mathcal{K}_m\|_{\infty}
\|K_{i_1}\|_{\infty}\cdots\|K_{i_m}\|_{\infty}\|\mathcal{K}_1\|_{\infty}^n
\\
&\le
\|\mathcal{K}_1\|_{\infty}^n\sum_{m=1}^{n-1}\sum_{i_1+\cdots+i_m=n}
\|\mathcal{K}_m\|_{\infty}\nu_{\infty}\mu_{\infty}^{i_1-1}\cdots
\nu_{\infty}\mu_{\infty}^{i_m-1},
\ea
where we have used Lemma \ref{sec9:lem1} to obtain the second inequality. Furthermore,
\ba
\|\mathcal{K}_n\|_{\infty}
&\le
\|\mathcal{K}_1\|_{\infty}^n\sum_{m=1}^{n-1}\|\mathcal{K}_m\|_{\infty}
\begin{pmatrix}n-1\\m-1\end{pmatrix}\nu_{\infty}^m\mu_{\infty}^{n-m}
\\
&\le
\|\mathcal{K}_1\|_{\infty}^n\left(\sum_{m=1}^{n-1}\|\mathcal{K}_m\|_{\infty}\right)
\left(\sum_{m=1}^{n-1}\begin{pmatrix}n-1\\m-1\end{pmatrix}\nu_{\infty}^m\mu_{\infty}^{n-m}\right)
\\
&\le
\nu_{\infty}\|\mathcal{K}_1\|_{\infty}^n\left(\sum_{m=1}^{n-1}\|\mathcal{K}_m\|_{\infty}\right)
\left(\sum_{m=0}^{n-1}\begin{pmatrix}n-1\\m\end{pmatrix}\nu_{\infty}^m\mu_{\infty}^{n-1-m}\right)
\\
&=
\nu_{\infty}\|\mathcal{K}_1\|_{\infty}^n(\mu_{\infty}+\nu_{\infty})^{n-1}
\sum_{m=1}^{n-1}\|\mathcal{K}_m\|_{\infty}.
\ea
Thus, $\mathcal{K}_m$ is a bounded operator and
\be
\|\mathcal{K}_n\|_{\infty}\le
(\nu_{\infty}+\mu_{\infty})^n\|\mathcal{K}_1\|_{\infty}^n
\sum_{m=1}^{n-1}\|\mathcal{K}_m\|_{\infty}.
\ee
We note that the above estimate for $\|\mathcal{K}_n\|_{\infty}$ has a recursive structure. It can be seen that
\be
\|\mathcal{K}_n\|_{\infty}\le
C_n\left((\nu_{\infty}+\mu_{\infty})\|\mathcal{K}_1\|_{\infty}\right)^n
\|\mathcal{K}_1\|_{\infty},
\ee
where $C_2=1$ and
\be
C_{n+1}=C_n+\left((\nu_{\infty}+\mu_{\infty})\|\mathcal{K}_1\|_{\infty}\right)^nC_n,\quad n\ge2.
\ee
We obtain
\be
C_n=\prod_{m=2}^{n-1}\left(1+((\nu_{\infty}+\mu_{\infty})\|\mathcal{K}_1\|_{\infty})^m\right).
\ee
We have
\ba
\ln C_n
&\le
\sum_{m=1}^{n-1}\ln\left(1+((\nu_{\infty}+\mu_{\infty})\|\mathcal{K}_1\|_{\infty})^m\right)
\le
\sum_{m=1}^{n-1}\left((\nu_{\infty}+\mu_{\infty})\|\mathcal{K}_1\|_{\infty}\right)^m
\\
&\le
\frac{1}{1-(\nu_{\infty}+\mu_{\infty})\|\mathcal{K}_1\|_{\infty}},
\ea
where the final inequality follows if $(\mu_{\infty}+\nu_{\infty})\|\mathcal{K}_1\|_{\infty}<1$. Therefore, all $C_n$ are bounded.
\end{proof}

We have the following error estimate for the inverse Born series \cite{Moskow-Schotland08}.

\begin{thm}
\label{sec9:errorthm}
We assume that the Born series and inverse Born series converge. Suppose that $\|\mathcal{K}_1\|_{\infty}<1/(\mu_{\infty}+\nu_{\infty})$ and $\|\mathcal{K}_1\phi\|_{L^{\infty}(B_a)}<1/(\mu_{\infty}+\nu_{\infty})$. Let $M=\max(\|\eta\|_{L^{\infty}(B_a)},\|\mathcal{K}_1K_1\eta\|_{L^{\infty}(B_a)})$ and assume that $M<1/(\mu_{\infty}+\nu_{\infty})$. Then there exists a positive constant $C=C(\mu_{\infty},\nu_{\infty},\|\mathcal{K}_1\|_{\infty},M)$ such that
\be
\left\|\eta-\sum_{n=1}^N\mathcal{K}_n\phi\otimes\cdots\otimes\phi\right\|_{L^{\infty}(B_a)}
\le
C\|(I-\mathcal{K}_1K_1)\eta\|_{L^{\infty}(B_a)}+
C\frac{\left[(\mu_{\infty}+\nu_{\infty})\|\mathcal{K}_1\|_{\infty}\|\phi\|_{L^{\infty}(\pp\Omega)}\right]^N}{1-(\mu_{\infty}+\nu_{\infty})\|\mathcal{K}_1\|_{\infty}\|\phi\|_{L^{\infty}(\pp\Omega)}}.
\ee
\end{thm}

\begin{proof}
Let us write
\be
\widetilde{\eta}=\sum_{n=1}^{\infty}\mathcal{K}_n\phi\otimes\cdots\otimes\phi.
\ee
We note that
\be
\phi=\sum_{n=1}^{\infty}K_n\eta\otimes\cdots\otimes\eta.
\ee
By substitution we obtain
\be
\widetilde{\eta}=\sum_{n=1}^{\infty}\widetilde{\mathcal{K}}_n\eta\otimes\cdots\otimes\eta,
\ee
where
\be
\widetilde{\mathcal{K}}_1=\mathcal{K}_1K_1,
\ee
and
\be
\widetilde{\mathcal{K}}_n=\left(\sum_{m=1}^{n-1}\mathcal{K}_m\sum_{i_1+\cdots+i_m=n}K_{i_1}\otimes\cdots\otimes K_{i_m}\right)
+\mathcal{K}_nK_1\otimes\cdots\otimes K_1,\quad n\ge2.
\ee
From (\ref{sec9:invgenterms}), it follows that
\be
\widetilde{\mathcal{K}}_n=\sum_{m=1}^{n-1}\mathcal{K}_m\sum_{i_1+\cdots+i_m=n}K_{i_1}\otimes\cdots\otimes K_{i_m}
\left(I-\mathcal{K}_1K_1\otimes\cdots\otimes\mathcal{K}_1K_1\right).
\ee
Hence,
\be
\widetilde{\eta}=
\mathcal{K}_1K_1\eta+\widetilde{\mathcal{K}}_2\eta\otimes\eta+\cdots.
\ee
We thus obtain
\ba
\eta-\widetilde{\eta}
&=
\left(I-\mathcal{K}_1K_1\right)\eta-\mathcal{K}_1K_2
\left(\eta\otimes\eta-\mathcal{K}_1K_1\eta\otimes\mathcal{K}_1K_1\eta\right)
\\
&-
\mathcal{K}_1K_3\left(\eta\otimes\eta\otimes\eta-\mathcal{K}_1K_1\eta\otimes\mathcal{K}_1K_1\eta\otimes\mathcal{K}_1K_1\eta\right)
\\
&-
\mathcal{K}_2\left(K_1\otimes K_2+K_2\otimes K_1\right)\left(\eta\otimes\eta\otimes\eta-\mathcal{K}_1K_1\eta\otimes\mathcal{K}_1K_1\eta\otimes\mathcal{K}_1K_1\eta\right)
\\
&-\cdots.
\ea
From the above equality we have
\ba
\|\eta-\widetilde{\eta}\|_{L^{\infty}(B_a)}
&\le
\sum_{n=1}^{\infty}\sum_{m=1}^{n-1}\sum_{i_1+\cdots+i_m=n}
\|\mathcal{K}_m\|_{\infty}\|K_{i_1}\|_{\infty}\cdots\|K_{i_m}\|_{\infty}
\\
&\times
\left\|\eta\otimes\cdots\otimes\eta-\mathcal{K}_1K_1\eta\otimes\cdots\otimes\mathcal{K}_1K_1\eta\right\|_{\infty}.
\ea
For the forward data $\phi_1,\phi_2$, the following identity holds:
\ba
&
\phi_1\otimes\cdots\otimes\phi_1-\phi_2\otimes\cdots\otimes\phi_2
\\
&=
(\phi_1-\phi_2)\otimes\phi_2\otimes\cdots\otimes\phi_2+
\phi_1\otimes(\phi_1-\phi_2)\otimes\phi_2\otimes\cdots\otimes\phi_2
\\
&+\cdots+
\phi_1\otimes\phi_1\otimes\cdots\otimes(\phi_1-\phi_2)\otimes\phi_2+
\phi_1\otimes\phi_1\otimes\cdots\otimes\phi_1\otimes(\phi_1-\phi_2).
\ea
We obtain
\be
\left\|\eta\otimes\cdots\otimes\eta-\mathcal{K}_1K_1\eta\otimes\cdots\otimes\mathcal{K}_1K_1\eta\right\|_{\infty}
\le
nM^{n-1}\|\psi\|_{L^{\infty}(B_a)},
\ee
where
\be
\psi=\eta-\mathcal{K}_1K_1\eta.
\ee
Using Lemma \ref{sec9:lem1}, we have
\ba
\|\eta-\widetilde{\eta}\|_{L^{\infty}(B_a)}
&\le
\sum_{n=1}^{\infty}\sum_{m=1}^{n-1}\sum_{i_1+\cdots+i_m=n}
\|\mathcal{K}_m\|_{\infty}\|K_{i_1}\|_{\infty}\cdots\|K_{i_m}\|_{\infty}
nM^{n-1}\|\psi\|_{L^{\infty}(B_a)}
\\
&\le
\sum_{n=1}^{\infty}\sum_{m=1}^{n-1}nM^{n-1}\|\mathcal{K}_m\|_{\infty}
\begin{pmatrix}n-1\\m-1\end{pmatrix}\nu_{\infty}^m\mu_{\infty}^{n-m}
\|\psi\|_{L^{\infty}(B_a)}
\\
&\le
\nu_{\infty}\sum_{n=1}^{\infty}\|\psi\|_{L^{\infty}(B_a)}nM^{n-1}
\left(\sum_{m=1}^{n-1}\|\mathcal{K}_m\|_{\infty}\right)
\left(\sum_{m=0}^{n-1}\begin{pmatrix}n-1\\m\end{pmatrix}
\nu_{\infty}^m\mu_{\infty}^{n-1-m}\right)
\\
&\le
\|\psi\|_{L^{\infty}(B_a)}\sum_{n=1}^{\infty}\sum_{m=1}^{n-1}nM^{n-1}
(\mu_{\infty}+\nu_{\infty})^n\|\mathcal{K}_m\|_{\infty}.
\ea

We now apply Lemma \ref{sec9:lem2} to obtain
\be
\|\eta-\widetilde{\eta}\|_{\infty}
\le
C\|\psi\|_{L^{\infty}(B_a)}\sum_{n=1}^{\infty}\sum_{m=1}^{n-1}nM^{n-1}
(\mu_{\infty}+\nu_{\infty})^{m+n}\|\mathcal{K}_1\|_{\infty}^m,
\ee
where the constant $C>0$ is independent of $m,n$. Performing the sum over $m$, we have
\be
\|\eta-\widetilde{\eta}\|_{\infty}
\le
C\|\psi\|_{L^{\infty}(B_a)}\sum_{n=1}^{\infty}nM^{n-1}
(\mu_{\infty}+\nu_{\infty})^n
\frac{(\mu_{\infty}+\nu_{\infty})^n\|\mathcal{K}_1\|_{\infty}^n-1}{(\mu_{\infty}+\nu_{\infty})\|\mathcal{K}_1\|_{\infty}-1}.
\ee
Recall $M(\mu_{\infty}+\nu_{\infty})<1$ and $(\mu_{\infty}+\nu_{\infty})\|\mathcal{K}_1\|_{\infty}<1$. Hence, there exists a positive constant $C=C(\mu_{\infty},\nu_{\infty},M,\|\mathcal{K}_1\|_{\infty})$ such that
\be
\|\eta-\widetilde{\eta}\|_{\infty}\le C\|(I-\mathcal{K}_1K_1)\eta\|_{\infty}.
\ee
Finally, using the triangle inequality, we can account for the error which arises from cutting off the remainder of the series. We thus obtain
\ba
&
\left\|\eta-\sum_{n=1}^N\mathcal{K}_n\phi\otimes\cdots\otimes\phi\right\|_{L^{\infty}(B_a)}
\\
&\le
\left\|\eta-\sum_{n=1}^{\infty}\widetilde{\mathcal{K}}_n\eta\otimes\cdots\otimes\eta\right\|_{L^{\infty}(B_a)}+
\sum_{n=N+1}^{\infty}\|\mathcal{K}_n\phi\otimes\cdots\otimes\phi\|_{L^{\infty}(B_a)}
\\
&\le
C_1\|(I-\mathcal{K}_1K_1)\eta\|_{L^{\infty}(B_a)}+
C_2\frac{\left((\mu_{\infty}+\nu_{\infty})\|\mathcal{K}_1\|_{\infty}\|\phi\|_{L^{\infty}(B_a)}\right)^{N+1}}{1-(\mu_{\infty}+\nu_{\infty})\|\mathcal{K}_1\|_{\infty}\|\phi\|_{L^{\infty}(B_a)}},
\ea
where $C_1,C_2$ are positive constants. This completes the proof.
\end{proof}

\subsection{Recursive algorithm}
\hfill\vskip1mm

It is possible to construct a recursive algorithm \cite{Moskow-Schotland09}. For the sake of numerical calculation, we discretize spatial variables and consider $\bv{\eta}\in N_{\Omega}$ instead of $\eta(x)$, where $N_{\Omega}$ is the number of discrete points in the domain.

We prepare the forward function,
\be
\Rm^{M_{\rm SD}}\ni\bv{K}\left(n,\bv{a}^{(1)},\dots,\bv{a}^{(n)}\right),
\quad n=1,2,\dots,
\ee
where $\bv{a}^{(1)},\dots,\bv{a}^{(n)}$ are $N_{\Omega}$ dimensional vectors. The $n$th term $\bv{\phi}^{(n)}\in\Rm^{M_{\rm SD}}$ in the Born series is obtained as
\be
\bv{\phi}^{(n)}=\bv{K}\left(n,\bv{\eta},\dots,\bv{\eta}\right),
\quad n=1,2,\dots.
\ee
We also prepare the inverse function,
\begin{equation}
\Rm^{N_{\Omega}}\ni\bv{\mathcal{K}}\left(n,\bv{b}^{(1)},\dots,\bv{b}^{(n)}\right),
\quad n=1,2,\dots,
\label{sec9:invfunc}
\end{equation}
where $\bv{b}^{(1)},\dots,\bv{b}^{(n)}$ are $M_{\rm SD}$ dimensional vectors. The $n$th term $\bv{\eta}^{(n)}\in\Rm^{N_{\Omega}}$ in the inverse Born series is calculated as
\be
\bv{\eta}^{(n)}=\bv{\mathcal{K}}\left(n,\bv{\phi},\dots,\bv{\phi}\right),
\quad n=1,2,\dots.
\ee

The inverse function in (\ref{sec9:invfunc}) has a recursive structure. To compute $\bv{\eta}^{(n)}$, first, obtain the compositions $[i_1,\dots,i_m]$ such that $i_1+\cdots+i_m=n$. For each $m$ ($1\le m\le n-1$) and each composition $(i_1,\dots,i_m)$, we calculate
\be
\bv{\eta}_{\rm tmp}=\bv{\mathcal{K}}\left(m,
\bv{K}(i_1,\bv{\eta}^{(1)},\dots,\bv{\eta}^{(i_1)}),\dots,
\bv{K}(i_m,\bv{\eta}^{(n-i_m+1)},\dots,\bv{\eta}^{(n)})\right).
\ee
Let $\bv{\Sigma}_2$ denote the sum of $\bv{\eta}_{\rm tmp}$ for all $\begin{pmatrix}n-1\\m-1\end{pmatrix}$ compositions. This procedure is repeated for all $m$ ($1\le m\le n-1$). Let $\bv{\Sigma}_1$ denote the sum of the results:
\be
\bv{\Sigma}_1=\sum_{m=1}^{n-1}\bv{\Sigma}_2.
\ee
Finally, the function $\bv{\mathcal{K}}$ returns $\bv{\eta}^{(n)}=\bv{\Sigma}_1$. 

Let us look at the first a few steps of the above procedure for the Born and inverse Born series:
\ba
\bv{\phi}&=\bv{\phi}^{(1)}+\bv{\phi}^{(2)}+\bv{\phi}^{(3)}+\cdots,
\\
\bv{\eta}&=\bv{\eta}^{(1)}+\bv{\eta}^{(2)}+\bv{\eta}^{(3)}+\cdots.
\ea
The algorithm is written as follows.
\begin{itemize}
\item[Step 1.] First order:
\be
\bv{\eta}^{(1)}=\bv{\mathcal{K}}(1,\bv{\phi}),
\ee
where $\bv{\mathcal{K}}(1,\bv{\phi})$ can be computed by a regularized pseudoinverse.
\item[Step 2.] Second order:
\be
\bv{\eta}^{(2)}=\bv{\mathcal{K}}(2,\bv{\phi},\bv{\phi})=
-\bv{\mathcal{K}}\left(1,\bv{K}(2,\bv{\eta}^{(1)},\dots,\bv{\eta}^{(1)})\right).
\ee
\item[Step 3.] Third order:
\be
\begin{aligned}
\bv{\eta}^{(3)}
&=
\bv{\mathcal{K}}(3,\bv{\phi},\bv{\phi},\bv{\phi})
\\
&=
-\bv{\mathcal{K}}\left(1,\bv{K}(3,\bv{\eta}^{(1)},\bv{\eta}^{(1)},\bv{\eta}^{(1)})\right)-
\bv{\mathcal{K}}\left(2,\bv{K}(1,\bv{\eta}^{(1)}),\bv{K}(2,\bv{\eta}^{(1)},\bv{\eta}^{(1)},\bv{\eta}^{(1)})\right)
\\
&\quad-
\bv{\mathcal{K}}\left(2,\bv{K}(2,\bv{\eta}^{(1)},\bv{\eta}^{(1)}),\bv{K}(1,\bv{\eta}^{(1)})\right).
\end{aligned}
\ee
\end{itemize}
The $N$th-order approximation is given by
\be
\bv{\eta}\approx\bv{\eta}^{(1)}+\bv{\eta}^{(2)}+\cdots+\bv{\eta}^{(N)}.
\ee

As an example, we here consider a two-dimensional radial problem of diffuse optical tomography for structured illumination. In the polar coordinate system we have $x=(r,\theta)$, where $r>0$ is the radial coordinate and $\theta\in[0,2\pi)$ is the angular coordinate. Let $\Omega$ be the disk of radius $R$ centered at the origin. Assuming that $\eta$ has the radial symmetry, we can write
\be
\eta(x)=\eta(r),\quad0<r<R.
\ee
Let us suppose that there is a target of radius $R_a$ in the disk:
\be
\eta(r)=\left\{\begin{aligned}
\eta_a,&\quad 0\le r\le R_a,
\\
0,&\quad R_a<r\le R.
\end{aligned}\right.
\ee
We assume the following spatially oscillating source:
\be
S(r,\theta)=e^{il\theta}\frac{1}{r}\delta(r-R),\quad l=1,\dots,M_S.
\ee

Since this problem is rather simple, we can solve the forward problem. Let us express the Green's function $G(x,y)$, which has the source term $\frac{1}{r_x}\delta(r_x-r_y)\delta(\theta_x-\theta_y)$, as
\be
G(x,y)=\frac{1}{2\pi}\sum_{n=-\infty}^{\infty}e^{in(\theta_x-\theta_y)}g_n(r_x,r_y),
\ee
where $g_n(r,r')$ satisfies
\ba
r^2\pp_r^2g_n(r,r')+r\pp_rg_n(r,r')-\left(k^2r^2+n^2\right)g_n(r,r')=-r\delta(r-r'),
\\
g_n(R,r')+\ell\pp_rg_n(R,r')=0.
\ea
We note that the homogeneous equation for the above equation is the modified Bessel differential equation. Hence the solution $u$ is given as a superposition of $I_n(kr),K_n(kr)$. Here, $I_n,K_n$ are the modified Bessel functions of the first and second kinds, respectively. We obtain
\begin{equation}
\begin{aligned}
g_n(r_x,r_y)
&=
K_n\left(k\max(r_x,r_y)\right)I_n\left(k\min(r_x,r_y)\right)
\\
&-
\frac{K_n(kR)+k\ell K'_n(kR)}{I_n(kR)+k\ell I'_n(kR)}I_n\left(kr_x\right)I_n\left(kr_y\right).
\end{aligned}
\label{sec9:lowergn}
\end{equation}
We note
\be
I_n'(x)=\frac{1}{2}\left(I_{n-1}(x)+I_{n+1}(x)\right),\quad
K_n'(x)=-\frac{1}{2}\left(K_{n-1}(x)+K_{n+1}(x)\right),\quad x\in\Rm.
\ee
Hence,
\be
u_0(x)=\int_0^{2\pi}\int_0^RG(x,y)S(r_y,\theta_y)r_y\,dr_yd\theta_y=
e^{il\theta_x}g_l(r_x,R),\quad l=1,2,\dots,M_S.
\ee
We have
\be
g_j(R,R)=I_j(kR)K_j(kR)-d_jI_j(kR),
\ee
where
\be
d_j=\frac{K_j(kR)+k\ell K'_j(kR)}{I_j(kR)+k\ell I'_j(kR)}I_j(kR).
\ee

For the forward data, we observe $u,u_0$ at $r_x=R$, $\theta_x=0$. That is, the outgoing light is measured at one point on the boundary while boundary values were observed at different points on the boundary in \cite{Moskow-Schotland09}. See Appendix \ref{fwd} for the calculation of $u$. Let us set $M_D=1$ (i.e., $M_{\rm SD}=M_S$). For the vector $\bv{\phi}\in\Rm^{M_{\rm SD}}$, we have
\begin{equation}\begin{aligned}
\phi_l
&=
e^{il\theta_x}\left(K_l(kR)-d_l\right)I_l(kR)-e^{il\theta_x}\left(I_l(kR)K_l(kR)+b_lK_l(kR)+c_lI_l(kR)\right)
\\
&=
-e^{il\theta_x}\left((d_l+c_l)I_l(kR)+b_lK_l(kR)\right)
\end{aligned}
\label{fdata}
\end{equation}
for $l=1,\dots,M_{\rm SD}$. We note that $b_l,c_l$, which are given in Appendix \ref{fwd}, depend on $\eta_a,R_a,k,R,\ell$.

Let us set
\be
G^{(j)}(r_x,r_y)=g_j(r_x,r_y)r_y,
\ee
where $g_j(r_x,r_y)$ was given in (\ref{sec9:lowergn}). We define $\bv{K}_1\in\Rm^{M_{\rm SD}N_r}$ as
\be
\{\bv{K}_1(\bv{a})\}_{i+(l-1)N_r}=
\frac{k^2\Delta r}{R}\sum_{n=1}^{N_r}G^{(l)}(r_i,r_n)G^{(l)}(r_n,R)\{\bv{a}\}_n
\ee
for $1\le l\le M_{\rm SD}$, $1\le i\le N_r$. Here, $\bv{a}\in\Rm^{N_r}$ is a vector and $N_r\Delta r=R$. Moreover we introduce $\bv{K}_n\in\Rm^{M_{\rm SD}N_r}$ as
\be
\{\bv{K}_n(\bv{a}^{(1)},\dots,\bv{a}^{(n)})\}_{i+(l-1)N_r}=
-k^2\Delta r\sum_{j=1}^{N_r}G^{(l)}(r_i,r_j)\{\bv{a}^{(n)}\}_j
\{\bv{K}_{n-1}(\bv{a}^{(1)},\dots,\bv{a}^{(n-1)})\}_{j+(l-1)N_r},
\ee
where $\bv{a}^{(1)},\dots,\bv{a}^{(n)}\in\Rm^{N_r}$ are vectors. Let $\bv{\mathcal{K}}_1$ be the regularized pseudoinverse of $\bv{K}_1$:
\be
\bv{\mathcal{K}}_1=\bv{K}_{1,{\rm reg}}^+.
\ee
For a vector $\bv{b}$, we have $\bv{\mathcal{K}}(1,\bv{b})=\bv{\mathcal{K}}_1\bv{b}$. Reconstructed results are shown in Fig.~\ref{sec9:fig01}. Since regularization is performed, the true shape cannot be reconstructed even in the absence of noise. In the figure, the projection means $\bv{\mathcal{K}}_1\bv{K}_1\bv{\eta}_{\rm true}$, where $\bv{\eta}_{\rm true}$ is the true shape of the target.

\begin{figure}[ht]
\begin{center}
\includegraphics[width=0.4\textwidth]{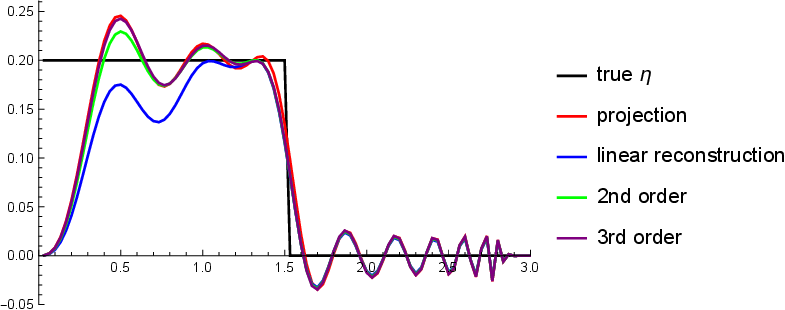}
\includegraphics[width=0.4\textwidth]{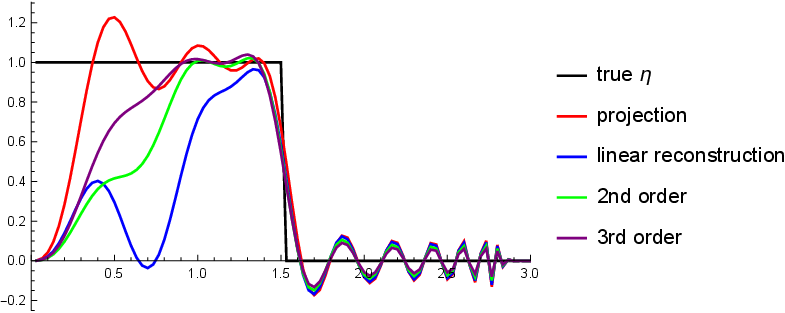}
\end{center}
\caption{
(Left) $\eta_a=0.2$ and (Right) $\eta_a=1.0$ for $N=3$ and $23$ largest singular values are used.
}
\label{sec9:fig01}
\end{figure}

\subsection{Inverse Rytov series}
\hfill\vskip1mm

The Rytov series is an expansion in the exponential function. Compared to (\ref{sec9:Bornu0u1}), the Rytov series is written as
\be
u=u_0e^{-\psi_1-\psi_2-\cdots}.
\ee
In particular, the first-order approximation is given by
\be
\psi_1=\ln{u_0}-\ln{u}.
\ee
This can be compared to the Born approximation $u_0-u=\phi\approx K_1\eta$.

Let us assume that the Born series converges. Then we observe
\ba
\ln\frac{u}{u_0}
&=
\ln\frac{u_0+u_1+\cdots}{u_0}=
\ln\left(1+\sum_{n=1}^{\infty}\frac{u_n}{u_0}\right)
\\
&=
\sum_{k=1}^{\infty}\frac{(-1)^{k+1}}{k}\left(\sum_{n=1}^{\infty}\frac{u_n}{u_0}\right)^k
\\
&=
\frac{u_1+u_2+\cdots}{u_0}-\frac{(u_1+u_2+\cdots)^2}{2u_0^2}+
\frac{(u_1+u_2+\cdots)^3}{3u_0^3}-\cdots
\\
&=
-\psi_1-\psi_2-\cdots.
\ea
By collecting the first- and second-order terms, the first two terms of the Rytov series are explicitly written as
\be
\psi_1=-\frac{u_1}{u_0},\quad
\psi_2=-\frac{u_2}{u_0}+\frac{1}{2}\left(\frac{u_1}{u_0}\right)^2.
\ee
In general, we have
\be
\psi_n=\sum_{m=1}^n\frac{(-1)^m}{mu_0^m}\sum_{i_1+\cdots+i_m=n}u_{i_1}\cdots u_{i_m},\quad n=1,2,\dots.
\ee

We introduce the forward operators $J_j:L^{\infty}(B_a)\times\cdots\times L^{\infty}(B_a)\to L^{\infty}(\pp\Omega)$ such that
\be
\psi_n=J_n\eta^{\otimes n}\quad(n=1,2,\dots).
\ee
Note that $J_n$ are multilinear. We have
\ba
&
J_1\eta=\frac{1}{u_0}K_1\eta=
\frac{g}{u_0(x)}\int_{\Omega}G(x,y)u_0(y)\eta(y)\,dy,
\\
&
J_2\eta\otimes\eta=
\frac{1}{u_0}K_2\eta\otimes\eta+\frac{1}{2}\left(\frac{1}{u_0}K_1\eta\right)^2
\\
&=
\frac{g^2}{u_0(x)}\int_{\Omega}\int_{\Omega}G(x,y)G(y,z)u_0(z)\eta(y)\eta(z)\,dydz
+\frac{g^2}{2u_0(x)^2}\left(\int_{\Omega}G(x,y)u_0(y)\eta(y)\,dy\right)^2.
\ea
In general, the $n$th term is given by
\be
J_n\eta^{\otimes n}=\sum_{m=1}^n\frac{1}{m}\sum_{i_1+\cdots+i_m=n}
\left(\frac{1}{u_0}K_{i_1}\eta^{\otimes i_1}\right)\cdots\left(\frac{1}{u_0}K_{i_m}\eta^{\otimes i_m}\right).
\ee

We can similarly construct the inverse Rytov series:
\be
\eta=\mathcal{J}_1\psi+\mathcal{J}_2\psi\otimes\psi+
\mathcal{J}_3\psi\otimes\psi\otimes\psi+\cdots.
\ee
We obtain
\be
\mathcal{J}_n=-\left(\sum_{m=1}^{n-1}\mathcal{J}_m\sum_{i_1+\cdots+i_m=n}
J_{i_1}\otimes\cdots\otimes J_{i_m}\right)
\mathcal{J}_1\otimes\cdots\otimes\mathcal{J}_1,
\quad n\ge2.
\ee
Here, $\mathcal{J}_1$ is a regularized pseudoinverse of $J_1$.

An error estimate similar to Theorem \ref{sec9:errorthm} was proved for the inverse Rytov series \cite{Machida23}.

\subsection{Remarks}
\hfill\vskip1mm

For the inverse Born and inverse Rytov series above, $L^{\infty}$-norm was used. Indeed, it is possible to use $L^p$-norm.

Since the cost function which appears in optical tomography has a complicated landscape with local minima, iterative schemes are difficult to apply; the calculation is trapped by a local minimum. The approach with inverse series does not have this issue of being trapped by a local minimum because it relies on perturbation theory.

It should be emphasized that the inverse Born and inverse Rytov series can be applied to different partial differential equations although we used diffuse optical tomography as an example of these methods.

\section{Other approaches and X-ray CT}
\label{sec10}

Instead of giving a comprehensive introduction to the field of inverse problems, rather this note has focused on a few examples. Statistical approaches provide important tools. We refer the reader to \cite{Kaipio-Somersalo04}. As examples of inverse solvers for medical imaging, below, we briefly introduce the Fourier transform for MRI and the Radon transform for X-ray CT. See, for example, \cite{Natterer86,Kaipio-Somersalo04,Epstein08} for more about X-ray CT.

\subsection{MRI}
\hfill\vskip1mm

In the magnetic resonance imaging (MRI), the magnetization of a hydrogen nucleus (i.e., proton) is detected. By applying magnetic fields, we can detect the change of the magnetization as voltage. Let us draw the $x_1$ and $x_2$ axes on a cross section of the human body. Let $f(x)$ be the density of hydrogen nuclei in the slice at $x\in\Rm^2$. Then the signal $g(q)$ ($q\in\Rm^2$) is written as
\be
g(q)=\int_{\Rm^2}f(x)e^{iq\cdot x}\,dx,\quad q\in\Rm^2.
\ee
Gradient magnetic fields are applied during measurements. The gradient is specified by $q_1,q_2\in\Rm$.

As an example, let us suppose that the signal is detected as
\be
g(q)=e^{-\frac{\sigma^2}{2}|q|^2+iq\cdot x_0},\quad q\in\Rm^2,
\ee
where $x_0\in\Rm^2$. Then we obtain
\ba
f(x)
&=
\frac{1}{(2\pi)^2}\int_{\Rm^2}g(q)e^{-iq\cdot x}\,dq
\\
&=
\frac{1}{(2\pi)^2}\int_{\Rm^2}
e^{-\frac{\sigma^2}{2}|q|^2-iq\cdot(x-x_0)}\,dq
\\
&=
\frac{1}{2\pi\sigma^2}e^{-\frac{|x-x_0|^2}{2\sigma_0^2}}.
\ea

\subsection{The Radon transform}
\hfill\vskip1mm

Let us consider the X-ray propagation in the two-dimensional plane $\Rm^2$. Let $\mu=\mu(x)$ be the absorption, where $\mu\ge0$ and $x=(x_1,x_2)\in\Rm^2$. We assume that $\mu$ is piecewise continuous and compactly supported in $\overline{\Omega}$, where $\Omega$ is a domain in $\Rm^2$. For example, $\Omega$ is a cross section of the human body. Let us assume the incident beam which is a unidirectional point beam at $x_0\in\pp\Omega$ in direction $\theta_0\in\Sm^2$. The specific intensity $I$ of X-ray obeys
\be
\left\{\begin{aligned}
&
\uv\cdot\nabla I+\mu I=0,\quad x\in\Omega,
\\
&
I=I_0\delta(x-x_0)\delta(\theta-\theta_0),\quad x\in\pp\Omega,
\end{aligned}\right.
\ee
where $I_0>0$. Here, $\theta\in\Sm^2$ and $\nabla=(\pp/\pp x_1,\pp/\pp x_2)^T$. Let $\omega=(\cos\va,\sin\va)^T$ ($0\le\va<2\pi$) be a unit vector perpendicular to $\theta_0$ and $s\in\Rm$ be the signed distance of the line for the ray corresponding to the incident beam. Let $L=L(\omega,s)$ denote this line:
\be
L(\omega,s)=\{x\in\Rm^2;\;\omega\cdot x=s\}.
\ee
When the specific intensity $I$ is detected on $\pp\Omega$, we obtain
\be
I=I_1\delta(\theta-\theta_0),
\ee
where
\be
I_1=I_0 e^{-\int_{L(\omega,s)}\mu(x)\,dx}.
\ee

By measurements, we obtain the data function\footnote{
The calculation in Sec.~\ref{sec10} is based on the lecture by Vadim A. Markel (University of Pennsylvania, 2008).
}
\be
\Phi(\omega,s)=-\ln\frac{I_1}{I_0}.
\ee
We note that
\be
\Phi(\omega,s)=\int_{L(\omega,s)}\mu(x)\,dx.
\ee
Indeed, this is called the Radon transform of $\mu$ and we can write
\ba
\Phi(\omega,s)
&=
(\mathcal{R}\mu)(\omega,s)
\\
&=
\int_{-\infty}^{\infty}\mu(s\omega+t\omega^{\perp})\,dt
\\
&=
\int_{\Rm^2}\mu(x)\delta(\omega\cdot x-s)\,dx.
\ea
Here, $\omega^{\perp}$ is the unit vector which is perpendicular to $\omega$ and satisfies $\det(\omega\omega^{\perp})>0$. That is, $\omega=(\cos\va,\sin\va)^T$, $\omega^{\perp}=(-\sin\va,\cos\va)^T$ ($0\le\va<2\pi$).

\subsection{The filtered back projection}
\hfill\vskip1mm

The Fourier transform of the data function is given by
\be
(\mathcal{F}\Phi)(\omega,\tau)=
\int_{-\infty}^{\infty}\Phi(\omega,s)e^{-i\tau s}\,ds,\quad\tau\in\Rm.
\ee
Similarly we introduce the Fourier transform of $\mu$ as
\be
(\mathcal{F}\mu)(q)=\int_{\Rm^2}\mu(x)e^{-iq\cdot x}\,dx,\quad q\in\Rm^2.
\ee
We have
\ba
(\mathcal{F}\Phi)(\omega,\tau)
&=
\int_{-\infty}^{\infty}e^{-i\tau s}\int_{\Rm^2}\mu(x)\delta(\omega\cdot x-s)\,dxds
\\
&=
\int_{\Rm^2}\mu(x)e^{-i\tau(\omega\cdot x)}\,dx
\\
&=
(\mathcal{F}\mu)(\tau\omega).
\ea
This relation $(\mathcal{F}\mathcal{R}\mu)(\omega,\tau)=(\mathcal{F}\mu)(\tau\omega)$ is a special case of the Fourier slice theorem.

By the inverse Fourier transform, we obtain
\ba
\mu(x)
&=
\frac{1}{(2\pi)^2}\int_{\Rm^2}(\mathcal{F}\mu)(q)e^{iq\cdot x}\,dq
\\
&=
\frac{1}{(2\pi)^2}\int_0^{2\pi}\int_0^{\infty}(\mathcal{F}\mu)(\tau\omega)
e^{i\tau\omega\cdot x}\tau\,d\tau d\va
\\
&=
\frac{1}{(2\pi)^2}\int_0^{2\pi}\int_0^{\infty}(\mathcal{F}\Phi)(\omega,\tau)
e^{i\tau\omega\cdot x}\tau\,d\tau d\va.
\ea
Noting that $(\mathcal{F}\Phi)(-\omega,\tau)=(\mathcal{F}\Phi)(\omega,-\tau)$, we can continue as
\ba
\mu(x)
&=
\frac{1}{(2\pi)^2}\int_0^{\infty}\left[
\int_0^{\pi}(\mathcal{F}\Phi)(\omega,\tau)e^{i\tau\omega\cdot x}\,d\va+
\int_{\pi}^{2\pi}(\mathcal{F}\Phi)(\omega,\tau)e^{i\tau\omega\cdot x}\,d\va
\right]\tau\,d\tau
\\
&=
\frac{1}{(2\pi)^2}\int_0^{\pi}\int_{-\infty}^{\infty}
(\mathcal{F}\Phi)(\omega,\tau)e^{i\tau\omega\cdot x}|\tau|\,d\tau d\va
\\
&=
\frac{1}{2\pi}\int_0^{\pi}g(\omega,\omega\cdot x)\,d\va,
\ea
where
\be
g(\omega,\omega\cdot x)=
\frac{1}{2\pi}\int_{-\infty}^{\infty}
(\mathcal{F}\Phi)(\omega,\tau)e^{i\tau\omega\cdot x}|\tau|\,d\tau.
\ee
Since the above expression has a filter $|\tau|$ compared to the projection 
$\frac{1}{2\pi}\int_{-\infty}^{\infty}(\mathcal{F}\Phi)(\omega,\tau)e^{i\tau\omega\cdot x}\,d\tau=\Phi(\omega,\omega\cdot x)$, the expression is said to be the filtered projection. Hence,
\begin{equation}
\mu(x)=
\frac{1}{(2\pi)^2}\int_0^{\pi}\int_{-\infty}^{\infty}\int_{-\infty}^{\infty}
\Phi(\omega,s)e^{-i\tau(s-\omega\cdot x)}|\tau|\,dsd\tau d\va.
\label{sec10:mux0}
\end{equation}

\subsection{Reconstruction with the Tikhonov regularization}
\hfill\vskip1mm

Numerical solutions with the inversion formula (\ref{sec10:mux0}) are unstable. We note that $\mathcal{R}$ is not a unitary transformation from $L^2(\Rm^2)$ to $L^2(\Rm\times\Sm^1)$ and $\mathcal{R}^*\neq\mathcal{R}^{-1}$. Similar to Sec.~\ref{sec4}, we will obtain a regularized pseudoinverse of $\mathcal{R}$.

We note that
\be
\int_0^{2\pi}\int_{-\infty}^{\infty}(\mathcal{R}\mu)(\omega,s)h(\omega,s)\,dsd\va
=\int_0^{2\pi}\int_{\Rm^2}\mu(x)h(\omega,\omega\cdot x)\,dxd\va
\ee
for a function $h(\omega,s)$. This calculation implies that the adjoint $\mathcal{R}^*$ is given by
\begin{equation}
(\mathcal{R}^*h)(x)=\int_0^{2\pi}h(\omega,\omega\cdot x)\,d\va.
\label{sec10:mux0sta}
\end{equation}
Thus,
\be
(\mathcal{R}^*\mathcal{R}\mu)(x)=\int_{\Rm^2}\left(
\int_0^{2\pi}\delta(\omega\cdot y-\omega\cdot x)\,d\va
\right)\mu(y)\,dy.
\ee
We obtain for $\phi_q(x)=\frac{1}{2\pi}e^{-iq\cdot x}$ ($q\in\Rm^2$),
\ba
(\mathcal{R}^*\mathcal{R}\phi_q)(x)
&=
\frac{1}{2\pi}\int_{\Rm^2}\left[\int_0^{2\pi}\int_{-\infty}^{\infty}
e^{ik\omega\cdot(x-y)}\,dkd\va\right]\phi_q(y)\,dy
\\
&=
\frac{1}{2\pi}\int_0^{2\pi}\int_{-\infty}^{\infty}
e^{ik\omega\cdot x}(\mathcal{F}\phi_q)(k\omega)\,dkd\va
\\
&=
\frac{1}{2\pi}\int_{\Rm^2}\left[\int_0^{2\pi}\int_{-\infty}^{\infty}
e^{ik|x-y|\cos\va}\,dkd\va\right]\phi_q(y)\,dy
\\
&=
2\int_{\Rm^2}\left[\int_0^{\infty}J_0(k|x-y|)\,dk\right]\phi_q(y)\,dy
\\
&=
2\int_{\Rm^2}\frac{\phi_q(y)}{|x-y|}\,dy,
\ea
where $J_0$ is the Bessel function of order zero. Furthermore,
\ba
\int_{\Rm^2}\frac{e^{-iq\cdot y}}{|x-y|}\,dy
&=
\int_{-\infty}^{\infty}\int_{-\infty}^{\infty}
\frac{e^{-i(q_1y_1+q_2y_2)}}{\sqrt{(x_1-y_1)^2+(x_2-y_2)^2}}\,dy_1dy_2
\\
&=
2e^{-iq_1x_1}\int_{-\infty}^{\infty}e^{-iq_2y_2}\left[\int_0^{\infty}
\frac{\cos{(q_1|y_2-x_2|t)}}{\sqrt{1+t^2}}\,dt\right]\,dy_2
\\
&=
2e^{-iq_1x_1}\int_{-\infty}^{\infty}K_0\left(|x_2-y_2||q_1|\right)e^{-iq_2y_2}\,dy_2
\\
&=
\frac{4}{|q_2|}e^{-iq_1x_1}e^{-iq_2x_2}\int_0^{\infty}K_0\left(\left|\frac{q_1}{q_2}\right|s\right)\cos{s}\,ds
\\
&=
\frac{4}{|q_2|}e^{-iq_1x_1}e^{-iq_2x_2}\frac{\pi}{2\sqrt{1+\left(\frac{q_1}{q_2}\right)^2}}
\\
&=
\frac{2\pi}{|q|}e^{-iq\cdot x},
\ea
where $K_0$ is the modified Bessel function of the second kind of order zero. We obtain
\be
(\mathcal{R}^*\mathcal{R}\phi_q)(x)
=\lambda_q\phi_q(x),\quad\lambda_q=\frac{4\pi}{|q|},\quad q\in\Rm^2.
\ee

Let us define
\be
\psi_q=\frac{1}{\sigma_q}\mathcal{R}\phi_q,\quad
\sigma_q=\sqrt{\lambda_q}.
\ee
We note that
\ba
\psi_q(\omega,s)
&=
\frac{1}{\sigma_q}\int_{-\infty}^{\infty}\phi_q(s\omega+t\omega^{\perp})\,dt=
\frac{1}{2\pi\sigma_q}\int_{-\infty}^{\infty}e^{-iq\cdot(s\omega+t\omega^{\perp})}\,dt
\\
&=
\frac{1}{\sigma_q}e^{-isq\cdot\omega}\delta\left(q\cdot\omega^{\perp}\right).
\ea
We have
\be
\mathcal{R}^*\psi_q=\sigma_q\phi_q,\quad
\mathcal{R}\phi_q=\sigma_q\psi_q.
\ee

With the relation $\mathcal{R}^*\mathcal{R}\mu=\mathcal{R}^*\Phi$, we obtain
$\mu=(\mathcal{R}^*\mathcal{R})^{-1}\mathcal{R}^*\Phi$. We begin with the expression
\be
\mu(x)=
\int_{\Rm^2}\left(\int_{\Rm^2}\mu(y)\overline{\phi_q(y)}\,dy\right)\phi_q(x)\,dq,
\ee
where $\bar{\phantom{\phi}}$ denotes complex conjugate. Using the relation
\ba
\int_0^{2\pi}\int_{-\infty}^{\infty}\Phi(\omega,s)\overline{\psi_q(\omega,s)}\,dsd\va
&=
\int_0^{2\pi}\int_{-\infty}^{\infty}(\mathcal{R}\mu)(\omega,s)\frac{1}{\sigma_q}\overline{(\mathcal{R}\phi_q)(\omega,s)}\,dsd\va
\\
&=
\int_{\Rm^2}\mu(x)\frac{1}{\sigma_q}\overline{(\mathcal{R}^*\mathcal{R}\phi_q)(x)}\,dx
\\
&=
\sigma_q\int_{\Rm^2}\mu(x)\overline{\phi_q(x)}\,dx,
\ea
we obtain
\be
\mu(x)=
\int_{\Rm^2}\frac{1}{\sigma_q}\left(
\int_0^{2\pi}\int_{-\infty}^{\infty}\Phi(\omega,s)\overline{\psi_q(\omega,s)}\,dsd\va\right)\phi_q(x)\,dq.
\ee

Instead of $\mathcal{R}^{-1}$, we consider its approximation $R_{\alpha}$ as follows by inserting the Tikhonov filter.
\be
\mu_{\alpha}(x)=(R_{\alpha}\Phi)(x)=
\int_{\Rm^2}\frac{\lambda_q}{\alpha+\lambda_q}\frac{1}{\sigma_q}\left(
\int_0^{2\pi}\int_{-\infty}^{\infty}\Phi(\omega,s)\overline{\psi_q(\omega,s)}\,dsd\va\right)\phi_q(x)\,dq.
\ee
We have
\ba
\mu_{\alpha}(x)
&=
\frac{1}{(2\pi)^2}\int_{\Rm^2}\frac{1}{\alpha+\lambda_q}\left(
\int_0^{2\pi}\int_{-\infty}^{\infty}\int_{-\infty}^{\infty}\Phi(\omega,s)e^{iq\cdot(s\omega+t\omega^{\perp})}\,dsdtd\va\right)e^{-iq\cdot x}\,dq
\\
&=
\frac{1}{2\pi}\int_{\Rm^2}\frac{1}{\alpha+\lambda_q}\left(
\int_0^{2\pi}\delta(q\cdot\omega^{\perp})(\mathcal{F}\Phi)(\omega,-q\cdot\omega)\,d\va\right)e^{-iq\cdot x}\,dq.
\ea
Let us express $q\in\Rm^2$ as $q=-\tau\omega+q_{\perp}\omega^{\perp}$ ($\tau,q_{\perp}\in\Rm$). Then we have
\ba
\mu_{\alpha}(x)
&=
\frac{1}{2\pi}\int_{-\infty}^{\infty}\int_{-\infty}^{\infty}\frac{1}{\alpha+\lambda_q}\left(
\int_0^{2\pi}\delta(q_{\perp})(\mathcal{F}\Phi)(\omega,\tau)\,d\va\right)e^{-iq\cdot x}\,d\tau dq_{\perp}
\\
&=
\frac{1}{8\pi^2}\int_0^{2\pi}\int_{-\infty}^{\infty}\frac{1}{1+\alpha|\tau|/(4\pi)}
(\mathcal{F}\Phi)(\omega,\tau) e^{i\tau \omega\cdot x}|\tau|\,d\tau d\va
\\
&=
\frac{1}{2\pi}\int_0^{\pi}g_{\alpha}(\omega,\omega\cdot x)\,d\va,
\ea
where
\be
g_{\alpha}(\omega,\omega\cdot x)=
\frac{1}{2\pi}\int_{-\infty}^{\infty}\frac{1}{1+\alpha|\tau|/4}
(\mathcal{F}\Phi)(\omega,\tau)e^{i\tau\omega\cdot x}|\tau|\,d\tau.
\ee

The reconstructed $\mu$ is written as
\begin{equation}\begin{aligned}
\mu_{\alpha}(x)
&=
\frac{1}{(2\pi)^2}\int_0^{\pi}\int_{-\infty}^{\infty}\int_{-\infty}^{\infty}
\frac{1}{1+\alpha|\tau|/4}
\Phi(\omega,s)e^{i\tau(\omega\cdot x-s)}|\tau|\,dsd\tau d\va
\\
&=
\frac{1}{2\pi}\int_0^{\pi}\int_{-\infty}^{\infty}
\Phi(\omega,s)\mathcal{I}(\omega\cdot x-s)\,dsd\va.
\end{aligned}
\label{sec10:mualpha}
\end{equation}
Here we introduced $\mathcal{I}(\xi)$, $\xi\in\Rm$, as
\be
\mathcal{I}(\xi)
=\frac{1}{2\pi}\int_{-\infty}^{\infty}
\frac{|\tau|\tau_{\rm max}}{\tau_{\rm max}+|\tau|}e^{i\tau\xi}\,d\tau
=\frac{\tau_{\rm max}}{\pi}\int_0^{\infty}
\frac{\tau}{\tau+\tau_{\rm max}}\cos{(\tau\xi)}\,d\tau,
\ee
where $\tau_{\rm max}=4/\alpha>0$. The above expression implies $\mathcal{I}(-\xi)=\mathcal{I}(\xi)$.

We note that
\ba
\mathcal{I}(\xi)
&=
\frac{\tau_{\rm max}}{\pi}\frac{d}{d\xi}\int_0^{\infty}
\frac{\sin{(\tau\xi)}}{\tau+\tau_{\rm max}}\,d\tau
\\
&=
\frac{\tau_{\rm max}}{\pi}\frac{d}{d\xi}\left[
\mathop{\mathrm{ci}}(\tau_{\rm max}\xi)\sin{(\tau_{\rm max}\xi)}-
\mathop{\mathrm{si}}(\tau_{\rm max}\xi)\cos{(\tau_{\rm max}\xi)}\right],
\ea
where the cosine integral function and sine integral function are given by
\be
\mathop{\mathrm{ci}}(\xi)=-\int_{\xi}^{\infty}\frac{\cos{t}}{t}\,dt,\quad
\mathop{\mathrm{si}}(\xi)=-\int_{\xi}^{\infty}\frac{\sin{t}}{t}\,dt.
\ee
The function $\mathcal{I}(\xi)$ is obtained as
\begin{equation}
\mathcal{I}(\xi)=
\frac{\tau_{\rm max}^2}{\pi}\left[
\mathop{\mathrm{ci}}(\tau_{\rm max}\xi)\cos{(\tau_{\rm max}\xi)}+
\mathop{\mathrm{si}}(\tau_{\rm max}\xi)\sin{(\tau_{\rm max}\xi)}\right].
\label{sec10:funcI}
\end{equation}

\subsection{Numerical experiments}
\hfill\vskip1mm

Let us perform numerical experiments for the target shown in Fig.~\ref{sec10:fig01}. Let $\Omega$ be an $L\times L$ square. The distances from the origin to the outer and inner edges are $L/2$ and $L/4$, respectively. We set $\mu=1/L$ inside the medium and $\mu=0$ otherwise. Hereafter, we take $L$ to be the unit of length.

\begin{figure}[ht]
\begin{center}
\includegraphics[width=0.3\textwidth]{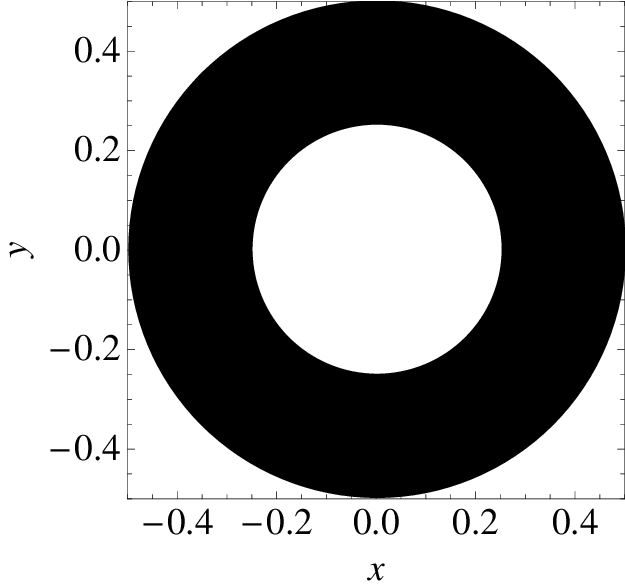}
\end{center}
\caption{
Target.
}
\label{sec10:fig01}
\end{figure}

For the ideal measurement in the absence of noise, the data function $\Phi(\omega,s)$ is given by
\be
\left\{\begin{aligned}
\Phi=0,&\quad 0.5<|s|,\\
\Phi=\sqrt{1-4s^2},&\quad 0.25<|s|\le 0.5,\\
\Phi=\sqrt{1-4s^2}-\sqrt{0.25-4s^2},&\quad |s|\le 0.25.
\end{aligned}\right.
\ee
The function $\Phi(\omega,s)$ is plotted in Fig.~\ref{sec10:fig02}. We see that
\be
\Phi_{\rm max}=\max_{\omega,s}{\Phi(\omega,s)}=
\Phi(\omega,0.25)=\frac{\sqrt{3}}{2}.
\ee

\begin{figure}[ht]
\begin{center}
\includegraphics[width=0.5\textwidth]{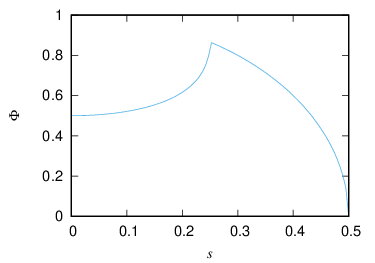}
\end{center}
\caption{
The data function $\Phi(\omega,s)$ as a function of $s$.
}
\label{sec10:fig02}
\end{figure}

For numerical calculation, we discretize $\va,s$ as follows:
\be
\Phi_{ij}=\Phi(\omega_i,s_j),\quad
\va_i=i\frac{\pi}{N_{\va}},\quad
s_j=j\frac{s_{\rm max}}{N_s},\quad
i=0,\dots,N_{\va},\quad j=0,\pm1,\dots,\pm N_s,
\ee
where $s_{\rm max}=\sqrt{2}(L/2)=1/\sqrt{2}$, and we set
\be
N_{\va}=256,\quad N_s=128.
\ee

The numerical reconstruction is done using (\ref{sec10:mualpha}) and (\ref{sec10:funcI}). We discretize $\mu$ as
\be
\mu_{ij}=\mu(x_{1k},x_{2l}),\quad
x_{1k}=\frac{k}{2N_L},\quad x_{2l}=\frac{l}{2N_L},\quad
k,l=0,\pm 1,\dots,\pm N_L,
\ee
where we set
\be
N_L=128.
\ee

Reconstructed $\mu_{\alpha}$ is plotted in Fig.~\ref{sec10:fig03} for $\tau_{\rm max}=$ (left) $10$ and (right) $100$. Hereafter we use $\tau_{\rm max}=100$, i.e., $\alpha=4/\tau_{\rm max}=0.04$.

\begin{figure}[ht]
\begin{center}
\includegraphics[width=0.4\textwidth]{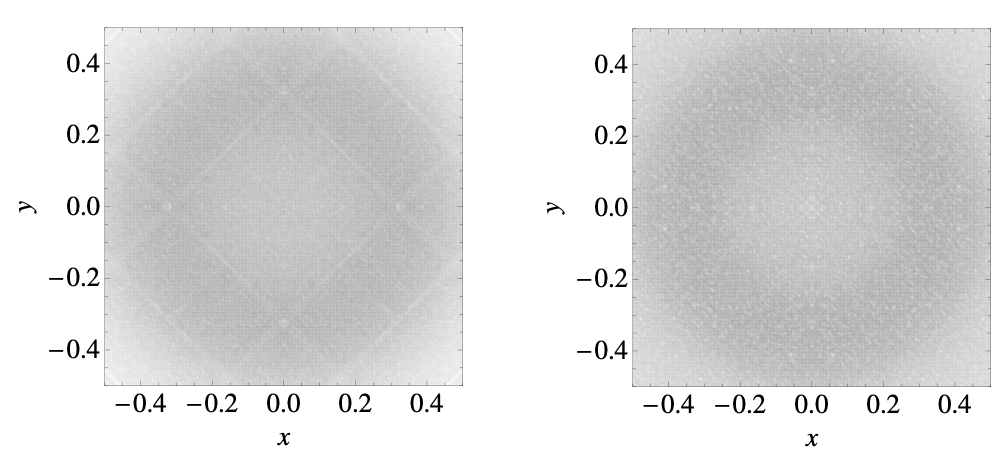}
\end{center}
\caption{
Density plots of $\mu_{\alpha}$ for $\tau_{\rm max}=10$ (left) and $100$ (right).
}
\label{sec10:fig03}
\end{figure}

Next we added Gaussian noise as
\be
\Phi_{\rm noise}(\omega,s)=\Phi(\omega,s)+\frac{\Phi_{\rm max}}{100}X,\quad
X\sim\frac{1}{\sqrt{2\pi}\sigma}e^{-X^2/(2\sigma^2)},\quad
\sigma=0.2\;\mbox{or}\;1.
\ee
We set $\tau_{\rm max}=100$. Results are shown in Fig.~\ref{sec10:fig04}.

\begin{figure}[ht]
\begin{center}
\includegraphics[width=0.4\textwidth]{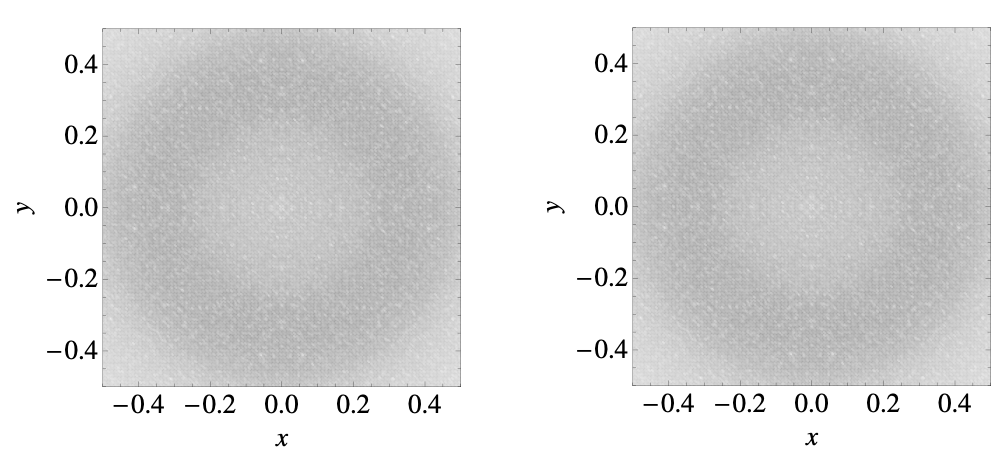}
\caption{
Density plots of $\mu_a$ in the presence of noise.
The left panel shows $\mu_a$ for the Gaussian noise with $\sigma=0.2$ and the right panel shows $\mu_a$ for the Gaussian noise with $\sigma=1.0$. In both panels, $\tau_{\rm max}=100$.
}
\label{sec10:fig04}
\end{center}
\end{figure}


\setcounter{section}{1}
\appendix

\section{Forward data}
\label{fwd}

Let $G_a$ be the Green's function of the two-dimensional radial problem for the equation in which $\eta=\eta_a$ in $r\in[0,R_a]$ and $\eta=0$ otherwise. In the case of the delta-function source $\delta(x-x_s)$, $x_s\in\pp\Omega$, we have \cite{Moskow-Schotland09}
\be
G_a(x,x_s)=\frac{1}{2\pi}\sum_{n=-\infty}^{\infty}a_ne^{in(\theta-\theta_s)}I_n(k_ar),\quad
r\in\Omega_1,
\ee
and
\ba
G_a(x,x_s)
&=
\frac{1}{2\pi}\sum_{n=-\infty}^{\infty}e^{in(\theta-\theta_s)}I_n(kr)K_n(kR)
\\
&+
\frac{1}{2\pi}\sum_{n=-\infty}^{\infty}e^{in(\theta-\theta_s)}\left(b_nK_n(kr)+c_nI_n(kr)\right),\quad r\in\Omega_2.
\ea
Here, coefficients $a_n,b_n,c_n$ can be computed as the solution to the following system of linear equations, which is derived from the interface and boundary conditions.
\ba
&
\left(\begin{array}{ccc}
I_n(k_aR_a) & -K_n(kR_a) & -I_n(kR_a) \\
k_aI_n'(k_aR_a) & -kK_n'(kR_a) & -kI_n'(kR_a) \\
0 & K_n(kR)+k\ell K_n'(kR) & I_n(kR)+k\ell I_n'(kR)
\end{array}\right)
\left(\begin{array}{c}
a_n \\ b_n \\ c_n
\end{array}\right)
\\
&=
\left(\begin{array}{c}
I_n(kR_a)K_n(kR) \\
kI_n'(kR_a)K_n(kR) \\
k\ell I_n(kR)K_n'(kR)+I_n(kR)K_n(kR)
\end{array}\right).
\ea
Therefore we obtain for $x\in\pp\Omega$,
\ba
u(x)
&=\int_0^{2\pi}\int_0^RG_a(x,y)S(r_y,\theta_y)r_y\,dr_yd\theta_y
\\
&=
e^{ij\theta_x}\left(I_j(kR)K_j(kR)+b_jK_j(kR)+c_jI_j(kR)\right)
\ea
for $j=1,2,\dots,M_S$.



\begin{thebibliography}{99}

\bibitem{Kac66}
M. Kac,
Can one hear the shape of a drum?,
{\it The American Mathematical Monthly} {\bf 73} 1--23 (1966)

\bibitem{Kaipio-Somersalo04}
J. P. Kaipio and E. Somersalo, 
{\it Statistical and Computational Inverse Problems} (Springer New York, 2004)

\bibitem{Colton-Kress98}
D. Colton and R. Kress, 
{\it Inverse Acoustic and Electromagnetic Scattering Theory}
2nd ed (Berlin: Springer, 1998)

\bibitem{Kamimura14}
Y. Kamimura,
{\it Gyakumondai no Kangaekata (The concept of inverse problems)} 
(Kodansha, 2014) [Japanese]

\bibitem{Carleman39}
T. Carleman, 
Sur un probl\`{e}me d'unicit\'{e} pour les syst\`{e}mes d'\'{e}quations aux d\'{e}riv\'{e}es partielles \`{a} deux variables ind\'{e}pendantes,
{\it Arkiv f\"{o}r Matematik, Astronomi och Fysik} {\bf 26B}, 1--9 (1939)

\bibitem{Hoermander63}
L. H\"{o}rmander,
{\it Linear Partial Differential Operators}
(Springer, 1963)

\bibitem{Bukhgeim-Klibanov81}
A. L. Bukhgeim and M. V. Klibanov,
Global uniqueness of class of multidimensional inverse problems,
{\it Soviet Math. Dokl.} {\bf 24} 244--247 (1981) 

\bibitem{Imanuvilov-Yamamoto01}
O. Y. Imanuvilov and M. Yamamoto,
Global uniqueness and stability in determining coefficients of wave equations,
{\it Communications in Partial Differential Equations} {\bf 26}, 1409--1425 (2001)

\bibitem{Isakov06}
V. Isakov,
{\it Inverse Problems for Partial Differential Equations}
(Springer, 2006)

\bibitem{Yamamoto09}
M. Yamamoto,
Carleman estimates for parabolic equations and applications,
{\it Inverse Problems} {\bf 25}, 123013 (2009)

\bibitem{Machida-Yamamoto14}
M. Machida and M. Yamamoto,
Global Lipschitz stability in determining coefficients of the radiative transport equation,
{\it Inverse Problems} {\bf 30}, 035010 (2014)

\bibitem{Yosida78}
K. Yosida,
{\it Bibunhouteishiki no Kaihou} (Solution methods for differential equations)
(Iwanami, 1978) [Japanese]

\bibitem{Yamamoto-Kim08}
M. Yamamoto and S. Kim,
{\it The Inverse Problem in the Heat Equation}
(Science, 2008) [Japanese]

\bibitem{Engl-Hanke-Neubauer00}
H. W. Engl, M. Hanke, A. Neubauer,
{\it Regularization of Inverse Problems}
(Kluwer Academic Publishers, 1996)

\bibitem{McDonald09}
G. C. McDonald,
Ridge regression,
{\it Wiley Interdiscip.~Rev.~Comput.~Stat.} {\bf 1}, 93--100 (2009)

\bibitem{Neculai20}
A. Neculai,
{\it Nonlinear Conjugate Gradient Methods for Unconstrained Optimization}
(Springer, 2020)

\bibitem{Polak-Ribiere69}
E. Polak and G. Ribiere, 
Note sur la convergence de m\'{e}thodes de directions conjugu\'{e}es,
{\it ESAIM: Mathematical Modelling and Numerical Analysis - Mod\'{e}lisation Math\'{e}matique et Analyse Num\'{e}rique} {\bf 3}, 35--43 (1969)

\bibitem{Polyak69}
B. T. Polyak,
The conjugate gradient method in extreme problems,
{\it USSR Computational Mathematics and Mathematical Physics} {\bf 9} 94--112 (1969).

\bibitem{Fletcher-Reeves64}
R. Fletcher and C. M. Reeves,
Function minimization by conjugate gradients,
{\it The Computer Journal} {\bf 7}, 149--154 (1964)

\bibitem{Hestensen-Stiefel52}
M. R. Hestenes and E. Stiefel,
Methods of conjugate gradients for solving linear systems,
{\it Journal of Research of the National Bureau of Standards} {\bf 49} 409--436 (1952)

\bibitem{Dwight57}
H. B. Dwight,
{\it Tables of Integrals and Other Mathematical Data}
(Macmillan, 1957) Sec.~50

\bibitem{Markel-OSullivan-Schotland03}
V. A. Markel, J. A. O'Sullivan, and J. C. Schotland,
Inverse problem in optical diffusion tomography.~IV. Nonlinear inversion formulas,
{\it Journal of the Optical Society of America A} {\bf 20} 903--912 (2003)

\bibitem{Moskow-Schotland08}
S. Moskow and J. C. Schotland,
Convergence and stability of the inverse scattering series for diffuse waves,
{\it Inverse Problems} {\bf 24} 065005 (2008)

\bibitem{Machida-Schotland15}
M. Machida and J. C. Schotland,
Inverse Born series for the radiative transport equation,
{\it Inverse Problems} {\bf 31} 095009 (2015)

\bibitem{Moskow-Schotland09}
S. Moskow and J. C. Schotland,
Numerical studies of the inverse Born series for diffuse waves,
{\it Inverse Problems} {\bf 25} 095007 (2009)

\bibitem{Machida23}
M. Machida,
The inverse Rytov series for diffuse optical tomography,
{\it Inverse Problems} {\bf 39} 105012 (2023)

\bibitem{Natterer86}
F. Natterer,
{\it The Mathematics of Computerized Tomography}
(Wiley, 1986)

\bibitem{Epstein08}
C. L. Epstein,
{\it Introduction to the Mathematics of Medical Imaging} 2nd ed.\
(SIAM, 2008)

\end{thebibliography}
\end{document}